\documentclass[11pt]{amsart}
\usepackage{color}
\usepackage[parfill]{parskip}    % Activate to begin paragraphs with an empty line rather than an indent
\usepackage{graphicx}
\usepackage{amsfonts}
\usepackage{amsmath}
\usepackage{amsthm}
\usepackage{amssymb}

\newtheorem{thm}{Theorem}[section]
\newtheorem{prop}[thm]{Proposition}
\newtheorem*{definition*}         {Definition}
\newtheorem{lemma}[thm]{Lemma}

\newtheorem{cor}[thm]{Corollary}

\DeclareMathOperator{\tr}{tr}

\newcommand*{\disc}{\text{disc}}
\newcommand*{\Oo}{\mathcal{O}}
\newcommand*{\Ooi}{\mathcal{O}_{\infty}}
\newcommand*{\F}{\mathbb{F}}
\newcommand*{\Tt}{\widetilde{\mathbb{H}}}
\newcommand{\setdef}[2]{\ensuremath{\left\{\left.{#1}\;\right|\;{#2}\right\}}}
\begin{document}

\title[Metaplectic Ramanujan conjecture over function fields]{Metaplectic Ramanujan conjecture over function fields with applications to quadratic forms}

\author{S. Al\.i Altu\u{g} and Jacob Tsimerman}

%\abbrevauthor{A. Altu\u{g}, and J. Tsimerman}
%\headabbrevauthor{Ali Altu\u{g}, and Jacob Tsimerman}
%\email{saltug@math.princeton.edu}
%\address{
%\affilnum{1,2}Mathematics department, Princeton University, Washington Road, Princeton, NJ 08544, USA}
%\author{ Jacob Tsimerman}
%\email{jtsimerm@math.princeton.edu}
%\address{Mathematics department\\ Princeton University\\Washington Road\\ Princeton, NJ \\ 08544-1000 USA}
%\correspdetails{\affilnum{1}saltug@math.princeton.edu}
%\correspdetails{\affilnum{2}jtsimerm@math.princeton.edu}
%\classification{11M38}

% Enter received/revised/accepted dates as necessary
%\received{7 April 2011}
%\revised{9 November 2012}

%\accepted{}
%\communicated{Editor}
\begin{abstract}
We formulate and prove the analogue of the Ramanujan Conjectures for modular forms of half-integral weight subject to some ramification restriction
in the setting of a polynomial ring over a finite field. This is applied to give an effective solution to the problem of representations of elements of the ring by ternary quadratic forms. Our proof develops the theory of half-integral weight forms and Siegel's theta functions in this context as well as the analogue of an explicit Waldspurger formula. As in the case over the rationals, the half-integral weight Ramanujan Conjecture is in this way converted into a question of estimating special values of members of a special family of L-functions. These polynomial functions have a growing number of roots (all on the unit circle thanks to Drinfeld and Deligne's work) which are shown to become equidistributed. This eventually leads to the key estimate for Fourier coefficients of half-integral weight cusp forms.
\end{abstract}

\maketitle

\section{Introduction}

\subsection{Representations of numbers by quadratic forms: History} % We need references here

The general question that motivates this paper is the following one:

\emph{Given a quadratic form Q(X) over some ring $R$, which elements of $R$ are represented by Q?}

This question has historically received a lot of attention, especially over number fields and their maximal orders. Over
$\mathbb{Q}$ and other number fields $K$, the question turns out to be a completely local one. Specifically, we have the following

\begin{thm}[Hasse-Minkowski \cite{O}]
 Let $Q(X)$ be a quadratic form over $K$, where $K$ is a number field. Then $Q$ represents an element $\alpha$ over $K$ iff $Q$ represents $\alpha$
 over the completion $K_v$ for every valuation $v$ of $K$.
\end{thm}

Over orders in number fields, however, the question becomes more delicate, as the Theorem of Hasse-Minkowski is no longer true. To
remedy this situation, one must introduce the notion of \emph{genus}. From now on we restrict ourselves to $\mathbb{Z}$.

\begin{definition*} Two quadratic forms $Q_1,Q_2$ over $\mathbb{Z}$ belong to the same genus if they are isomorphic over $\mathbb{Z}_p$
 for all prime numbers $p$ and over $\mathbb{R}$.
\end{definition*}

With this definition, we have the following analogue of the Hasse-Minkowski Theorem:

\begin{thm}[Minkowski-Siegel \cite{K}]
Let $Q_1,\dots,Q_g$ constitute a genus $G$ of quadratic forms over $\mathbb{Z}$. Then an integer $n$ is representable by some form
in $G$ iff it is representable by some form (and hence all forms) in $G$ over $\mathbb{Z}_p$ for each prime number $p$, as well as
over $\mathbb{R}$.
\end{thm}

Even though the Theorem of Hasse-Minkowski can fail for individual quadratic forms in $3$ or more variables, it turns out that the set of exceptions is small. In particular the set of square-free exceptions is finite. The techniques leading to such results are analytic ones. By developing the circle method, Hardy and Littlewood managed to prove such a theorem for non-degenerate quadratic forms in at least $5$ variables. Kloosterman refined this method, and extended it to non-degenerate diagonal quadratic forms in $4$ variables.
However, the extension to quadratic forms in $3$ variables had to wait for the theory of automorphic forms and L-functions to be developed, and was finally proved by Duke and Schulze Pillot in 1988. More precisely, the theorem is the following:

\begin{thm}[Hardy-Littlewood $n\geq 5$, Kloosterman $n=4$, Duke-S.P. $n=3$ \cite{D-SP}]
Let $Q(x)$ be a non-degenerate integral quadratic form in $n\geq 3$ variables. Then if $m$ is a large enough square-free number, $Q$ represents $m$ over $\mathbb{Z}$ iff it represents $m$ integrally over all completions of $\mathbb{Z}$.
\end{thm}

It is worth noting that all the analytic proofs of the Theorem show that the number of times a number is represented by a quadratic form in fact is asymptotically ``what it should be''. It is also worth noting that the Theorem is ineffective due to a possible Siegel zero.

We now focus on the case of 3 variables and discuss the theory of an ternary integral quadratic form $Q(\vec{x})$. The story for positive definite and indefinite forms is
slightly different, so we assume for now that $Q(\vec{x})$ is positive definite. The theta function attached to $Q(\vec{x})$ is then defined by
\begin{equation}\label{thetaQ}
 \Theta_Q(z) = \displaystyle\sum_{\vec{x}\in\mathbb{Z}^3}e(Q(\vec{x})z)
 \end{equation}
$\Theta_Q(z)$ is absolutely convergent and defines a holomorphic function in the upper half-plane. Moreover, if we let $r_Q(m)$ be
the number of solutions to $Q(\vec{x})=m $, then we have the equality:
$$\Theta_Q(z) = 1+\displaystyle\sum_{m>0} r_Q(m)e(mz)$$
We can therefore read off the number of representations of an integer $m$ by $Q$ from the
Fourier coefficients of $\Theta_Q$! Moreover, using Poisson summation, we can see that                      %Expand here!!!!!!
$\Theta_Q(z)$ is a modular form of weight $ {3/2}$ for a congruence group $\Gamma_0(N)$ of $SL_2(\mathbb{Z})$ for $N=4\disc(Q)$. These observations are the foundation of this approach to studying $r_Q(m)$. %What we really want is to be dealing with
%\emph{cusp forms}, modular forms that vanish at all the cusps of $\Gamma_0(N)\backslash SL_2(\mathbb{R})$.
Although the theta function $\Theta_Q$ contains all the information about the number of representations, it is hard to analyze directly. The way to proceed is to define another theta function for the genus
$G$ of $Q$. For each form $Q'\in G,$ define $n_{Q'}$ to be the order of the group of automorphisms of $Q'$ over $\mathbb{Z}$, so that $n(Q')=|SO_{Q'}(\mathbb{Z})|.$ We note that since $Q'$ is definite, $n_{Q'}$ is finite. Define $n_G=\sum_{Q'\in G}\frac{1}{n_{Q'}}$ and
\begin{equation}\label{thetaG}
\Theta_G(z)=\frac{1}{n_G}\sum_{Q'\in G} \frac{\Theta_{Q'}(z)}{n_{Q'}}
\end{equation}
First, notice that $\Theta_G$ is also a modular form of weight {3/2} for $\Gamma_0(N)$. More importantly, the coefficients $r_G(m)$
of $\Theta_G$ correspond to the number of representations of numbers by the whole genus $G$, and are therefore very well understood
by Siegel's mass formula to be products of local densities. The idea is to show that the coefficients of $\Theta_Q$ and
$\Theta_G$ aren't very different. The key identity is the following:

\begin{lemma}[Siegel]
$\Theta_G(z)-\Theta_Q(z)$ is a cusp form.
\end{lemma}

We give a proof of this in the function field case when $Q$ is anisotropic in section \ref{sectiongauss} (c.f. \S \ref{SiegelsTheorem}).

All that is left to prove the Theorem is to show that cusp forms have relatively small Fourier coefficients compared to the mass in Siegel's formula. It turns out there
are 2 types of cusp forms of weight ${3/2}$. The first kind are theta series associated to one dimensional lattices; these are easy
to deal with in our case since their coefficients are supported on finitely many square-classes, and so contain only finitely many
non-zero square-free coefficients. The second kind are the cusp forms orthogonal to these, and these are all obtained by the Shimura
correspondence from weight 2 cusp forms on the upper half plane. The main task is to bound the Fourier coefficients of
these.

To get a sense for the kind of bound we need, observe that by Siegel's mass formula and his lower bound\footnote{Note that this bound is ineffective over $\mathbb{Q}$.} on quadratic $L$-functions\footnote{See Theorem \ref{largegenusrepresentations}.}  either there are no representations of a squarefree
integer $m$ by a form in the genus $G$, or else $m^{{1/2}-\epsilon}\ll r_G(m)\ll m^{{1/2}+\epsilon}$. The kind of bound we are
looking for on the Fourier coefficients of cusp forms is $a(m)\ll m^{{1/2}-\delta}$ for some $\delta >0$. As it happens, the bound
$a(m)\ll m^{{1/2}+\epsilon}$ is easy to prove and any kind of improvement would be sufficient for
our purposes. There are two possible ways to proceed in this case:

The first (and this is the method developed by Iwaniec and Duke) is to develop a Kuznetzov formula, and then proceed to bound
sums of Sali\'e sums. This approach is arguably quicker and more direct, since it never goes through Waldspurger's formula.

The second approach is to use Waldspurger's formula. This is a formula which relates the Fourier coefficients of a half-integral weight
modular form to the central value of a certain L-function associated with its Shimura lift. In this way, the problem becomes a
subconvexity problem. There is by now a great deal of machinery to deal with subconvexity problems, and so one can get better
bounds this way. This also has the advantage that the optimal bound for the central value of an L-function follows from the Riemann
hypothesis for that L-function combined with the knowledge of the distribution of its zeroes, and in the function field case the Riemann hypothesis is known. This is mainly why we follow this
approach. We would also like to emphasize that the metaplectic Ramanujan conjecture is not a spectral statement, in that it is not a consequence of the underlying local representations being tempered. As such, although the Riemann hypothesis is essential for bounding the central value of the L-function, the statement of purity alone is insufficient for our purposes.

We now focus on the state of affairs for function fields. Once and for all, we fix a function field $k=\mathbb{F}_p(T)$, where $p$ is a prime such that $p\equiv 1\pmod{4}$ for convenience. In this setting, the story for rational forms, that is quadratic forms defined over $k$, is the same as for number fields
in that the Hasse-Minkowski Theorem still holds.

\begin{thm}[Hasse-Minkowski \cite{O}]
 Let $Q(X)$ be a quadratic form over $k$. Then $Q$ represents an element $\alpha$ over $k$ iff $Q$ represents $\alpha$
 over the completion $k_v$ for every valuation $v$ of $k$.
\end{thm}

The completions of $k$ have a slightly different nature than that of $\mathbb{Q}$ in that there are no archimedean places. Indeed, there is one completion  $k_P$
for each monic irreducible polynomial $P,$ as well as the completion ``at infinity", $k_{\infty}=\mathbb{F}_p((T^{-1})).$ It should be emphasized that while we
single out $k_{\infty}$ as a fixed place over which to work, it is really no different than any other completion of $k,$ and everything that follows would
work equally well at any completion. The analogue of an integral quadratic form in the function field case is a quadratic form $Q$ over an $\mathbb{F}_p[T]$-lattice $L$ and the question of interest is which polynomials $Q$ represents. It is important here to distinguish 2 cases. We define $Q$ to be \emph{isotropic} over a field if it represents 0
over that field non-trivially, and \emph{anisotropic} otherwise. If $Q$ is isotropic than it often represents a polynomial infinitely many times, so it is natural to
count in ``boxes", that is, we fix a coordinate system $x_1,x_2,\dots,x_n$ for $L$ and bound the degrees of each coordinate. A  coordinate-free way of doing this is to fix an $\mathbb{F}_p[[T^{-1}]]$-lattice $\mathcal{O}_V$ in $V=L\otimes_{\mathbb{F}_p[T]}k_{\infty}$ and define the restricted
representation numbers $$r(D;L,\mathcal{O}_V,m) = \#(v\in L\cap T^m\mathcal{O}_V\mid Q(v)=D).$$ Note that these numbers are always finite since $L\cap\mathcal{O}_V$ is finite.
In the anisotropic case, however, the representation numbers $r(L,D)=\#(v\in L\mid Q(v)=D)$ are always finite, and no restriction is necessary. The reason is
that in this case the orthogonal group $SO_Q(k_{\infty})$ is compact, while the lattice $L$ is discrete in $V.$ Another immediate consequence of this compactness
in the anisotropic case is the finiteness of the integral stabilizer group, $SO_Q(\F_p[T]).$

Merrill-Walling \cite{MW} and later Hoffstein-Merrill-Walling \cite{HMW} studied the case of representations of polynomials by sums of $n$-squares, which in our
notation corresponds to $L=\mathbb{F}_q[T]^n,\: Q(x_1,\dots,x_n)=x_1^2+\cdots+x_n^2,\:\mathcal{O}_V=\left\{(x_1,\dots,x_n)\in k_{\infty}^n\mid \deg(x_i)\leq 0\right\}$
and $D$ and $m$ vary. They get an asymptotic for $r(D,m)$, the number of representatios of $D$ as a sum of $n$ squares of degree $\leq m$ as long as $n\geq 3,$ and show that these numbers are non-zero for sufficiently large $D$
subject to the necessary condition $m\geq \deg(D)/2.$ Their method develops the spectral theory on the space
$SL_2(\F_q[T])\backslash SL_2(k_{\infty})/SL_2(\F_q[[T^{-1}]])$
and then expresses an appropriate theta function in terms of the Eisenstein spectrum. They note that this is only possible because on this space there are no cusp forms, a phenomenon which does not persist when one adds level and in particular when one considers more general forms.

In \cite{Car}, Car developed the circle method and the Kloosterman refinement of the circle method for diagonal forms
$Q(x_1,x_2,\dots,x_n)= A_1x_1^2+\cdots+A_n x_n^2$ in at least 4 variables where for every $j,$ $A_j\in\mathbb{F}_p[T]$. As a result, she proves an asymptotic for the number of representations
of $D$ even under the strictest degree conditions, namely $$\deg(x_i)\leq \frac12(\deg(D)-\deg(A_i) +1)$$ It turns out that for anisotropic diagonal forms
every representation of a polynomial satisfies the strictest degree conditions, so an immediate consequence of Car's work is an asymptotic for representation
numbers of diagonal anisotropic forms in $\geq 4$ variables.

We now describe our results, focusing primarily on the anisotropic case. Let $Q$ be a quadratic form on a 3-dimensional $\mathbb{F}_q[T]$-lattice $L$, such that $Q$ is
anisotropic over $k_{\infty},$ that is, $Q$ does not represent $0$ non-trivially over $V=L\otimes_{\F_p[T]}k_{\infty}.$ The theta function for $Q$
is then defined by
\[\Theta_Q\left(\left(\left(\begin{smallmatrix} \sqrt{v} & u/\sqrt{v}\\ 0 & 1/\sqrt{v}\end{smallmatrix}\right),1\right)\right)=|v|^{3/4}(\sqrt{v},\varpi_{\infty})_{\infty}^{\frac{v_{\infty}(v)}{2}}\sum_{l\in L}e(Q(l)T^2u)\chi_{\mathcal{O}_\infty}(Q(l)v) \]
\[= |v|^{3/4}(\sqrt{v},\varpi_{\infty})_{\infty}^{\frac{v_{\infty}(v)}{2}}\sum_{D\in \F_p[T]}r_Q(D)e(DT^2u)\chi_{\Ooi}(Dv)\]

Where $\left(\left(\begin{smallmatrix}\sqrt{v}&u/\sqrt{v}\\0&1/\sqrt{v}\end{smallmatrix}\right),1\right)\in \widetilde{SL}_2(k_\infty)$, $e(x)$ is an additive character, $(\cdot\,,\cdot)_{\infty}$ denotes the Hilbert symbol at the completion $k_{\infty}$ (for a precise definition of these see \S\ref{not11} and \S\ref{metgp}), $|v|=|v|_{\infty}$, $\Ooi = \F_p[[T^{-1}]]$ and $\chi_{\mathcal{O}_{\infty}}$ is the characteristic function of $\mathcal{O}_{\infty}$. We form the genus theta function $\Theta_G$ as in the number field case and prove that $\Theta_G-\Theta_Q$ is a cusp form. This reduces the question of representation
numbers to understanding the Fourier coefficients of $\Theta_G$ and bounding the Fourier coefficients of cusp forms. The Fourier coefficients of $\Theta_G$ are
readily read off from the Siegel mass formula. To bound the Fourier coefficients of cusp forms we follow \cite{KS} and \cite{B} to prove a Waldspurger type formula.
A nice feature of the function field case is that since the Riemann hypothesis is known, we can both prove what is essentially the optimal bound for Fourier coefficients of cusp
forms, and give an \emph{effective} lower bound on the Fourier coefficients of $\Theta_G$ coming from the main term. This is primarily because the Riemann hypothesis
rules out Siegel zeroes. This being said we also remark that unlike in the number field case, the necessary estimates \emph{do not} immediately follow from the Riemann hypothesis, cf. \S\ref{sectiondrinfeld}.

Before describing our main results we need to introduce some notation and terminology; this will be explained in detail in \S\ref{sectionfunctionfields}. Let $\widetilde{SL}_2(k_{\infty})$ be the metaplectic cover of $SL_2(k_{\infty})$ given by the explicit cocycle described in \S\ref{metgp}. We write elements of $\widetilde{SL}_2(k_{\infty})$ as $(g;\delta)$ where $g\in SL_2(k_{\infty})$ and $\delta\in\{\pm1\}$. Let $\varpi_{\infty}=T^{-1}$ be a generator for the maximal ideal of the ring of integers $\mathcal{O}_{\infty}$ of $k_{\infty}$. For $N\in\mathbb{F}_p[T]$ and $n\in\mathbb{N}$, let $\widetilde{\Gamma}_0(N)$ and $\iota(\widetilde{K}_{0}(\varpi_{\infty}^n))$ denote the following subgroups of $\widetilde{SL}_2(k_{\infty})$:
\[\widetilde{\Gamma}_0(N)=\setdef{\eta\left(\left(\begin{smallmatrix}a&b\\c&d\end{smallmatrix}\right)\right)}{\left(\begin{smallmatrix}a&b\\ c&d\end{smallmatrix}\right)\in SL_2(\mathbb{F}_p[T]),\,\,c\equiv0\;\text{mod}\,{N}}\subset\widetilde{SL}_2(k_{\infty})\]
\[\iota(\widetilde{K}_0(\varpi_{\infty}^n))=\setdef{\iota\left(\left(\begin{smallmatrix}a&b\\c&d\end{smallmatrix}\right)\right)}{\left(\begin{smallmatrix}a&b\\ c&d\end{smallmatrix}\right)\in SL_2(\mathcal{O}_{\infty}),\,\,c\equiv0\; \text{mod}\, \varpi_{\infty}^n}\subset \widetilde{SL}_2(k_{\infty})\]
Where $\eta:SL_2(\mathbb{F}_p[T])\rightarrow \widetilde{SL}_2(k_{\infty})$ and $\iota: SL_2(\mathcal{O}_{\infty})\rightarrow \widetilde{SL}_2(k_{\infty})$ are as defined in \S\ref{metgp}. We call functions, $F$, on the space $\widetilde{\Gamma}_0(N)\backslash\widetilde{SL}_2(k_{\infty})\slash \iota(\widetilde{K}_0(\varpi_{\infty}^n))$ that satisfy $F((I_2,-1))=-F((I_2,1))$, where $I_2$ stands for the $2\times2$ identity matrix, metaplectic functions of level $N$, depth $n$. We note that the latter condition is to ensure that the functions $F$ does not factor through $SL_2$. We denote the space of level $N$ depth $n$ metaplectic functions by $\widetilde{M}_n(\widetilde{\Gamma}_0(N))$.

At this point we want to take a step back momentarily and explain some of the technical and philosophical points of the paper. First of all we will (as much as possible) be working with the singled out place $k_{\infty}$ and local group $\widetilde{SL}_2(k_{\infty})$. The reason (for us choosing such an approach) is mainly the application to quadratic forms we described above. We want to keep the analogy with the corresponding representability problem over number fields as close as possible and the classical language serves this purpose better.

We should, however, note the appearance of the subgroups $\iota(\widetilde{K}_0(\varpi_{\infty}^n))$. As is well known in the number field case, the weight of a modular form is dictated by the action of the maximal compact subgroup of $SL_2(\mathbb{R})$ on the associated representation. In the case of a function field a similar phenomenon occurs. The difference, however, is that the maximal compact subgroup (which is the circle group $S^1$ in the classical picture) is now $\iota(SL_2(\mathcal{O}_{\infty}))\times (I_2,\pm1)$. In particular it is nonabelian and hence has a much richer representation theory\footnote{This feature is already present in the number field case when one works over an imaginary quadratic field and the group $SL_2(\mathbb{C})$ where the maximal compact subgroup is $SU(2)$.}. The theta functions, which are functions on the group $\widetilde{SL}_2(k_{\infty})$, transform under a representation $\rho$ of the maximal compact. This representation being smooth implies that there exists a number $n\in \mathbb{N}$ such that $\rho$ has a fixed vector under the congruence subgroup $\iota(\widetilde{K}_0(\varpi_{\infty}^n))$. In accordance with Moy-Prasad \cite{MP} we call the minimal such $n$ the \emph{depth}. In this paper we will be considering metaplectic forms of depth $0$ or $1$ and arbitrary level. Theta functions associated to anisotropic ternary quadratic forms turn out to be of depth $1$ and hence this depth restriction is enough for the application we have in mind. We also note that all of the above can be treated in a completely uniform manner (as is done in \S \ref{sectiongauss}) by the Weil representation on $\widetilde{SL}_2(\mathbb{A}_k)$ where $\mathbb{A}_{k}$ is the adele ring of $k$. We should further mention that the half-integral weight Ramanujan conjecture can also be expressed in a uniform adelic language and the methods of this paper are capable of establishing this in the most general case without any depth restrictions. One can also replace $\mathbb{F}_p(T)$ with an arbitrary function field $k$ over a finite field. This is the subject of the upcoming paper \cite{AJ}.

With the remarks of the preceding paragraph we now restrict to the depth $n\leq1$ case. A metaplectic function, $F\in \widetilde{M}_n(\widetilde{\Gamma}_0(N))$ is called \emph{cuspidal} if the constant Fourier coefficient of $F$ vanishes at each cusp (see equation (\ref{cuspidal})). The space of cuspidal metaplectic functions of level $N$ and depth $n$ are denoted by $\widetilde{S}_n(\widetilde{\Gamma}_0(N))$. We will be considering cuspidal metaplectic functions that are eigenfunctions of certain operators that fill the role of the Laplacian in the archimedean picture. These are explained in detail in \S\ref{halfintegralwhittakerun} and \S\ref{halfintegralwhittakerram}. Let $\widetilde{\Delta}$ denote right convolution with the characteristic function of the double coset $\iota(SL_2(\mathcal{O}_{\infty}))\backslash \left(\left(\begin{smallmatrix}\varpi_{\infty}&0\\0&\varpi_{\infty}^{-1}\end{smallmatrix}\right),1\right)/\iota(SL_2(\mathcal{O}_{\infty}))$. For any polynomial $N$, $\widetilde{\Delta}$ acts on the space, $\widetilde{S}_0(\widetilde{\Gamma}_0(N))$, of cuspidal metaplectic functions of level $N$, depth $0$.  An $F\in \widetilde{S}_0(\widetilde{\Gamma}_0(N))$ that is an eigenfunction of $\widetilde{\Delta}$ will be called \emph{a cuspidal metaplectic form} of level $N$, depth $0$. For cuspidal metaplectic functions of depth $1$ the operator $\widetilde{\Delta}$ gets replaced by $\widetilde{W}_{\infty}$, which is defined to be right convolution by the characteristic function of the double coset $\iota(\widetilde{K}_0(\varpi_{\infty}))\backslash\left(\left(\begin{smallmatrix}0&-\varpi_{\infty}^{-1}\\ \varpi_{\infty}&0\end{smallmatrix}\right),1\right)/\iota(\widetilde{K}_0(\varpi_{\infty}))$. For any polynomial $N$, $\widetilde{W}_{\infty}$ acts on the space ($\widetilde{S}_1(\widetilde{\Gamma}_0(N))$) of cuspidal metaplectic functions of level $N$, depth $1$. We will call  an $F\in \widetilde{S}_1(\widetilde{\Gamma}_0(N))$ that is an eigenfunction of $\widetilde{W}_{\infty}$,  \emph{a cuspidal metaplectic form} of level $N$, depth $1$.

On the other hand, at each finite place $P\nmid N$ we have the usual Hecke operators $\widetilde{T}_{P^2}$ acting on the metaplectic functions of level $N$, depth $n$, $n\leq1$. We will call a cuspidal metaplectic form $F$ of level $N$, depth $n$, that is an eigenfunction of all Hecke operators $\widetilde{T}_{P^2}$ for $P\nmid N$, \emph{a cuspidal metaplectic Hecke eigenform} of level $N$, depth $n$.

A cuspidal metaplectic form $F\in \widetilde{S}_0(\widetilde{\Gamma}_0(N))$ has a Fourier-Whittaker expansion at the standard cusp as
\[F(w)=\sum_{D\in\mathbb{F}_p[T]}\lambda_{F}(D)(D,\sqrt{v})_{\infty}e(T^2Du)\widetilde{W}_{n,i\theta}(Dv)\]
where  $w=\left(\left(\begin{smallmatrix}\sqrt{v}&u/\sqrt{v}\\0&1/\sqrt{v}\end{smallmatrix}\right),1\right)$, $(\cdot\,,\cdot)_{\infty}$ denotes the Hilbert symbol at the place $\infty$ as usual, and the Whittaker functions $\widetilde{W}_{n,i\theta}$ are as defined in \S\ref{halfintegralwhittakerun} and \S\ref{halfintegralwhittakerram} respectively.  With these definitions in hand we can now state the main result of the paper, which gives an optimally sharp estimation for the Fourier-Whittaker coefficients of metaplectic forms.\\

 \begin{thm}[Metaplectic Ramanujan conjecture]\label{ramanujan}
Let $F(w)$ be a cuspidal metaplectic form of level $N$ and depth $0$ or $1$. If $D$ is a square-free polynomial of even degree that is relatively prime to $N$, the Fourier-Whittaker coefficients
$\lambda_F(D)$ of $F$
satisfy the bound $$|\lambda_F(D)|\leq C_{F,\epsilon}\;|D|^{-\frac12+\epsilon}$$
where $|D|=p^{deg(D)}$ and $C_{F,\epsilon}$ is an effective constant depending only on $F$ and $\epsilon.$

\end{thm}

The main application of Theorem \ref{ramanujan} to the representation of polynomials by quadratic forms is the following assertion:\\

\begin{thm}\label{representationTheorem}
Let $Q$ be a ternary quadratic form over $\mathbb{F}_p[T]$ that is anisotropic over $k_{\infty}$. For $D\in \mathbb{F}_p[T]$ such that $\deg(D)$ has the same parity as $\deg(\disc(Q))$ and $(D,\disc(Q))=1$, let $r_Q(D)$ denote the number of times $Q$ represents the polynomial $D$ and $r_G(D)$ the number of times $D$ is represented by the genus of $Q$. Then
$$r_Q(D)=r_{G}(D)+O_{Q,\epsilon}(|D|^{1/4 +\epsilon}).$$
where the implied constant is effective and depends only on $Q$ and $\epsilon$.
\end{thm}

The following local to global principle now follows immediately from Theorem \ref{representationTheorem}.\\

\begin{cor}\label{representationCor} For $Q$ as above there exists an effective constant $A_{Q}$ such that every square-free $D\in\mathbb{F}_p[T]$ with $\deg(D)$ of the same parity as $\deg(\disc(Q))$, satisfying $(D,\disc(Q))=1$ and $|D|>A_Q$ is represented by $Q$ as long as there are no local obstructions.
\end{cor}

The paper is structured as follows: In \S\ref{sectionfunctionfields} we introduce and develop the basic notions over function fields including quadratic
forms, cuspidal automorphic and metaplectic forms, their Fourier expansions, Hecke operators, Whittaker functions and L-functions attached to cuspidal automorphic functions. For each of these notions we first give a treatment for depth $0$ functions and then describe the relevant changes for depth $1$ functions.
In \S\ref{sectiondrinfeld} we identify precise conditions under which the Riemann Hypothesis implies the Lindel\"of Hypothesis for a family of zeta functions in the function field setting. In \S\ref{sectionshimura} we study the Shimura and Maass-Shintani lifts and derive explicit normalisations. To clarify the technical proofs we first give detailed proofs for the full level and depth $0$ case and then point out the differences in the case of arbitrary level and depth $1$. \S\ref{sectionwaldspurger} applies
the results of \S\ref{sectionshimura} to prove a Waldspurger type formula and the metaplectic Ramanujan conjecture. \S\ref{sectionquadratic} gives the promised application to
representing polynomials by quadratic forms, and in \S\ref{Example} we provide a concrete example to the type of representability question we answer in the paper. Finally \S{A} develops the Weil representation, uses it to define theta functions, and proves their transformation
properties.Though it is fundamental and referred to throughout the paper, \S{A} is quite technical and can be skipped on a first reading.
\section{Automorphic Forms Over Function Fields}\label{sectionfunctionfields}

\subsection{Notation}\label{not11}

We will fix some notation which will be used throughout the paper. Let $\mathbb{F}_p$ denote the finite field with $p$ elements. For convenience, throughout  the
paper  we will be assuming that $p\equiv1\pmod{4}$. Let $R$ denote the ring $\mathbb{F}_p[T]$, and $k$ the field $\mathbb{F}_p(T)$. We also fix the prime at ``infinity'' which corresponds to the
valuation $v_{\infty}$ such that $v_{\infty}(S_1(T)/S_2(T))=\deg{S_2}-\deg{S_1}$, for $S_1,S_2\in\mathbb{F}_p[T]$. The completion of $k$ with respect to this valuation is isomorphic to $\mathbb{F}_p((T^{-1}))=k_{\infty}$, and the ring of integers of $k_{\infty}$ is $\mathcal{O}_{\infty}$. We will denote the uniformizer $T^{-1}$ of $\mathcal{O}_{\infty}$ by $\varpi_{\infty}$ and use them interchangably. We denote the norm induced by $v_{\infty}$ by $|\cdot|_{\infty}$, i.e. for $x\in k_{\infty}$, $|x|_{\infty}=p^{-v_{\infty}(x)}$. For any set $A$ we will use the notation, $\chi_{A}$, to denote the characteristic function of $A$. Once and for all we fix an additive character, $e(x)$, on $k_{\infty}$ defined as follows. Let $x=\sum_{j=-\infty}^{k}a_jT^j\in k_{\infty}$ where $a_j\in\mathbb{F}_p$. Choose a lift $a_1^*$ of $a_1$ to $\mathbb{Z}$, and define $e(x):=e^{2\pi ia_1^*/p}$. Since $e(x)$ is independent of the choice of the lift $a_1^*$, we will simply write $e(x)=e^{2\pi i a_1/p}$. We also note that the conductor of $e(x)$ is $\mathcal{O}_{\infty}$, i.e. $e(x)=1$ for all $x\in\mathcal{O}_{\infty}$ and $e(x)$ is non-trivial on $\varpi_{\infty}^{-1}\mathcal{O}_{\infty}$.

The norm of an element $S\in\mathbb{F}_p[T]$ will be denoted by $|S|=|S|_{\infty}=p^{\deg{S}}$. The finite valuations of $k$ can be identified with the monic irreducible polynomials $P\in\mathbb{F}_p[T]$ by setting $k_P$ to be the fraction field of the completion of $R$ with respect to $P$. Abusing notation, for a valuation $v$ corresponding to $P$ and an element $c\in R$ we sometimes write $v\mid c$ to denote $P\mid c$.

Define the function field analogues of the classical gamma and zeta functions by
\[\Gamma_k(s)=\frac{1}{1-p^{-s}}\]
\[\zeta_k(s)=\sum_{a\in R\backslash \{0\}}\frac{1}{|a|^s}\]
Note that because the valuations of $k$ are indexed by \emph{monic} irreducible polynomials, the Euler product for $\zeta$ is $\zeta_k(s)=(p-1)\prod_P(1-|P|^{-s})^{-1}$, where the product is over all monic irreducible polynomials $P\in R$.

\subsubsection{Hilbert Symbols and Quadratic Reciprocity}
 %For $u\in k_v$ we will denote the quadratic character which is $1$ if $u$ is a square in $k_v$ and $-1$ if not, by $(u\:|\:\varpi_{v})$.
 Recall that we are assuming $p\equiv1 \bmod 4$. For $c\in R\backslash\{0\}$ such that the ideal generated by $c$ is prime in $R$ (i.e. $c$ is an irreducible and we do not care about the leading coefficient), and for $d\in R$ we will define the Legendre symbol mod $c$ by $$\left(\frac{d}{c}\right):=\begin{cases}0&\text{if $\gcd(c,d)\neq1$}\\ 1 & \textrm{$\gcd(c,d)=1$ and $d$ is a square modulo $c$}\\ -1 & \textrm{otherwise}\end{cases}.$$
Where we use $\gcd(c,d)=1$ to mean that the ideal generated by $c$ and $d$ is the whole ring $R$. As in the classical case of the Jacobi symbol, ee then extend this definition multiplicatively in the $c$ variable. Now let $v$ be a valuation of $k$ and let $k_v$ be the completion of $k$, $\mathcal{O}_v\subset k_v$ the ring of integers of $k_v$, and $\varpi_v$ a uniformizer at $v$.
For $\alpha,\beta\in k_v^{\times}$ we denote the Hilbert symbol at the place $v$ by $(\alpha,\beta)_v$. By definition
$$(\alpha,\beta)_v=\begin{cases}1 & \textrm{if there is a non-zero solution to $z^2=\alpha x^2 +\beta y^2$ with $(x,y,z)\in k_v^3$}\\-1 & \text{otherwise}\end{cases}$$ Recall that the Hilbert symbol is multiplicative in each variable. Moreover, since we are
working over $\F_p$ with $p\equiv 1\mod{4}$ it can easily be checked that the $(-1)^{(\cdot,\cdot)_v}$ is the unique symplectic form on $(k_v)^{\times}/((k_v)^{\times})^2\times(k_v)^{\times}/((k_v)^{\times})^2$.

Explicitly, for $v=\infty$, and $\varpi_{\infty}=T^{-1}$, we have $$(x,y)_{\infty} = \left(\frac{x_0}{T}\right)^m\left(\frac{y_0}{T}\right)^n$$ where $x=x_0\varpi_{\infty}^n,y=y_0\varpi_{\infty}^m$ and $x_0,y_0\in\Ooi^{\times}$.\\

The most important property of the Hilbert symbol is the product formula (cf. \cite{Hasse}), $\prod_{v}(\alpha,\beta)_v=1$, where the product is now over all valuations (including valuation at infinity).\\

\begin{lemma}\label{hilquad}
For $\alpha,\beta\in R\backslash\{0\}$ relatively prime we have $$\left(\frac{\alpha}{\beta}\right)=\prod_{v|\beta}(\alpha,\beta)_v$$
\end{lemma}
\begin{proof}
By the product formula,$$\prod_{v|\beta}(\alpha,\beta)_v=\prod_{v\nmid\beta}(\alpha,\beta)_v=(\alpha,\beta)_{\infty}\prod_{v|\alpha}(\alpha,\beta)_v$$ which shows that both sides are defined multiplicatively in $\beta$. Hence it is enough to consider the case when $\beta$ is a prime and denote the valuation corresponding to $\beta$ by $\mathfrak{b}$. Let $k_{\mathfrak{b}}$ be the completion of $k$ at $\mathfrak{b}$. Then we have the following equalities
\[\prod_{v\mid \beta}(\alpha,\beta)_v=\left(\alpha,\beta\right)_{\mathfrak{b}}= \left(\frac{\alpha}{\beta}\right)\]
%Where the last equality of the first line comes from the fact that $\gcd(\alpha,\beta)=1$.
\end{proof}

The product formula implies, in particular, the following quadratic reciprocity formula for $k=\mathbb{F}_p(T)$:\\

\begin{lemma}\label{quadrec}
For $\alpha,\beta\in R\backslash\{0\}$ relatively prime we have $$\left(\frac{\alpha}{\beta}\right)\cdot\left(\frac{\beta}{\alpha}\right)=(\alpha,\beta)_{\infty}.$$
\end{lemma}
\begin{proof}
Using the fact that if $v$ doesn't divide $\alpha$ or $\beta$ we have $(\alpha,\beta)_v=1$ together with lemma \ref{hilquad} we have
$$\left(\frac{\alpha}{\beta}\right)\cdot\left(\frac{\beta}{\alpha}\right)=\prod_{v\neq\infty}(\alpha,\beta)_v=(\alpha,\beta)_{\infty}$$ as desired.
\end{proof}

\subsection{The Metaplectic group}\label{metgp}

Following the treatment by Gelbart given in \cite{G}, we define the metaplectic group, $\widetilde{SL}_2(k_{\infty})$, to be the double cover of $SL_2(k_{\infty})$ defined by the following cocycle: For a matrix $g\in SL_2(k_{\infty})$ define $$X(g):=\begin{cases} c& c\neq 0\\ d & \textrm{otherwise}\end{cases}$$ and $$\epsilon(g_1,g_2):=(X(g_1),X(g_2))_{\infty}(X(g_2),X(g_3))_{\infty}(X(g_1),X(g_3))_{\infty}$$ where $g_3=g_1g_2$. That this is actually a cocycle is the main result of \cite{Ku}. Thus elements of $\widetilde{SL}_2(k_{\infty})$ can be written as $(g,\delta)$ where $g\in SL_2(k_{\infty}),\delta=\pm 1$ and the multiplication is
$(g_1,\delta_1)(g_2,\delta_2)=(g_1g_2,\delta_1\delta_2\epsilon(g_1,g_2))$.

In what follows, we will be interested in various operators (for instance Hecke operators), whose definition will depend on splitting properties of $\epsilon$ over various subgroups of $k_{\infty}$. First consider the subgroup $SL_2(\mathcal{O}_{\infty})$. By Lemma 2.9 of \cite{G} we have the following splitting of $\widetilde{SL}_2(k_{\infty})$ over $SL_2(\Ooi)$: For $g\in SL_2(\Ooi)$ define $$\kappa(g):=\begin{cases}(c,d)_{\infty}&\text{if $c\neq 0$ and $c\notin\Ooi^{\times}$} \\ 1&\text{otherwise}\end{cases}$$ Then the map $\iota:SL_2(\Ooi)\rightarrow\widetilde{SL}_2(k_{\infty})$ given by $\iota(g)=(g,\kappa(g))$ is a
group homomorphism. We mention that while Gelbart only works over fields of characteristic 0 an inspection of the proof shows that it carries over verbatim for local fields of odd residual characteristic.

The other subgroup we would like to consider is $SL_2(R)$. Let $\eta$ be the map of sets $\eta:SL_2(R)\rightarrow \widetilde{SL}_2(k_{\infty})$ given by
$$\eta(g)=\begin{cases}(g,\left(\frac{d}{c}\right))&\text{if $c\neq0$}\\ (g,1)&\text{if $c=0$ }\end{cases},$$
where $g=\left(\begin{smallmatrix} a&b\\ c&d\end{smallmatrix}\right)$. Then,\\

\begin{lemma}
$\eta$ is a group homomorphism.
\end{lemma}
\begin{proof}For simplicity we will assume that all of the $c_i$'s are non-zero though the other cases can be treated similarly. We will show that $\epsilon(g_1,g_2)=\prod_{i=1}^{3}\left(\tfrac{d_i}{c_i}\right)$. By the product formula for the Hilbert symbol we have
\begin{align*}
\epsilon(g_1,g_2)&=\prod_{1\leq i<j\leq 3}(c_i,c_j)_{\infty}\\
&=\prod_{v\neq \infty}\prod_{1\leq i<j\leq 3}(c_i,c_j)_v
\end{align*}

We will analyze the factors in the above product for each valuation separately. First note that for $u_1,u_2\in\mathcal{O}^{\times}$ we have $(u_1,u_2)=1$. Therefore the above product is finite and runs over places dividing $c_1c_2c_3$. Furthermore note that since $g_3=g_1g_2$, $c_3=c_1a_2+d_1c_2$. Therefore if $v$ divides two of the $c_i$'s it divides the third. We will calculate the above product separately for each valuation.

Before we start the computation we recall that for $a\in\mathcal{O}_v^{\times}$ and $b\in\varpi_v\mathcal{O}_v$, then $(a,b)_v$ depends only on the square class of $a\bmod b$.

\begin{itemize}
\item Let $v\nmid \gcd(c_1,c_2,c_3)$. In this case there are three possibilities according to $v\mid c_1$, $c_2$ or $c_3$.
\begin{itemize}
\item If $v\mid c_1$, then $v\nmid c_2c_3$, therefore $c_2,c_3\in\mathcal{O}_v^{\times}$ and hence $(c_2,c_3)_v=1$. Therefore $(c_1,c_2)_v(c_1,c_3)_v(c_2,c_3)_v=(c_1,c_2c_3)_v$. Since $c_3=c_1a_2+d_1c_2$, $c_2c_3\equiv d_1c_2^2\bmod c_1$. Therefore $(c_1,c_2c_3)_v=(c_1,d_1c_2^2)_v=(c_1,d_1)_v$.
\item If $v\mid c_2$, then $v\nmid c_1c_3$, therefore $c_1,c_3\in\mathcal{O}_v^{\times}$ and hence $(c_1,c_3)_v=1$. Therefore $(c_1,c_2)_v(c_1,c_3)_v(c_2,c_3)_v=(c_1c_3,c_2)_v$. Since $c_3=c_1a_2+d_1c_2$, $c_1c_3=c_1^2a_2\bmod c_2$. Also since $g_2\in SL_2(R)$ we have $a_2d_2-b_2c_2=\det(g_2)=1$, and therefore $a_2d_2\equiv1\bmod c_2\Rightarrow a_2\equiv d_2^{-1}\bmod c_2$. Combining these we then get $(c_1c_3,c_2)_v=(c_1^2a_2,c_2)_v=(c_1^2d_2^{-1},c_2)_v=(d_2,c_2)_v$.
\item If $v\mid c_3$, then $v\nmid c_1c_2$, therefore $c_1,c_2\in\mathcal{O}_v^{\times}$ and hence $(c_1,c_2)_v=1$. Therefore $(c_1,c_2)_v(c_1,c_3)_v(c_2,c_3)_v=(c_1c_2,c_3)_v$.

\begin{align*}
(c_1c_2,c_3)_v&=(-c_1c_2,c_3)_v\\
&=(-c_1c_2+d_2c_2c_3,c_3)_v\\
&=(-c_1c_2+d_2c_2(c_1a_2+d_1c_2),c_3)_v\\
&=(-c_1c_2+c_2c_1a_2d_2+d_1d_2c_2^2,c_3)_v\\
&=(-c_1c_2+c_2c_1(1+b_2c_2)+d_1d_2c_2^2,c_3)_v\\
&=(c_2^2(c_1b_2+d_1d_2),c_3)_v\\
&=(c_2^2d_3,c_3)_v\\
&=(d_3,c_3)_v
\end{align*}
\end{itemize}

\item Let $v\mid \gcd(c_1,c_2,c_3)$. Note that since $\det(g_i)=1$ for every $i$, $d_i\in\mathcal{O}_v^{\times}$ for every $i$. Let $c_i=u^{t_i}\varpi_v^{k_i}$ where $u\in\mathcal{O}_v^{\times}\backslash (\mathcal{O}_v^{\times})^2$, $k_i\geq1$ and $t_i\in\{0,1\}$. Then
\begin{align*}
(c_1,c_2)_v(c_1,c_3)_v(c_2,c_3)_v&=(u^{t_1}\varpi_v^{k_1},u^{t_2}\varpi_v^{k_2})_v(u^{t_1}\varpi_v^{k_1},u^{t_3}\varpi_v^{k_3})_v(u^{t_2}\varpi_v^{k_2},u^{t_3}\varpi_v^{k_3})_v\\
&=(u^{t_1},\varpi_v^{k_2+k_3})_v(u^{t_2},\varpi_v^{k_1+k_3})_v(u^{t_3},\varpi_v^{k_1+k_2})_v
\end{align*}

Note that since $c_3=c_1a_2+d_1c_2$ and $d_1,a_2\in\mathcal{O}_v^{\times}$, we need to have either $k_1=k_2\leq k_3$ of $k_3=\min\{k_1,k_2\}$. In particular at least two of the $k_1,k_2,k_3$ have to be the same. Then by the above computation we have
\begin{align*}
(c_1,c_2)_v(c_1,c_3)_v(c_2,c_3)_v&=(u^{t_1},\varpi_v^{k_2+k_3})_v(u^{t_2},\varpi_v^{k_1+k_3})_v(u^{t_3},\varpi_v^{k_1+k_2})_v\\
&=1
\end{align*}
On the other hand since $d_3=c_1b_2+d_1d_2$, $d_3\equiv d_1d_2\bmod \varpi_v$. Then we have,
\begin{align*}
(c_1,d_1)_v(c_2,d_2)_v(c_3,d_3)_v&=(c_1,d_1)_v(c_2,d_2)_v(c_3,d_1d_2)_v\\
&=(u^{t_1}\varpi_v^{k_1},d_1)_v(u^{t_2}\varpi_v^{k_2},d_2)_v(u^{t_3}\varpi_v^{k_3},d_1d_2)_v\\
&=1
\end{align*}
Therefore in this case we have $(c_1,d_1)_v(c_2,d_2)_v(c_3,d_3)_v=(c_1,c_2)_v(c_1,c_3)_v(c_2,c_3)_v$.

\end{itemize}

By the above argument we therefore get

 \begin{align*}
\epsilon(g_1,g_2)&=\prod_{1\leq i<j\leq 3}(c_i,c_j)_{\infty}\\
&=\prod_{v\neq \infty}\prod_{1\leq i<j\leq3}(c_i,c_j)_v\\
&= \prod_{v\mid c_1c_2c_3}\prod_{1\leq i<j\leq3}(c_i,c_j))v\\
&=\prod_{i=1}^3\prod_{v\mid c_1}(c_i,d_i)_v\\
&=\prod_{i=1}^3\left(\frac{d_i}{c_i}\right)
 \end{align*}
 as desired.
 \end{proof}

\subsection{Quadratic forms over function fields}

We begin with a review of quadratic forms over function fields. Let $Q(\vec{x})$ be a quadratic form in n-variables,
$\vec{x}=(x_1,x_2,\dots,x_n)$. $Q$ is said to be non-degenerate if its associated bilinear form
$$B(\vec{x},\vec{y})=\frac{1}{2}\left(Q(\vec{x}+\vec{y})-Q(\vec{x})-Q(\vec{y})\right)$$ is non-degenerate.
We also say that a form is \emph{anisotropic} over a field $k$ if it does not represent 0 non-trivially over that field. We remark that
over the field $\mathbb{R}$ of real numbers, the anisotropic forms are precisely the definite ones.

We shall be considering also quadratic forms $Q$ over rings $S$ instead of fields. Given such a ring $S$,a free $S$-lattice $L$, and a quadratic form $Q$ on $L$, we can represent $Q$
as a symmetric matrix $A_Q$ by picking a basis $l_1,l_2,\dots,l_n$ for $L$ over $S$, and writing $A_Q=[Q(l_i,l_j)\cdot\frac{1+\delta_{ij}}{2})]_{\{i,j\}}$ so that
$$Q(\sum_{i=1}^n \alpha_i l_i) = \vec{\alpha}\cdot A_Q\cdot \vec{\alpha}^t$$ where $\vec{\alpha}=(\alpha_1,\alpha_2,\dots,\alpha_n)$ and $\vec{\alpha}^t$ denotes the transpose vector. Changing bases by an element of $g\in GL_n(S)$ amounts to changing $A_Q$ to $gA_Qg^t$. Thus we can define the discriminant $\disc(Q)$ to be the determinant of $A_Q$, and this is a well
defiend element of $S/(S^{\times})^2.$ In particular, note that the ideal $(\disc(Q))$ is well defined in $S$.

\begin{lemma}\label{inverselattice}
Let $Q(x,y,z)$ be an anistropic quadratic form over $k_{\infty}$. Then for all $m\in\mathbb{Z}$, $Q^{-1}(\varpi_{\infty}^m\mathcal{O}_{\infty})$ is an $\mathcal{O}_{\infty}$ lattice in $k_{\infty}^3$.
\end{lemma}
\begin{proof}
First note that since we are not over a field of even characteristic, we can diagonalize the form $Q$, so we can assume it is of the form $Q(x,y,z)=ax^2+by^2+cz^2$. Only the square-class of $a,b,c$ is relevant, so each of $a,b,c$ can be assumed to be one of $1,u,T,uT$, for $u$ a quadratic non-residue in
$\mathbb{F}_p$. Moreover, we can scale to assume that $a=1$. Note that the operations of changing basis and scaling done so far do not affect the statement of
the lemma.

Next, for $Q$ to be anisotropic, we can not have both $b$ and $c$ be elements of $\mathbb{F}_p$. Going through the cases, we see that by scaling (we note that scaling does not effect the statement of the lemma) and changing
basis we can turn every ternary anisotropic form into $Q_0(x,y,z)=x^2+uy^2+Tz^2$. Now, it is easily seen that $v_{\infty}(x^2+uy^2+Tz^2)=\max\{v_{\infty}(x^2),v_{\infty}(uy^2),v_{\infty}(Tz^2)\}$ and so $Q_0^{-1}(\Ooi) = \{(x,y,z)\mid x,y\in\Ooi,z\in\varpi_{\infty}\Ooi\}.$
\end{proof}

\subsection{Background on symmetric spaces}\label{backgr}
%Let $G$ will denote the algebraic group $PGL_2$, and $K$ will be the field $\mathbb{F}_p(t)$. Also, we define the Adeles
%$$\mathbb{A}=\prod_{v}^{'}[K_v : \mathcal{O}_V]$$ to be the restricted direct product over all valuations $v$ of $K$ of the
%completion $K_v$ with respect to the ring of integers $\mathcal{O}_V$, and define $\mathbb{O}=\prod_v \mathcal{O}_V.$ Now, since we are in a function
%field setting, the valuations of $K$ correspond to closed points on the non-singular complete algebraic curve corresponding to $K$;
%in this case, the projective space $\mathbb{P}^1(\mathbb{F}_p)$. The symmetric space
%
%$$X := G(K)\backslash G(\mathbb{A})\slash G(\mathbb{O})$$

We denote by $\mathbb{H}$ the ``upper half plane"  $$\mathbb{H} = PGL_2(k_{\infty})\slash PGL_2(\mathcal{O}_{\infty})$$ A complete set of coset
representatives for $\mathbb{H}$ is
$$\mathbb{H}=\setdef{
\left(\begin{smallmatrix}
y & x \\
0 & 1
\end{smallmatrix}\right)}{y\in k_{\infty}^{\times}\slash \mathcal{O}_{\infty}^{\times},\; x\in k_{\infty}\slash y\mathcal{O}_{\infty}}$$
Note that $\{T^m\mid m\in\mathbb{Z}\}$ forms a complete set of representatives for $y$. The space $\mathbb{H}$ carries a measure invariant under the action of $PGL_2(k_{\infty})$, which, in the $(x,y)$ coordinates, is given by $$\frac{dx\,dy}{|y|^2}$$ where $dx$ and $dy$ are normalized to give $\mathcal{O}_{\infty}$ measure 1.

\begin{subsubsection}{Metaplectic spaces}\label{metsp}

Recall from \S\ref{metgp} that we have a group homomorphism $\iota:SL_2(\Ooi)\rightarrow\widetilde{SL}_2(k_{\infty})$ given by $g\rightarrow (g,\kappa(g))$.
We will also be interested in functions on the space $$\widetilde{\mathbb{H}}=\widetilde{SL}_2(k_{\infty})/\iota(SL_2(\Ooi)).$$
A complete set of representatives for $\Tt$ is given by
$$\Tt  = \setdef{\left(\left(\begin{smallmatrix}
\sqrt{v} & u/\sqrt{v} \\
0 & 1/\sqrt{v} \\
\end{smallmatrix}\right),\pm1\right)}{ v\in (k_{\infty}^{\times})^2/(\mathcal{O}^{\times})^2,\; u\in k_{\infty}\slash v\mathcal{O}_{\infty}}$$
 Note that a function $F:\widetilde{SL}_2(k_{\infty})\rightarrow \mathbb{C}$ which is invariant under the central element $(Id,-1)\in \widetilde{SL}_2(k_{\infty})$ descends to function on $SL_2(k_{\infty})$. We would like to disregard these functions which are not ``really'' metaplectic. We thus insist on our functions $f$ to satisfy $F((Id,-1)\cdot g)=-F(g)$.
 If $F$ is a function on $\Tt$ and $g\in SL_2(k_{\infty})$ we often write $F(g)$ to mean $F((g,1))$ by abuse of notation.

The arithmetic of the metaplectic group enters when we study functions on $\widetilde{SL}_2(k_{\infty})$ which are left invariant under $\eta(SL_2(R))$. We also note that the metaplectic group is considered from the adelic viewpoint in \S\ref{sectiongauss}.
%For $\gamma=\left(\begin{smallmatrix}a&b\\c&d\end{smallmatrix}\right)\in SL_2(R)$ and $g\in\Tt$ define the \emph{standard multiplier} by
%\[j(\gamma,g)=\left(\frac{d}{c}\right)\kappa(g)\kappa(\gamma g)\epsilon(\gamma,g).\] This is motivated by the following lemma
%\begin{lemma}
%Let $\gamma\in SL_2(R)$. A function $F$ on $\Tt$ satisfies $F(\gamma g)=j(\gamma,g)F(g)$ for all $g\in\Tt$ if and only if $\widetilde{F}$ is left-invariant
%under $\eta(\gamma)$.
%\end{lemma}
%\begin{proof}
%For $g\in SL_2(R)$, we have
%\begin{align*}
%\widetilde{F}(\eta(\gamma)\iota(g))&=\widetilde{F}(\gamma g, \epsilon(\gamma,g)\kappa(k(g))\left(\frac{d}{c}\right))\\
%&= \widetilde{F}(\gamma g,\kappa(\gamma g)\cdot j(\gamma,g))\\
%&= j(\gamma,g)F(\gamma g)\\
%\end{align*}
%as desired.
%\end{proof}

\end{subsubsection}

\subsubsection{Depth $1$}

For the application to anisotropic quadratic forms, we shall also need to talk about automorphic forms with certain ramification type at the place ``$\infty$''. More precisely, for each integer $n\geq0$ we introduce the subgroups $K_0(\varpi_{\infty}^n)$ and $\widetilde{K}_0(\varpi_{\infty}^n)$,
\[K_0(\varpi_{\infty}^n)=\begin{setdef}{\left(\begin{smallmatrix}a&b\\c&d\end{smallmatrix}\right)\in PGL_2(\mathcal{O}_{\infty})}{c\equiv0\mod{\varpi_{\infty}^n}}\end{setdef}\]
\[\widetilde{K}_0(\varpi_{\infty}^n)=\begin{setdef}{\left(\begin{smallmatrix}a&b\\c&d\end{smallmatrix}\right)\in SL_2(\mathcal{O}_{\infty})}{c\equiv0\mod{\varpi_{\infty}^n}}\end{setdef}\]

and denote $PGL_2(k_{\infty})\slash K_0(\varpi_{\infty})\textrm{ and } \widetilde{SL}_2(k_{\infty})\slash \iota(\widetilde{K}_0(\varpi_{\infty}))$ by  $\mathbb{H}_1$ and $\Tt_1$, respectively.

$\mathbb{H}_1$ can be decomposed into two ``components" with natural coordinate systems, which we will denote by $\mathbb{H}_1^u$ and $\mathbb{H}_1^l$ and refer to as the `\emph{upper}' and `\emph{lower}' components respectively. A set of representatives
for the upper component $\mathbb{H}_1^u$ is
\begin{equation}\label{measu}
\mathbb{H}^u_1=\setdef{\left(\begin{smallmatrix}
y & x \\
0 & 1
\end{smallmatrix}\right)}{y\in k_{\infty}^{\times}\slash \mathcal{O}_{\infty}^{\times},\; x\in k_{\infty}\slash y\mathcal{O}_{\infty}}
\end{equation}
and a set of representatives for the lower part $\mathbb{H}_1^l$ is given by
\begin{equation}\label{measl}
\mathbb{H}^l_1=\setdef{\left(\begin{smallmatrix}
x & y \\
1 & 0
\end{smallmatrix}\right)}{y\in k_{\infty}^{\times}\slash \mathcal{O}_{\infty}^{\times},\; x\in k_{\infty}\slash y\varpi_{\infty}\mathcal{O}_{\infty}}
\end{equation}

The space $\mathbb{H}_1$ has a measure invariant under left multiplication by $PGL_2(k_{\infty})$ exactly as we did before. We will normalize our measure so that it is compatible with the measure on $\mathbb{H}$ under the projection map $\mathbb{H}_1\rightarrow \mathbb{H}$. Each point $z\in\mathbb{H}$ has $p+1$ pre-images in $\mathbb{H}_1$, and in our coordinate system we have
\[pr^{-1}\left(\begin{smallmatrix}y&x\\0&1\end{smallmatrix}\right)=\left\{\left(\begin{smallmatrix}y&x\\0&1\end{smallmatrix}\right)\right\}\cup\begin{setdef}{\left(\begin{smallmatrix}x+jy&y\\1&0\end{smallmatrix}\right)}{j\in \mathbb{F}_p}\end{setdef}\]

The measure normalization we will use on the upper and lower components will be different because of the compatibility with the measure on $\mathbb{H}$. More precisely, for a fixed $y$, the point $\left(\begin{smallmatrix}y&x\\ 0&1\end{smallmatrix}\right)\in\mathbb{H}_1^u$ and the point $\left(\begin{smallmatrix}x&y\\1&0\end{smallmatrix}\right)\in \mathbb{H}_1^l$ have different measures, with respect to the measure $\mu$, when considered as an equivalence class. This is because $x$ is defined $\bmod\, y\mathcal{O}_{\infty}$ on $\mathbb{H}_1^u$, which gives the equivalence class of $\left(\begin{smallmatrix}y&x\\0&1\end{smallmatrix}\right)$ mass $1$, whereas since $x$ is defined $\bmod\, y\varpi_{\infty}\mathcal{O}_{\infty}$ on $\mathbb{H}_1^l$, the equivalence class of $\left(\begin{smallmatrix}x&y\\1&0\end{smallmatrix}\right)$ gets mass $p^{-1}$. We will normalize our measure to compensate for the power of $p$. Thus, the invariant measure which pushes forward to $dxdy$ on $\mathbb{H}$ is
\[\mathbb{H}_1^u\longmapsto\frac{1}{p+1}\frac{dx\:dy}{|y|^2}\]
\[\mathbb{H}_1^l\longmapsto\frac{p}{p+1}\frac{dx\:dy}{|y|^2}\]

and this is the measure we will be using.

Likewise, $\Tt$ has two ``components", $\Tt^{u}_1$ and $\Tt^{l}_1$. A set of representatives for $\Tt_1^{u}$ is $$\Tt_1^u=\setdef{
\left(\left(\begin{smallmatrix}
v^{1/2} & u/v^{1/2} \\
0 & 1/v^{1/2} \\
\end{smallmatrix}\right),\pm1\right)}{ v\in (k_{\infty}^{\times})^2/(\mathcal{O}^{\times})^2,\; u\in k_{\infty}\slash v\mathcal{O}_{\infty}}$$

and a representative set for $\Tt^{l}_1$ is $$\Tt_1^l=\setdef{
\left(\left(\begin{smallmatrix}
 u/v^{1/2} & v^{1/2}\\
1/v^{1/2} & 0\\
\end{smallmatrix}\right),\pm1\right)}{v\in (k_{\infty}^{\times})^2/(\mathcal{O}^{\times})^2,\; u\in k_{\infty}\slash v\varpi_{\infty}\mathcal{O}_{\infty}}$$

Reasoning as above we get the invariant measures on the upper and lower components:

\[\Tt_1^u\longmapsto\frac{1}{p+1}\frac{du\:dv}{|v|^2}\]
\[\Tt_1^l\longmapsto\frac{p}{p+1}\frac{du\:dv}{|v|^2}\]

\subsection{Automorphic and metaplectic functions}

\subsubsection{Automorphic Functions} We denote $PGL_2(R)$ by $\Gamma,$ and for $N\in R$, we define $\Gamma_0(N)\subset\Gamma$ to be
$$\Gamma_0(N) = \setdef{\left(\begin{smallmatrix} a& b\\ c & d\end{smallmatrix}\right)\in PGL_2(R)}{c\equiv0\mod N}$$

We define the space of automorphic functions of level $N$, depth $0$,
to be the space of complex valued function on $\Gamma_0(N)\backslash\mathbb{H}.$
 $$M_0(\Gamma_0(N))=\{\phi:\Gamma_0(N)\backslash\mathbb{H}\rightarrow\mathbb{C}\}$$

Such a $\phi$ is moreover called \emph{cuspidal} if for any unipotent subgroup $U$ of $PGL_2(k_{\infty})$ such that $U\cap\Gamma_0(N)\neq 1$, the following
identity holds:
\begin{equation}\label{cuspidal}
\int_{\Gamma_0(N)\cap U\backslash U} \phi(nz) dn= 0,\textrm{   }\forall z\in \mathbb{H}
\end{equation}

We denote the space of cuspidal automorphic functions of level $N$, depth 0, by $S_0(\Gamma_0(N))$. If we take $U$ to be the upper-triangular matrix group
$U_{\infty} = \left\{\left(\begin{smallmatrix} 1& x\\ 0 & 1\end{smallmatrix}\right)\mid x\in k_{\infty}\right\}$
then (\ref{cuspidal}) says that the constant term in the Fourier expansion of $\phi\left(\begin{smallmatrix} y& x\\ 0 & 1\end{smallmatrix}\right)$ with respect to $x,$ vanishes. Cuspidality means that this is true for all the cusps of $\Gamma_0(N)\backslash\mathbb{H}$. It is well-known (cf. Corollary 1.2.3 of \cite{Ha}) that
the functions in $S_0(\Gamma_0(N))$ are supported on finitely many points and thus $S_0(\Gamma_0(N))$ is finite dimensional.

\subsubsection{Metaplectic Functions} We denote $\eta(SL_2(R))$ by $\widetilde{\Gamma},$ and for $N\in R$ we define $\widetilde{\Gamma}_0(N)\in\widetilde{\Gamma}$ to be
$$\widetilde{\Gamma}_0(N) = \setdef{\eta\left(\left(\begin{smallmatrix} a& b\\ c & d\end{smallmatrix}\right)\right)\in \widetilde{\Gamma}}{ c\equiv0\mod N}$$	
We define the space of metaplectic functions of level $N$ and depth $0$, $\widetilde{M}_0(\widetilde{\Gamma}_0(N))$ as follows.
$$\widetilde{M}_0(\widetilde{\Gamma}_0(N))=\begin{setdef}{F:\Tt\rightarrow\mathbb{C}}{F(\gamma w) = F(w)=-F(w\cdot(Id,-1)) ,\textrm{   }\forall \gamma\in \widetilde{\Gamma}_0(N)}\end{setdef}.$$

Such an $F$ is moreover called \emph{cuspidal} if it satisfies
\begin{equation}\label{mcuspidal}
\int_{\Gamma_0(N)\cap U\backslash U} \phi(\eta(n)w) dn= 0,\textrm{   }\forall w\in \Tt
\end{equation}  for every unipotent subgroup $U$ of $SL_2(k_{\infty})$ such that $U\cap \Gamma_0(N)\neq1$. We denote the space of cuspidal metaplectic functions of level $N$, and depth $0$ by $\widetilde{S}_0(\widetilde{\Gamma}_0(N)).$ As before (cf. \cite{Ha}, in particular Lemma 1.2.2, which works equally well over the metaplectic group), the functions in $\widetilde{S}_0(\widetilde{\Gamma}_0(N))$ are supported on finitely many $\widetilde{\Gamma}_0(N)$-orbits, and so
$\dim\left(\widetilde{S}_0(\widetilde{\Gamma}_0(N))\right)<\infty.$ \\

%\begin{rmk}Since the projection $\widetilde{SL}_2(R)\rightarrow SL_2(R)$ has a section (see Lemma \ref{modularity}), it is enough to consider the action of $SL_2(R)$ in order to define metaplectic functions.
%\end{rmk}

\subsubsection{Depth $1$}

If a function $\phi$ (resp. $F$) on $PGL_2(k_{\infty})$ (resp. $\widetilde{SL}_2(k_{\infty})$) satisfies all the conditions of being automorphic, with respect to the congruence subgroup $\Gamma_0(N)$ for some $N$ (resp. metaplectic) with respect to $\widetilde{\Gamma}_0(N)$), but with the condition of right invariance under $PGL_2(\mathcal{O}_{\infty})$ (resp. $\iota(SL_2(\mathcal{O}_{\infty}))$) weakened
to right invariance under $K_0(\varpi_{\infty})$ (resp. $\iota(\widetilde{K}_0(\varpi_{\infty}))$),  then $\phi$ (resp. $F$) is called an automorphic function of level $N$, depth $1$. (resp. metaplectic function of level $N$, depth $1$) Denote the space of all automorphic functions of level $N$, depth $1$, by $M_1(\Gamma_0(N)).$ Define the spaces $S_1(\Gamma_0(N)), \widetilde{M}_1(\widetilde{\Gamma}_0(N)),\textrm{ and }\widetilde{S}_1(\widetilde{\Gamma}_0(N))$ analogously.

\subsubsection{Fourier expansions}\label{Fourierexpand}

Let $\phi(z)$, on $\mathbb{H}$, and $F(w)$, on $\Tt$, be an automorphic and a metaplectic function of depth $0$ respectively. Since $\phi(z)$ and $F(w)$ are left invariant under the upper triangular groups
$\left\{\left(\begin{smallmatrix} 1 & a\\ 0 & 1\end{smallmatrix}\right)| a\in R\right\},$ and$\left\{\eta\left(\left(\begin{smallmatrix} 1 & a\\ 0 & 1\end{smallmatrix}\right)\right)\mid a\in R\right\} $ we can Fourier expand $\phi(z)$ and $F(w)$ for $z=\left(\begin{smallmatrix} y & x\\ 0 & 1\end{smallmatrix}\right)$ and $w=\left(\left(\begin{smallmatrix} \sqrt{v}& u/\sqrt{v}\\ 0 & 1/\sqrt{v}\end{smallmatrix}\right),1\right)$, as follows:
$$\phi(z)=\sum_{a\in R}e(T^2a   x)\phi_a(y)$$
\[F(w)=\sum_{a\in R}e(T^2au)F_a(v)\]
(Note that the introduction of the $T^2$ factor is because under our normalization of the character $e(x)$, the orthogonal complement of the ring $R$ is $T^2R$.) Since $x$ is only defined up to addition by $y\mathcal{O}_{\infty}$, and $u$ is defined modulo $v\mathcal{O}_{\infty}$, we must have $\phi_a(y) = 0$ unless $v_{\infty}(ay)>1$, and respectively $F_a(v)=0$ unless $v_{\infty}(av)>1$. Moreover, if $\phi(z)$ (or $F(w)$) is a cusp form, then
$\phi_0(y) = 0$ (and $F_0(v)=0$ respectively). By convention, we write $\phi_a(y)=0$ (respectively $F_a(v)=0$) when $a\notin R.$

\subsubsection{Depth $1$}\label{Fourierexpandinfty}

Let $z = \left(\begin{smallmatrix} y & x\\ 0 &1 \end{smallmatrix}\right)$ (or $z=\left(\begin{smallmatrix}x&y\\1&0\end{smallmatrix}\right)$ depending on whether we are on $\mathbb{H}_1^u$ or $\mathbb{H}_1^l$ respectively), and
$w=\left(\left(\begin{smallmatrix} \sqrt{v}& u/\sqrt{v}\\ 0& 1/\sqrt{v}\end{smallmatrix}\right),1\right)$ (or $w=\left(\left(\begin{smallmatrix} u/\sqrt{v} & \sqrt{v}\\ 1/\sqrt{v} & 0\end{smallmatrix}\right),1\right)$ depending on whether we are on $\Tt_1^u$ or $\Tt_1^l$ respectively) denote our chosen coordinates as before, and let $\phi(z)$ be an automorphic function of depth $1$ on $\mathbb{H}_1$, and $F(w)$ a metaplectic function of depth 1 on $\Tt_1$. Then reasoning as above, for both $\phi$ and $F$ we have two Fourier expansions, for the upper and lower components, as follows:

For $z\in\mathbb{H}_1^u$ and $w\in\Tt_1^u$ we have;
$$\phi(z)=\sum_{a\in R}e(T^2ax)\phi^u_a(y)$$
\[F(w)=\sum_{a\in R}e(T^2au)F_a^u(v)\]
For $z\in\mathbb{H}_1^l$ and $w\in\Tt_1^l$ we have;
$$\phi(w)=\sum_{a\in R}e(T^2ax)\phi^l_a(y)$$
\[F(w)=\sum_{a\in R}e(T^2au)F_a^l(v)\]
We reason as before that $\phi^u_a(y)$ and $F_a^u(v)$ are $0$ unless $v_{\infty}(ay)>1$ and $v_{\infty}(av)>1$ respectively, and $\phi^l_a(y)$ and $F_a^l(v)$ are $0$ unless $v_{\infty}(ay)>0$ and $v_{\infty}(av)>0$ respectively. Moreover, if $\phi(z)$, or $F(x)$, is cuspidal then $\phi^l_0(y)=\phi^u_0(y)=0$, and $F_0^u(w)=F_0^l(w)=0$.

\subsection{Petersson inner product}
Given automorphic functions $\phi_1,\phi_2$ of level $N$ and depth $0$, we define the Petersson inner product of $\phi_1$ and $\phi_2$ to be
$$\langle \phi_1,\phi_2\rangle = [\Gamma:\Gamma_0(N)]^{-1}\int_{\Gamma_0(N)\backslash\mathbb{H}} \phi_1(z)\overline{\phi_2(z)}d\mu(z)$$ where $d\mu(z)$ is the left invariant measure as defined in \S\ref{backgr}.

As in the number field case, this is a positive definite inner product with respect to which all the Hecke operators  $T_P$ (defined in the next section) for $(P,N)=1$ are self-adjoint. We define the Petersson inner product for the other spaces of functions that we have introduced analogously, and note that because of the measure normalizations on $\mathbb{H}_1$, the norm of a automorphic (or metaplectic) function is independent of its realization in $\mathbb{H}$ or in $\mathbb{H}_1$, as an old vector (similarly for the metaplectic functions).

\subsection{Whittaker Functions and Hecke Operators}

\subsubsection{Whittaker Functions on $\mathbb{H}$} In this section we define the analogue of the Laplacian on the space $\mathbb{H}$. As there is no
differential structure, we will define it as a convolution operator (the Hecke operator at $\infty$). Specifically, we shall define $\Delta$ to be right convolution with the double coset
$$PGL_2(\mathcal{O}_{\infty})\left(\begin{smallmatrix} \varpi_{\infty} & 0\\ 0 & 1\end{smallmatrix}\right) PGL_2(\mathcal{O}_{\infty})$$
and we normalize it (as in the classical Hecke operators, where weight is taken to be $0$) by dividing by $p$. Given an automorphic function $\phi(z)$ on $\mathbb{H}$ we can Fourier expand it as in \S\ref{Fourierexpand}:
$$\phi(z)=\sum_{a\in R}e(T^2ax)\phi_a(y)$$
Then the action of the Laplacian operator $\Delta$ on $\phi$ is realized by
\begin{equation}\label{Laplacianaction}
\Delta(\phi)(z)=p^{-1}\sum_{a\in R}e(T^2ax)\phi_a(yT)+p^{-1}\sum_{a\in R}\sum_{j\in\mathbb{F}_p}e(T^2a(x+jy))\phi_a\left(yT^{-1}\right)
\end{equation}
We will call automorphic functions $\phi$ that are eigenfunctions of $\Delta$, \emph{automorphic forms}. Now let $\phi$ be an automorphic form with Laplacian eigenvalue $\lambda_{\infty,\phi}$. We are interested in the relation between
the Fourier coefficients of $\phi$ and the Fourier coefficients of $\Delta(\phi)$. \\

\begin{lemma}\label{Laplacianrelation}Let $\chi_{\mathcal{O}_{\infty}}$ denote the characteristic function of $\mathcal{O}_{\infty}$. Then,

\[\lambda_{\infty,\phi}\:\phi_a(y)=p^{-1}\phi_a(yT)+\chi_{\Ooi}(ayT^2)\phi_a(yT^{-1})\]

\end{lemma}

\begin{proof}
By (\ref{Fourierexpand}) we know that $\phi_a(yT^{-1}) = 0$ unless $v_{\infty}(ay)>0$, and in this case  the character sum over $j$ in (\ref{Laplacianaction}) vanishes iff $v_{\infty}(ay)=1.$ The statement now follows directly from (\ref{Laplacianaction}).
\end{proof}

Since, for fixed $a\in R$, $\phi_a(y)$ depends only on $v_{\infty}(y),$ we can determine $\phi_a(y)$ up to a constant using Lemma \ref{Laplacianrelation}.
We denote $\phi_1(T^{-n})$ by $c(n)$. Then the $c(n)$'s satisfy the recursion
\[\lambda_{\infty,\phi}\:c(n)=\frac{1}{p}c(n-1)+\delta_{n\geq 2}c(n+1)\]
and by \S\ref{Fourierexpand},  $c(n)=\phi_1(T^{-n})=0$ when $v_{\infty}(T^{-n})\leq 1$, i.e. for $n<2$. We write $\lambda_{\infty,\phi}$ as
$$\lambda_{\infty,\phi}\:p^{1/2}=e^{i\theta} + e^{-i\theta}$$
where\footnote{It will turn out that $\theta\in\mathbb{R}$ but we do not use it here.} $e^{i\theta}\in\mathbb{C}$. We solve the 2-step recursion to get the following

\begin{lemma}\label{wh} For $n\geq2$
\[c(n)=\frac{p^{\frac{1}{2}}}{e^{i\theta}-e^{-i\theta}}\left[\left(\frac{e^{i\theta}}{p^{\frac{1}{2}}}\right)^{n-1}-\left(\frac{e^{-i\theta}}{p^{\frac{1}{2}}}\right)^{n-1}\right]c(2)\]
and $c(n)=0$ for $n<2$ (We note that this is a special case of the Cassleman-Shalika formula, see \cite{CasSha}.).
\end{lemma}

Before we go on, we shall take a moment to motivate the definitions of Whittaker functions to come. Recall that in the classical theory of automorphic forms one asks for automorphic forms to be eigenfunctions of certain differential operators (Hecke operators at $\infty$ so to speak), which in the case of $GL(2)$ reduces to the Laplacian on the hyperbolic plane. This condition (together with the invariance under the unipotent group) imposes the $\phi_a(y)$'s to satisfy a certain differential equations in the $y$-variable (Bessel's differential equation in the case of $GL(2)$) whose solutions are the Whittaker functions, and we expand our automorphic forms in terms of these. In the case of a function field the role of the Laplacian is ``replaced" (in a sense it was always the ``Hecke operator" that we were interested in) by the Hecke operator defined above (in terms of the underlying graph this is the classical discrete Laplacian, which is averaging over nearest neighbors of vertices), and this operator gives rise to difference equations instead of differential equations, whose solutions, given in Lemma \ref{wh}, are (going to be) the Whittaker functions we will use to expand our automorphic functions. Note that while in the number field case the differential equation we get has two solutions and growth conditions at the cusp separate one solution out, in the function field case we are invariant under a much bigger group (more precisely the initial condition that $c(n)=0$ for $n<2$, described in \S\ref{Fourierexpand}, reduces the space of solutions by one dimension) on the right and so the growth conditions on the solutions follow for free.

By the above paragraph, we define our Whittaker functions (of depth $0$) by
\begin{equation}\label{integralWhittaker}
 W_{0,i\theta}(y) = \begin{cases} 0 & \text{                             if     } v_{\infty}(y) < 2\\
\frac{p^{\frac{1}{2}}}{e^{i\theta}-e^{-i\theta}}\left[\left(\frac{e^{i\theta}}{p^{\frac{1}{2}}}\right)^{n-1}-\left(\frac{e^{-i\theta}}
{p^{\frac{1}{2}}}\right)^{n-1}\right]&  \text{                                 $ n=v_{\infty}(y)\geq2$}\end{cases}
\end{equation}
The above analysis shows that there are constants $\lambda_{\phi}(a)$ such that
$$\phi(z) = \sum_{a\in R} \lambda_\phi(a)e(aT^2x)W_{0,i\theta}(ay)$$
We call $\lambda_{\phi}(a)$ the Fourier-Whittaker coefficients of $\phi.$

\subsubsection{Depth $1$} {\label{whittakerwithdepth}}

To get a nice theory of Whittaker functions on $\mathbb{H}_1$, we first have to restrict to those automorphic functions that first appear in depth 1.
There are two injections $\pi_1,\pi_2$ from $S_0(\Gamma_0(N))$ into $S_1(\Gamma_0(N))$ given by
$$\pi_1(\phi)(z) := \phi(z),\,\,\, \pi_2(\phi)(z) := \phi\left(z\left(\begin{smallmatrix} T & 0\\0 & 1\end{smallmatrix}\right)\right)$$ where we Define $S_1^{old}(\Gamma_0(N))$ to be the space
spanned by the image of both $\pi_1$ and $\pi_2$. We define the space of `new' cuspidal functions, $S_1^{new}(\Gamma_0(N))$, to be the orthogonal complement to
$S_1^{old}(\Gamma_0(N))$ with respect to the Petersson inner product. We mention that in the language of representation theory, these correspond to the representations
of depth 1 at the infinite place.

We also have a trace operator (which corresponds to summing over the fibers of the covering map $\mathbb{H}_1\xrightarrow{\rho} \mathbb{H}$) given by $$\tr: S_1(\Gamma_0(N))\rightarrow
S_0(\Gamma_0(N))$$
 $$\tr(\phi)(z_0):=\sum_{\substack{z\in\mathbb{H}_1\\\rho(z)=z_0}}\phi(z) $$
The trace operator is the adjoint to the inclusion $\pi_1$ and therefore annihilates $S_1^{new}(\Gamma_0(N))$. The inverse image of
$\left(\begin{smallmatrix} y & x\\ 0 & 1 \end{smallmatrix}\right)\in\mathbb{H}$ under the restriction map is
$$\left\{\left(\begin{smallmatrix} y & x\\ 0 & 1 \end{smallmatrix}\right)\right\}\cup\begin{setdef}{\left(\begin{smallmatrix} x+jy & y\\ 1 & 0 \end{smallmatrix}\right)}{j\in\mathbb{F}_p}\end{setdef}$$
Then since $\phi\in S_1^{new}(\Gamma_0(N))\Rightarrow\tr(\phi)=0$, we get the following relation between the ``upper" and ``lower" Fourier coefficients of an automorphic function $\phi(z)\in S_1^{new}(\Gamma_0(N))$;
\begin{equation}\label{genuineupperlower}
\phi_a^u(y)=-p\phi_a^l(y)\chi_{\Ooi}(aT^2y)\,\,\,\,\,\,\,\,\forall a\in R,\,y\in k_{\infty}^{\times}/\mathcal{O}_{\infty}^{\times}
\end{equation}

Taking the place of the Laplacian in this context is the Atkin-Lehner involution at infinity $W_{\infty}$, which we define to be

$$W_{\infty}(\phi)(z) := \phi\left(z\left(\begin{smallmatrix} 0 & 1\\ T^{-1} & 0\end{smallmatrix}\right)\right).$$

Note that $K_0(\varpi_{\infty})\left(z\left(\begin{smallmatrix} 0 & 1\\ T^{-1} & 0\end{smallmatrix}\right)\right)K_0(\varpi_{\infty})=\left(z\left(\begin{smallmatrix} 0 & 1\\ T^{-1} & 0\end{smallmatrix}\right)\right)K_0(\varpi_{\infty})$ so that $W_{\infty}$ can still be considered as convolution with a double coset.

Since $\left(\begin{smallmatrix} 0 & 1\\ T^{-1} & 0\end{smallmatrix}\right)^2$ is the identity, $W_{\infty}$ is indeed an involution. Moreover,
$\pi_1W_{\infty}=W_{\infty}\pi_2$ and $W_{\infty}$ is self-adjoint for the Petersson inner product, so $W_{\infty}$ preserves $S_1^{new}(\Gamma_0(N))$. We define a cuspidal form of depth 1 and level $N$ to be a new cuspidal function of level $N$ which is furthermore an eigenfunction of $W_{\infty}$.
Consider an automorphic form $\phi(z)$ of depth 1 with eigenvalue $w_{\phi,\infty}\in\{\pm1\}$ under $W_{\infty}$. By considering Fourier expansions we deduce
that $$w_{\infty,\phi}\phi_a^l(Ty)=\phi^u_a(y)$$ Combining this with (\ref{genuineupperlower}) we can solve the recursion for $\phi_a^u(y)$ and get
$$\phi_a^u(y)=\chi_{\Ooi}(aT^2y)\phi_a^u(T^{-2})(\lambda_{\infty,\phi})^{v_{\infty}(y)-2}$$ where $\lambda_{\infty,\phi}=-\frac{w_{\infty,\phi}}{p}$.
We thus define our Whittaker functions of depth $1$ by
\begin{equation}\label{integralWhittakerinfty}
W_{1,\lambda_{\infty,\phi}}(y) = \delta_{v_{\infty}(y)>1}\cdot\lambda_{\infty,\phi}^{v_{\infty}(y)-2}
\end{equation}
By the above analysis, we see that there are constants $\lambda^u_{\phi}(a)$ such that for $z\in\mathbb{H}_1^u$,
$$\phi(z) = \displaystyle\sum_{a\in R}\lambda^u_{\phi}(a)e(T^2ax)W_{1,\lambda_{\infty,\phi}}(ay)$$
As in the previous section we will refer to the $\lambda^u_{\phi}(a)$ as the Fourier-Whittaker coefficients of $\phi.$

\subsubsection{Hecke Operators on $\mathbb{H}$} Given a monic polynomial $P\in R$ such that $P$ and $N$ are relatively prime,  we associate to it a Hecke
operator $T_P$ that acts on $S_0(\Gamma_0(N))$ by the following formula
\[T_{P}(\phi)(z)=|P|^{-1}\sum_{\substack{GH=P\\J \bmod{H}}}\phi\left(\left(\begin{smallmatrix}G & J\\ 0 & H\end{smallmatrix}\right)\left(\begin{smallmatrix}y & x\\ 0 & 1\end{smallmatrix}\right)\right)\]
where $G,H$ are monic, and $z=\left(\begin{smallmatrix}y & x\\ 0 & 1\end{smallmatrix}\right)$ as usual.

As in the number field case these operators constitute a commuting family of normal operators, so they can be simultaneously diagonalised.
We will be interested in automorphic forms that are common eigenfunctions of all the Hecke operators. An automorphic form $\phi\in S_0(\Gamma_0(N))$ that is a common eigenfunction of all Hecke operators $T_P$ with $(P,N)=1$ will be called a cuspidal \emph{Hecke eigenform}.

Let $\phi$ be a Hecke eigenform and $P\in R$ a monic irreducible polynomial. Denote the
eigenvalue of $T_P$ on $\phi$ by $\lambda_{P,\phi}$. Then we have
\[\lambda_{P,\phi}\phi(z)=T_P(\phi)(z)=|P|^{-1}\sum_{\gamma} \phi(\gamma z)\]

where the summation is over $\gamma\in\left\{\left(\begin{smallmatrix}P & 0\\ 0 & 1\end{smallmatrix}\right),\left(\begin{smallmatrix}1 & H\\ 0
& P\end{smallmatrix}\right)\right\}$ with $H\in R/PR$, and $z=\left(\begin{smallmatrix}y&x\\0&1\end{smallmatrix}\right)$. Plugging this formula in the Fourier expansion of $\phi$ gives

\begin{equation}\label{integralHeckeaction}
|P|\lambda_{P,\phi}\sum_{a\in R}e(T^2ax)\phi_a(y)=\sum_{\substack{a\in
R\\H\bmod P}}e\left(T^2a\left(\frac{x+H}{P}\right)\right)\phi_a\left(\frac{y}{P}\right)+\sum_{a\in R}e(T^2aPx)\phi_a(Py)
\end{equation}
Note that when $a\notin PR$ in the first sum on the right hand side of the equality in (\ref{integralHeckeaction}), the coefficient for $a$ vanishes, since in that case $e(\cdot T^2H/P)$ is a \emph{non-trivial} character. Taking this into account and equating the $e(T^2ax)-$Fourier coefficients in (\ref{integralHeckeaction}) we end up with
\begin{equation}\label{integralHeckerelation}
\lambda_{P,\phi}\phi_a(y)=\phi_{aP}\left(\frac{y}{P}\right)+|P|^{-1}\phi_{\frac{a}{P}}(Py)
\end{equation}
Note that specializing to $a=1$, this relation implies that $\phi_P(y)=\lambda_{P,\phi}\phi_1(yP)$.  Note also that for a Hecke eigenform $\phi$ of level $N$, depth $0$, for $a\in R$ if we denote the $a$'th Fourier-Whittaker coefficient of $\phi$ by $\lambda_{\phi}(a)$ then for any $P$ relatively prime to $aN$, relation (\ref{integralHeckerelation}) implies that $\lambda_{\phi}(aP)=\lambda_{P,\phi}\lambda_{\phi}(a)$.

\subsubsection{Depth $1$}\label{integralheckedepth1}

Now suppose that $\phi(z)\in S_1(\Gamma_0(N))$ for some $N\in R$. Then we define the Hecke operators $T_P$ in the same way as before:
\[T_{P}(\phi)(g)=|P|^{-1}\sum_{\substack{GH=P\\Q\bmod H}}\phi\left(\left(\begin{smallmatrix}G & Q\\ 0 & H\end{smallmatrix}\right)g\right)\]
where $G,H$ are monic. If $\phi(z)$ is an eigenfunction of $T_P$ with eigenvalue $\lambda_{P,\phi},$ then in the notation of \S\ref{Fourierexpandinfty} this translates to:
\[\lambda_{P,\phi}\phi^u_a(y)=\phi^u_{aP}\left(\frac{y}{P}\right)+|P|^{-1}\phi^u_{\frac{a}{P}}(Py)\]
and
\[\lambda_{P,\phi}\phi^l_a(y)=\phi^l_{aP}\left(\frac{y}{P}\right)+|P|^{-1}\phi^l_{\frac{a}{P}}(Py)\]

Specializing to $a=1$ we again get $\phi_P^{\cdot}(y)=\lambda_{P,\phi}^{\cdot}\phi_1^{\cdot}(Py)$, for $\cdot=u$ or $l$. As in the previous section an automorphic form $\phi\in S_1(\Gamma_0(N))$ that is a common eigenfunction of all the Hecke operators $T_P$ with $(P,N)=1$ will be called a cuspidal Hecke eigenform of depth $1$.

\subsubsection{Whittaker functions on $\Tt$}\label{halfintegralwhittakerun}

%To define the appropriate analogue of the Laplacian for metaplectic functions, we need to recall how to consider these functions as functions on the metaplectic group $\widetilde{SL}_2(k_{\infty})$:

%Given $F:\tilde{H}\rightarrow \C$ we first lift to an function $F_0:SL_2(k_{\infty})\rightarrow \C$ which is right invariant under $SL_2(\Ooi)$. Now, as in $\S\ref{metsp}$ we define the function $\widetilde{F}:\widetilde{SL}_2(k_{\infty})\rightarrow\C$ such that $$\widetilde{F}(\iota(g)=F_0(g),\forall g\in SL_2(k_{\infty})$$ and $$\widetilde{F}(g\cdot(id,-1))=-\widetilde{F}(g),\forall g\in\widetilde{SL}_2(k_{\infty}).$$ The map $F\rightarrow\widetilde{F}$ is a bijection between functions on $\tilde{H}$ and functions on $\widetilde{SL}_2(k_{\infty})$ right invariant
%under $\iota(SL_2(\Ooi))$ which transform by $-1$ under $g\rightarrow (1,-1)\cdot g$.

For an $F$ that is a metaplectic function of level $N$, depth $0$, we define the Laplacian $\widetilde{\Delta}$ on $\widetilde{M}_{0}(\widetilde{\Gamma}_0(N))$ to be right convolution of $F$ with the characteristic function of
the double coset $\iota\left({SL}_2(\Ooi)\right)\left(\left(\begin{smallmatrix}\varpi_{\infty}&0\\0&\varpi_{\infty}^{-1}\end{smallmatrix}\right),1\right)\iota\left(SL_2(\Ooi)\right).$ A set of right coset representatives can be
computed to be
$$\begin{cases}
\alpha_b = \left(\left(\begin{smallmatrix}\varpi_{\infty}&\varpi_{\infty}^{-1}b\\0&\varpi_{\infty}^{-1}\end{smallmatrix}\right),1\right) & b\in \Ooi/\varpi_{\infty}^2\Ooi\\
\beta_h = \left(\left(\begin{smallmatrix}1&\varpi_{\infty}^{-1}h\\0&1\end{smallmatrix}\right), (h,\varpi_{\infty})_{\infty}\right) & h\in (\Ooi/\varpi_{\infty}\Ooi)^{\times}\\
\sigma = \left(\left(\begin{smallmatrix}\varpi_{\infty}^{-1}&0\\0&\varpi_{\infty}\end{smallmatrix}\right), 1\right)
\end{cases}$$

To see the above, note that $$\alpha_b = \iota\left(\left(\begin{matrix}1&b\\0&1\end{matrix}\right)\right)\cdot\left(\left(\begin{matrix}\varpi_{\infty}&0\\0&\varpi_{\infty}^{-1}\end{matrix}\right),1\right),$$
$$\sigma = \iota\left(\left(\begin{smallmatrix}0&1\\-1&0\end{smallmatrix}\right)\right)\cdot\left(\left(\begin{smallmatrix}\varpi_{\infty}&0\\0&\varpi_{\infty}^{-1}\end{smallmatrix}\right),1\right)\cdot
\iota\left(\left(\begin{smallmatrix}0&-1\\1&0\end{smallmatrix}\right)\right),$$ and
$$ \beta_h = \iota\left(\left(\begin{smallmatrix}0&h\\-h^{-1}&\varpi_{\infty}\end{smallmatrix}\right)\right)\cdot
\left(\left(\begin{smallmatrix}\varpi_{\infty}&0\\0&\varpi_{\infty}^{-1}\end{smallmatrix}\right),1\right)
\cdot\iota\left(\left(\begin{smallmatrix}1&0\\h^{-1}\varpi_{\infty}&1\end{smallmatrix}\right)\right).$$

The sign in $\beta_h$ comes from the fact that
\begin{align*}
\epsilon\left(\left(\begin{smallmatrix}0&h\\-h^{-1}&\varpi_{\infty}\end{smallmatrix}\right),\left(\begin{smallmatrix}\varpi_{\infty}&0\\0&\varpi_{\infty}^{-1}\end{smallmatrix}\right)\right)
&=(-h^{-1}\varpi_{\infty},\varpi_{\infty}^{-1})_{\infty}(-h^{-1},-h^{-1}\varpi_{\infty})_{\infty}(-h^{-1},\varpi_{\infty}^{-1})_{\infty}\\
&=(h,\varpi_{\infty})_{\infty}
\end{align*}

Let $F\in \widetilde{M}(\widetilde{\Gamma}_0(N))$. Carrying out the convolution we compute the action on the Fourier expansion of $F$ to be
\begin{equation}\label{halfweightLaplacianaction}
%\begin{align*}
%\widetilde{\Delta}(F)(w)&=(\varpi_{\infty},\sqrt{v})_{\infty}\sum_{b\in \Ooi/T^{-2}\Ooi}\sum_{a\in R}e(T^2a(u+vb))F_a(T^{-2}v)\\
%&+\sum_{h\in (\Ooi/T^{-1}\Ooi)^{\times}}\sum_{a\in R}(h,\varpi_{\infty})_{\infty}e(T^2a(u+hTv))F_a(v)\\
%&+(\varpi_{\infty},\sqrt{v})_{\infty}\sum_{a\in R}e(T^2au)F_a(T^2v)\\
%\end{align*}
\end{equation}
A metaplectic function $F$ that is an eigenfunction of $\widetilde{\Delta}$ is called a \emph{metaplectic form}. Now suppose $F$ is a metaplectic form with eigenvalue $\lambda_{\infty,F}.$ In computing the Fourier coefficient of $e(T^2au)$ the second sum on the RHS of \eqref{halfweightLaplacianaction} gives the sum
\begin{equation}\label{fouriergauss}
\sum_{h\in (\Ooi/T^{-1}\Ooi)^{\times}} (h,\varpi_{\infty})_{\infty}e(h\cdot T^3av)F_a(v)
\end{equation}

$F_a(v)$ vanishes unless $v_{\infty}(av)\geq 2$, and if $v_{\infty}(av)>2$ then the character $e(h\cdot T^3av)$ vanishes and hence we're left with $\sum_{h\in (\Ooi/T^{-1}\Ooi)^{\times}} (h,\varpi_{\infty})_{\infty}=0$. If $v_{\infty}(av)=2$ then \eqref{fouriergauss} becomes a Gauss sum which determined by the square class of $av$. Since $v\in (k_{\infty}^{\times})^2$ this is the same as the square class of $a$. Hence \eqref{fouriergauss} gives $F_a(v)\cdot p^{1/2}\delta_{v_{\infty}(av)=2}\cdot(a,\varpi_{\infty})_{\infty}$.

Equating Fourier coefficients we arrive at
\begin{equation}\label{halfweightLaplacianrelation}
\lambda_{\infty,F}F_a(v) = (\varpi_{\infty},\sqrt{v})_{\infty}p^2F_{a}(T^{-2}v)\chi_{\Ooi}(T^2av) + p^{1/2}F_a(v)\delta_{v_{\infty}(av)=2}(a,\varpi_{\infty})_{\infty} + (\varpi_{\infty},\sqrt{v})_{\infty}F_a(T^2v)
\end{equation}
Let $e^{i\gamma}$ be a complex number such that
$$e^{i\gamma} + e^{-i\gamma} = p^{-1}\lambda_{\infty,F}$$ Taking $\xi\in\{1,\epsilon,T^{-1},\epsilon T^{-1}\}$, where $\epsilon\in\mathbb{F}_p$ is a non-square, we define the metaplectic Whittaker function (of depth $0$) $\widetilde{W}_{0,i\gamma}(\xi v)$ to be:
\begin{multline*}\label{halfweightWhittaker}
\widetilde{W}_{0,i\gamma}(\xi v)=\\
\begin{cases}
\frac{p(\varpi_{\infty},\sqrt{v})_{\infty}^n}{e^{i\gamma}-e^{-i\gamma}}\left[\left(1-\frac{\delta_{v_{\infty}(\xi)=0}(\xi,\varpi_{\infty})_{\infty}e^{-i\gamma}}{\sqrt{p}}\right)\left(e^{i\gamma}p^{-1}\right)^n-
\left(1-\frac{\delta_{v_{\infty}(\xi)=0}(\xi,\varpi_{\infty})_{\infty}e^{i\gamma}}{\sqrt{p}}\right)\left(e^{-i\gamma}p^{-1}\right)^n\right] & v_{\infty}(\sqrt{v})=n\geq 1\\
0 & \textrm{else}
\end{cases}
\end{multline*}

Solving the recursion \eqref{halfweightLaplacianrelation} starting from the fact that $F_a(v)=0$ if $v_{\infty}(av)<2$, we see that there are constants $\lambda_F(a)$ such that $F_a(v) =
\widetilde{W}_{0,i\gamma}(av)(a,v^{1/2})_{\infty}\lambda_F(a).$ We call these the Fourier-Whittaker coefficients of $F.$

\subsubsection{Depth $1$}\label{halfintegralwhittakerram}

We begin by defining the analogue of the Laplacian in this context. Let $\widetilde{W}_{\infty}$ be right convolution with the characteristic function of the double coset

\[\iota(\widetilde{K}_0(\varpi_{\infty}))\left(\left(\begin{smallmatrix}0&-\varpi_{\infty}^{-1}\\\varpi_{\infty}&0\end{smallmatrix}\right),1\right)\iota(\widetilde{K}_0(\varpi_{\infty}))\]
A set of right coset representatives for the above is given by
\[\begin{setdef}{\left(\left(\begin{smallmatrix}0&-\varpi_{\infty}^{-1}\\ \varpi_{\infty}&i\end{smallmatrix}\right),1\right)}{i\in\mathbb{F}_p}\end{setdef}\]
The effect of $\widetilde{W}_{\infty}$ on the upper and lower coordinates is given by
\[\left(\left(\begin{smallmatrix}\sqrt{v}&u/\sqrt{v}\\0&1/\sqrt{v}\end{smallmatrix}\right),1\right)\longrightarrow
\begin{cases}\left(\left(\begin{smallmatrix}\sqrt{v}&u/\sqrt{v}-\sqrt{v}\varpi_{\infty}^{-1}i^{-1}\\0&\frac{1}{\sqrt{v}}\end{smallmatrix}\right),\left(\frac{i}{T}\right)\right)&\text{if $i\neq0$}\\ \left(\left(\begin{smallmatrix}u\varpi_{\infty}/\sqrt{v}&-\varpi_{\infty}^{-1}\sqrt{v}\\ \varpi_{\infty}/\sqrt{v}&0\end{smallmatrix}\right),(\varpi_{\infty},\sqrt{v})_{\infty}\right)&\text{if $i=0$}\end{cases}\]
\[\left(\left(\begin{smallmatrix}u/\sqrt{v}&\sqrt{v}\\ -1/\sqrt{v}&0\end{smallmatrix}\right),1\right)\longrightarrow\left(\left(\begin{smallmatrix}\varpi_{\infty}\sqrt{v}&-u\varpi_{\infty}^{-1}/\sqrt{v}+i\sqrt{v}\\0&\varpi_{\infty}^{-1}/\sqrt{v}\end{smallmatrix}\right),
(\varpi_{\infty},\sqrt{v})_{\infty}\right)\]
If $F$ is an eigenfunction of $\widetilde{W}_{\infty}$ with eigenvalue $\widetilde{w}_{\infty,F}$ we see that
\begin{equation}
\widetilde{w}_{\infty,F}F_a^l(v)=(\varpi_{\infty},\sqrt{v})_{\infty}\sum_{i\in\mathbb{F}_p}e(Taiv)F_a^u(T^{-2}v)
\end{equation}
and
\begin{equation}\label{laplacianeq}
\widetilde{w}_{\infty,F}F_a^u(v)=\sum_{i\in\mathbb{F}_p^{\times}}\left(\frac{i}{T}\right)F_a^u(v)e(T^3aiv)+(\varpi_{\infty},\sqrt{v})_{\infty}F_a^l(T^2v)
\end{equation}
Note that by Section \ref{Fourierexpandinfty} we know that $F_a^u(v)=0$ for $v_{\infty}(av)<2$ and $F^l_a(v)=0$ for $v_{\infty}(av)<1$. Using these and the relations above we find that the eigenvalue $
\widetilde{w}_{\infty,F}$   satisfies
\begin{equation}\label{laplacian}
\widetilde{w}_{\infty,F}^2=p
\end{equation}
Using these representatives it is an easy check to see that $\widetilde{W}_{\infty}$ is self adjoint for the Petersson inner product. As before, there is an inclusion $$\pi_1:\widetilde{S}_0(\widetilde{\Gamma}_0(N),1)\rightarrow \widetilde{S}_1(\widetilde{\Gamma}_0(N),1)$$ given by $\pi_1(F)(g):=F(g)$.
We define $\widetilde{S}^{old}_1(\widetilde{\Gamma}_0(N))\subset \widetilde{S}_1(\widetilde{\Gamma}_0(N))$ to be the space spanned by the image of $\pi_1$ and $\widetilde{W}_{\infty}\circ\pi_1$. We denote the orthogonal complement of $\widetilde{S}^{old}_1(\widetilde{\Gamma}_0(N))$ by $\widetilde{S}_1^{new}(\widetilde{\Gamma}_0(N))$. Each element of $\widetilde{S}_1^{new}(\widetilde{\Gamma}_0(N))$ is called a new cuspidal metaplectic function of level $N$. Since $\widetilde{W}_{\infty}$ is self adjoint,
it preserves the space of new cuspidal metaplectic functions of level $N$.

We call a function $F(z)\in \widetilde{S}^{new}_1(\widetilde{\Gamma}_0(N))$ which is an eigenfunction
of $\widetilde{W}_{\infty}$ a new metaplectic form of level $N$ . As before, we have the trace operator $\tr$ which is adjoint to $\pi_1$, which annihilates
$\widetilde{S}^{new}_1(\widetilde{\Gamma}_0(N))$, and thus for a new cuspidal metaplectic form $F(z)$ we have $$F_a^u(v)=-p\:F_a^l(v)\chi_{\Ooi}(aT^2v).$$

 Combining this with equation (\ref{laplacianeq}) we deduce that $$-\frac{\widetilde{w}_{\infty,F}}{p^2}F_a^u(T^2v)=(\varpi_{\infty},\sqrt{v})_{\infty}\chi_{\Ooi}(av)F_a^u(v)$$
For $\xi\in\{1,\epsilon,T^{-1},\epsilon T^{-1}\}$, where $\epsilon\in\mathbb{F}_p$ is a non-square, we define  Whittaker functions of depth $1$ by

$$\widetilde{W}_{1,\lambda_{\infty,F}}(\xi v)=\chi_{\Ooi}(T^2\xi v)\lambda_{\infty,F}^{v_{\infty}(T v^{1/2})}(\varpi_{\infty},\sqrt{v})_{\infty}^{v_{\infty}(v^{1/2})}$$ where $\lambda_{\infty,F}=-\frac{\widetilde{w}_{\infty,F}}{p^2}$.
We define the Fourier-Whittaker coefficients $\lambda_F(a)$ to be such that $F^u_a(v)=\lambda_F(a)(a,v^{1/2})_{\infty} \widetilde{W}_{1,\lambda_{\infty,F}}(av)$.

\subsubsection{Hecke operators on $\Tt$}

Recall the embedding $\eta:SL_2(R)\hookrightarrow\widetilde{SL}_2(k)$ given by
$$\eta\left(\left(\begin{smallmatrix} a&b\\c&d\end{smallmatrix}\right)\right) = \begin{cases}\left(\left(\begin{smallmatrix} a&b\\c&d\end{smallmatrix}\right),\left(\tfrac{d}{c}\right)\right)&\text{if $c\neq0$}\\ \left(\left(\begin{smallmatrix}a&b\\ c&d\end{smallmatrix}\right),1\right)&\text{if $c=0$}\end{cases}.$$

Let $P$ be a monic irreducible  polynomial that
is relatively prime to $N.$ We define the $P^2-$ metaplectic Hecke operator, $\widetilde{T}_{P^2}$, to be the \emph{left convolution} of $F$ with the characteristic function of the double coset
$$\eta(\Gamma_0(N))\left(\left(\begin{smallmatrix}P^{-1}&0\\0&P\end{smallmatrix}\right);1\right)\eta(\Gamma_0(N)).$$
 A set of left coset representatives for the double coset is given by:
\[\begin{cases}
\alpha_b = \left(\left(\begin{smallmatrix}P^{-1}&P^{-1}b\\0&P\end{smallmatrix}\right),1\right) & b\in R/P^2R\\
\beta_h = \left(\left(\begin{smallmatrix}1&P^{-1}h\\0&1\end{smallmatrix}\right), \left(\frac{h}{P}\right)\right) & h\in (R/PR)^{\times}\\
\sigma = \left(\left(\begin{smallmatrix}P&0\\0&P^{-1}\end{smallmatrix}\right), 1\right)
\end{cases}\]

We indicate the hardest case, $\beta_h$: Pick $d,\bar{h}\in R$ such $Pd-h\bar{h}=1$ and $N|h$. Then:

$$ \beta_h = \eta\left(\left(\begin{smallmatrix}1&0\\-P\bar{h}&1\end{smallmatrix}\right)\right)\cdot\left(\left(\begin{smallmatrix}P^{-1}&0\\0&P\end{smallmatrix}\right),1\right)
\cdot\eta\left(\left(\begin{smallmatrix}P&h\\\bar{h}&d\end{smallmatrix}\right)\right).$$

Carrying out the product using the cocycle we get that the sign of $\beta_h$ is $(h,P)_{\infty}\cdot\left(\frac{P}{h}\right)$ which is $\left(\frac{h}{P}\right)$ by lemma \ref{quadrec}.

The effect of $\widetilde{T}_{P^2}$ on the Fourier coefficients of a metaplectic form $F$ of level $N$, depth $0$ is:
\begin{align*}\label{halfweightHeckeaction}
\widetilde{T}_{P^2}(F)(w)&=(P,\sqrt{v})_{\infty}\sum_{b\in R/P^2R}\sum_{a\in R}e(T^2a(P^{-2}u+bP^{-2}))F_a(P^{-2}v)\\
&+\sum_{h\in (R/PR)^{\times}}\sum_{a\in R}\left(\frac{h}{P}\right)e(T^2a(u+hP^{-1}))F_a(v)\\
&+(P,\sqrt{v})_{\infty}\sum_{a\in R}e(T^2aP^2u))F_a(P^{2}v)\\
\end{align*}
In evaluating the above we come up against the following Gauss sum modulo $P$:
\[G_a(P)= \sum_{h\in R/PR}\left(\frac{h}{P}\right)e\left(\frac{haT^2}{P}\right)\]
If $P|a$ then $G_a(P)=\sum_{h\in R}\left(\frac{h}{P}\right) = 0$. Else, $G_a(P)$ is a Gauss sum and depends on $a$ through the quadratic character $\left(\frac{a}{P}\right)$.
Thus we can write $G_a(P)=\delta_{P\nmid a}\left(\frac{a}{P}\right)G_1(P)$.

We will call a metaplectic form $F\in \widetilde{M}_0(\widetilde{\Gamma}_0(N))$ that is an eigenfunction of all the metaplectic Hecke operaors $\widetilde{T}_{P^2}$ for $(P,N)=1$, a \emph{metaplectic Hecke eigenform}. Now suppose $F$ is an eigenfunction of $\widetilde{T}_{P^2}$ with eigenvalue $\lambda_{P^2,F}.$ Equating Fourier coefficients, we get
\begin{equation}\label{halfweightHeckerelation}
\lambda_{P^2,F}F_a(v) = (P,\sqrt{v})_{\infty}|P|^2F_{aP^2}(P^{-2}v) + G_1(P)\delta_{P\nmid a}\left(\frac{a}{P}\right)F_a(v) + (P,\sqrt{v})_{\infty}F_{aP^{-2}}(P^2v)
\end{equation}
Observe that if $a$ is relatively prime to $P$ then $$\lambda_{P^2,F}F_a(v) = (P,\sqrt{v})_{\infty}|P|^2F_{aP^2}(P^{-2}v) + G_1(P)\left(\frac{a}{P}\right)F_a(v)$$

\subsubsection{Depth $1$} We add this section for completeness as everything works the same as \S\ref{integralheckedepth1}. Hecke operators are defined in the same way as above and their action is now calculated on the ``upper" and ``lower" Fourier coefficients separately. For $F\in\widetilde{M}_1(\widetilde{\Gamma}_0(N))$ that is an eigenfunction of $\widetilde{T}_{P^2}$ with eigenvalue $\lambda_{P^2,F}$, the formula reads
\[\lambda_{P^2,F}F^u_a(v) = (P,\sqrt{v})_{\infty}|P|^2F^u_{aP^2}(P^{-2}v) +  G_1(P)\delta_{P\nmid a}\left(\frac{a}{P}\right)F^u_a(v) + F^u_{aP^{-2}}(P^2v)\]
and
\[\lambda_{P^2,F}F^l_a(v) = (P,\sqrt{v})_{\infty}P|^2F^l_{aP^2}(P^{-2}v) +  G_1(P)\delta_{P\nmid a}\left(\frac{a}{P}\right) + F^l_{aP^{-2}}(P^2v)\]
As in depth $0$ case, a metaplectic form $F\in\widetilde{M}_1(\widetilde{\Gamma}_0(N),1)$ that is a common eigenfunction of all the metaplectic Hecke operators $\widetilde{T}_{P^2}$ for $(P,N)=1$ will be called a metaplectic Hecke eigenform of depth $1$.

\subsection{Non-monic Fourier Coefficients}\label{dnotsquare}

We would like to only deal with Fourier coefficients $\lambda_F(D)$ where $D$ is a \emph{monic} polynomial. Every $D\in R$ can be decomposed as
$D=u_D D_0$ where $D_0$ is monic and $u_D\in\mathbb{F}_p^{\times}$. If $u_D = \alpha^2$ for some $\alpha\in\mathbb{F}_p^{\times}$, then since
$$\left(\left(\begin{smallmatrix}\alpha&0\\0&\alpha^{-1}\end{smallmatrix}\right),1\right)\left(\left(\begin{smallmatrix}y&x\\0&1\end{smallmatrix}\right),1\right)\left(\left(\begin{smallmatrix}\alpha^{-1}&0\\0&\alpha\end{smallmatrix}\right),1\right)=
\left(\left(\begin{smallmatrix}y&u_Dx\\0&1\end{smallmatrix}\right),1\right)$$
$\lambda_{F}(D)=\lambda_{F}(D_0)$. Now, let $u_D\in\mathbb{F}_p^{\times}$ be a non-square. We can extend our cocycle $\epsilon(g,h)$ to $g,h\in GL_2(k_{\infty})$ by making the following definitions (See \cite{G},pp 15-16).
\begin{enumerate}
\item $p(g):= \left(\begin{smallmatrix} 1 & 0\\ 0 & det(g)\end{smallmatrix} \right) \cdot g$
\item For $y\in k_{\infty},\, g\in GL_2(k_{\infty}), g^y:=\left(\begin{smallmatrix} 1 & 0\\ 0 & y\end{smallmatrix} \right)^{-1}g\left(\begin{smallmatrix} 1 & 0\\ 0 & y\end{smallmatrix} \right)$.
\item For $y\in k_{\infty},g=\left(\begin{smallmatrix} a & b\\ c & d\end{smallmatrix} \right)\in GL_2(k_{\infty})$ define $v(y,g):=\begin{cases} 1 & c\neq 0\\ (y,d)_{\infty} & else\end{cases}$.
\item For $g,h\in GL_2(k_{\infty})$ we extend $\epsilon$ by $$\epsilon(g,h):=\epsilon(p(g)^{\det(h)},p(h))\cdot v(\det(h),g).$$
\end{enumerate}

Moreover, the cocycle has splittings over $GL_2(\mathcal{O}_{\infty})$ and $GL_2(R)$, extending $\iota$ and $\eta$, defined as follows: For $g=\left(\begin{smallmatrix} a & b\\ c & d\end{smallmatrix} \right)$, we set
$$\iota(g)=\begin{cases} (g,1) & c=0\textrm{ or } c\in\mathcal{O}_{\infty}^{\times}\\ (g,(c,d\cdot \det(g))_{\infty}) & \textrm{else} \end{cases}$$
and for $g\in GL_2(R)$, $$\eta(g)=\begin{cases}(g,\prod_{v|c} (c,d\cdot \det(g))_v) = \left(g,\left(\frac{d\cdot \det(g)}{c}\right)\right)&\text{if $c\neq0$}\\ (g,1)&\text{if $c=0$}\end{cases}.$$

Define $d_{u_D}:=\left(\left(\begin{smallmatrix} 1 & 0\\ 0& u_D\end{smallmatrix} \right),1\right)$. For an element $\tilde{g}\in\widetilde{SL}_2(k_{\infty})$ define $\tilde{g}_{u_D}:= d_{u_D}\tilde{g}d_{u_D}^{-1}$. Given an element $F\in\widetilde{M}_1(\widetilde{\Gamma}_0(N))$ we define $F_{u_D}$ via
$$F_{u_D}(\tilde{g}) := F(\tilde{g}_{u_D}).$$ Since $d_{u_D}\in \iota(GL_2(k_{\infty}))$ it follows that $F_{u_D}$ is $\iota(SL_2(\Ooi))$ invariant on the right, and since $d_{u_D}\in\eta(GL_2(R))$ it follows that $F_{u_D}$ is $\eta(SL_2(R))$ invariant on the left. Thus, we have the following\\
\begin{lemma}\label{nonsquarecoef}
$F_{u_D}\in\widetilde{M}_1(\widetilde{\Gamma}_0(N))$.
\end{lemma}

Computing products as defined above, it is easy to verify that if $\tilde{g}=\left(\left(\begin{smallmatrix} \sqrt{v} & u/\sqrt{v}\\ 0& 1/\sqrt{v}\end{smallmatrix}\right),1\right)$ then
$\tilde{g}_{u_D}=\tilde{g}\cdot(1,(-1)^{v_{\infty}(v)/2}).$

Since the map $F\rightarrow F_{u_D}$ anti-commutes with $\Delta_{\infty}$, it is easy to see that if $F$ is a cuspidal metaplectic form then so is $F_{u_D}$, and moreover $\lambda_F(D) = (-1)^{\lfloor deg(a)/2\rfloor}\lambda_{F_{u_D}}(u_D D)$.  We can therefore restrict ourselves to studying monic Fourier coefficients.

\subsection{Atkin-Lehner Operators}

\subsubsection{Atkin-Lehner Operators on $\mathbb{H}$} Let $\phi$ be a  Hecke eigenform in $M_0(\Gamma_0(N))$. For a prime $l$ (i.e. monic irreducible) such that $\ell^{\alpha}||N$, in analogy with the number field case we define the Atkin-Lehner involution $W_{\ell^{\alpha}}$ to be the following matrix
\[W_{\ell^{\alpha}}=\left(\begin{smallmatrix}\ell^{\alpha}&a\\N&b\ell^{\alpha}\end{smallmatrix}\right)\]
where $a,b\in R$ and $\ell^{2\alpha}b-aN=\ell^{\alpha}$. Then $W_{\ell^{\alpha}}\Gamma_0(N)W_{\ell^{\alpha}}=\Gamma_0(N)$ and $W_{\ell^{\alpha}}^2\in\Gamma_0(N)$. An important property of these operators it that $W_{\ell^{\alpha}}$ commutes with Hecke operators $T_P$ for $(P,N)=1$. Therefore if $\phi$ is also an eigenfunction of $W_{\ell^{\alpha}}$, then denoting the eigenvalue of $\phi$ by $w_{\ell^{\alpha},\phi}$ we have
\[\phi(W_{\ell^{\alpha}}z)=w_{\ell^{\alpha},\phi}\:\phi(z)\]
with $w_{\ell^{\alpha},\phi}\:\phi(z)=\pm1$.

\subsection{L-functions of cuspidal automorphic forms}

Let $\phi\in S_0(\Gamma_0(N))$ be a Hecke eigenform with the Laplacian eigenvalue $\lambda_{\infty,\phi}=
p^{-1/2}(e^{i\theta}+e^{-i\theta}).$ We define the L-function of $\phi$ by
$$L(s,\phi) = \sum_{a\in R\backslash\{0\}} \frac{\lambda_{\phi}(a)}{|a|^s}.$$

As the $\lambda_{\phi}(a)$ are the fourier coefficients of $\phi$ , and $\phi$ is bounded above as it is supported on finitely many $\Gamma_0(N)$ orbits, we have that
$\lambda_{\phi}(a)$ is bounded above independently of $a$. Hence, $L(s,\phi)$ converges absolutely in the region $Re(s)>1$.

As in the number field case, the L-function attached to $\phi$ is, up to an explicit factor, the Mellin transform of $\phi$. The purpose of this section is
to explicitly derive this relation over function fields. Following \cite{HMW} we define the Mellin transform of $\phi$ to be
\[M\phi(s)=\int_{k_{\infty}^{\times}}\phi\left(\left(\begin{smallmatrix} y & 0\\ 0 & 1\end{smallmatrix}\right)\right)|y|^sd^{\times}y\]
where the multiplicative measure is normalized so that $\int_{\Ooi^{\times}}d^{\times}y=1$. Note that since $\phi$ is cuspidal the integrand is supported only on
 finitely many points in $\mathbb{H}$, $M\phi(s)$ is a polynomial function of $p^{-s}$.

 Plugging in the Fourier expansion of $\phi$ we have
\[M\phi(s)=\int_{k_{\infty}^{\times}}\sum_{a\in R}\phi_a(y)|y|^sd^{\times}y.\]
For $Re(s)>2$ the sum and integral converge absolutely. Interchanging the sum and integral and using the Hecke relations (see end of $\S$2.7.1) gives
\begin{align*}M\phi(s)&=\sum_{a\in R}\int_{k_{\infty}^{\times}}\lambda_{\phi}(a)W_{0,i\theta}(ay)|y|^sd^{\times}y\\
&=\sum_{a\in R}\frac{\lambda_{\phi}(a)}{|a|^{s}}\int_{k_{\infty}^{\times}}W_{0,i\theta}(y))|y|^sd^{\times}y
\end{align*}
So we are left with the integral
\[I=\int_{k_{\infty}^{\times}}W_{0,i\theta}(y)|y|^sd^{\times}y\]
Now using (\ref{integralWhittaker})
\begin{align*}
I&=\frac{\sqrt{p}}{e^{i\theta}-e^{-i\theta}}\sum_{n=2}^{\infty}\left(\frac{e^{i(n-1)\theta}}{p^{(n-1)/2}}-\frac{e^{-i(n-1)\theta}}{p^{(n-1)/2}}\right)p^{-ns}\\
&=p^{-2s}\Gamma_k(s+1/2+i\gamma)\Gamma_k(s+1/2-i\gamma)
\end{align*}
Where $\gamma$ is such that $e^{i\theta}=p^{i\gamma}$. Hence the Mellin transform is
\begin{equation}
M\phi(s)=p^{-2s}\Gamma_k(s+1/2+i\gamma)\Gamma_k(s+1/2-i\gamma)L(s,\phi).
\end{equation}

It follows in particular that $L(s,\phi)$ is a polynomial function of $p^{-s}$.

\subsubsection{Depth $1$}

We now describe the Mellin transform of a $\phi\in S_1(\Gamma_0(N))$  which is a Hecke eigenform with the Laplacian eigenvalue $\lambda_{\infty,\phi}.$ We define the L-function of $\phi$ to be
$$L(s,\phi) = \sum_{a\in R\backslash\{0\}}\frac{\lambda^u_{\phi}(a)}{|a|^s}$$
Define the Mellin transform of $f(z)$ by
\[M\phi(s) = \int_{k_{\infty}^{\times}}\phi\left(\left(\begin{smallmatrix} y & 0\\ 0 & 1\end{smallmatrix}\right)\right))|y|^sd^{\times}y\]
Plugging in the Fourier expansion and proceeding as above, we arrive at
\[M\phi(s) = \sum_{a\in R}\frac{\lambda^u_{\phi}(a)}{|a|^{s}}\int_{k_{\infty}^{\times}}W_{1,\lambda_{\infty,\phi}}(y)|y|^sd^{\times}y\]
Evaluating the integral yields
\begin{equation}
M\phi(s)= \frac{p^{-2s}}{1-\lambda_{\infty,\phi}\:p^{-s}}\cdot L(s,\phi)
\end{equation}

\section{Lindel\"of Hypothesis for Families of L-functions}\label{sectiondrinfeld}

In this section we clarify the role that the generalised Riemann hypothesis and Drinfeld's work play in obtaining the relevant L-function bounds for proving the metaplectic Ramanujan conjecture. We start with a smooth, projective, geometrically irreducible curve $X$ over $\mathbb{F}_p$ of genus $g$. The zeta function attached to $X$ is defined as follows: \[\zeta(t,X)=\prod_{w\in X}\left(1-t^{\deg(w)}\right)^{-1}\]
where $w$ runs over closed points of $X$.  This product converges and defines a holomorphic function on the open unit disc $|t|<\frac1p$. Furthermore by the Riemann-Roch theorem there exists a degree $2g$ polynomial $P(t,X)$ such that
\[\zeta(t,X)=\frac{P(t,X)}{(1-t)(1-pt)}\]
By the work of Weil on the Riemann hypothesis for curves we know that the zeroes of $P(t,X)$ all lie on the circle $S=\{t\mid |t|=p^{-1/2}\}$. We see by the Riemann-Roch theorem again that $\zeta(t,X)$ satisfies a functional equation relating it to $\zeta(\tfrac{1}{pt},X)$. Translated to $P(t,X)$, this functional equation reads $$P\left(\frac{1}{pt},X\right)=C(X)t^{-2g}P(t,X).$$
Where $C(X)$ is the conductor of $X$ and it related to the genus of the curve by $C(X)=p^g$.
We will be considering the size of the polynomial $P(t,X)$ on the circle $S$ when $X$ is varying in a family $\mathcal{F}$ of smooth projective curves over $\mathbb{F}_p$ with increasing genus. We start with defining a special class of families of such curves.
\begin{definition*} A family $\mathcal{F}$ of curves with increasing genus is called \emph{Lindel\"of} if for $\epsilon>0$
\[M(X):=\max\left\{ |P(t,X)|\mid t\in S\right\}\] satisfies \[M(X)=O_{\epsilon}(C(X)^{\epsilon})\]
as $X$ varies in $\mathcal{F}$.
\end{definition*}

The asymptotic behavior of $M(X)$ depends heavily on the chosen family, and especially on the distribution of the zeroes of $P(t,X)$.
In particular we have the following Theorem:\\

\begin{thm}\label{lindelofcondition} A family $\mathcal{F}$ is Lindel\"of if and only if the zeros of $P(t,X)$ become equidistributed with respect to the Haar
measure on $S$ as genus goes to infinity.
\end{thm}

\begin{proof} Let $\mathcal{F}$ be a family of curves. For each curve $X$ we have a \emph{Weil measure} on the unit circle given by
$\mu_X=\frac1{2g(X)}\displaystyle\sum_{i=1}^{2g(X)}\delta_{\alpha_i}$ where $p^{-\frac12}\alpha_i$ are the roots of $P(t,X)$. Note that the set of roots
is closed under complex conjugation, and hence $$\{\alpha_i, 1\leq i\leq 2g\} = \{\alpha_i^{-1},1\leq i\leq2g\}.$$ As in \cite{TV} we can pick a sequence $X_j\subset\mathcal{F}$ of curves with genera $g_j$ such that the sequence $\mu_j=\mu_{X_j}$ converges to a measure $\mu$ in the weak-* topology.

Next, we observe that for $t\in S$,
\[\frac{1}{2g_j}\log(P(t,X_j))=\frac{1}{2g_j}\sum_{i=1}^{2g_j}\log(1-p^{1/2}\alpha_{j,i}t)=\int_{S^1}\log(1-wz)d\mu_j(z)\] where we set $w=p^{1/2}t$ and $S^1$ is the unit circle. We thus deduce that the family $\mathcal{F}$ is Lindel\"of iff
\begin{equation}\label{integraloflog1}
\limsup_{j}\left\{\max_{w\in S^1}\int_{S^1}\log(1-wz)d\mu_j(z)\right\}\leq 0
\end{equation}

We would now like to apply the outer limit to the measures $\mu_j(z)$, but there is a slight technical annoyance coming from the fact that $\log(1-e^{i\theta})$ is singular at $\theta=0$. This can be addressed by a smoothing argument: For $\epsilon>0$ define
$$L_{\epsilon,w}(z)=\int_{\theta=0}^{\epsilon}\log(1-we^{i\theta}z)d\theta.$$

Then by equation \eqref{integraloflog1} it follows that for the family $\mathcal{F}$ to be Lindel\"of we must have
\begin{equation}
\limsup_j\int_{S^1}L_{\epsilon,w}(z)d\mu_j(z)\leq 0.
\end{equation}
As $L_{\epsilon,w}(z)$ is continuous, we may now exchange the limit and integral to arrive at

\begin{equation}\label{intermint}
\int_{S^1}L_{\epsilon,w}(z)d\mu(z)\leq 0.
\end{equation}

By \cite{TV} the measure $\mu$ can be written as $F(\theta)d\theta$ where $z=e^{i\theta}$ and $F(\theta)$ is a non-negative \emph{continuous} probability function. It follows from this that the singularity of $\log(1-z)$ at $z=1$ is integrable with respect to $\mu$. Since equation \eqref{intermint} holds for all $\epsilon>0$
we deduce that for $\mathcal{F}$ to be Lindel\"of we must have

\begin{equation}\label{integraloflog}
\int_{S^1}\log(1-wz)d\mu(z) \leq 0
\end{equation}

Since $$\limsup_j\int_{S^1}\log(1-wz)d\mu_j(z)\leq \int_{S^1}\log(1-wz)d\mu(z)$$ equation \ref{integraloflog} is actually equivalent to $\mathcal{F}$ being a
Lindel\"of family.

Since $\displaystyle\int_{|w|=p^{1/2}}\log(1-wz) |dw|=0$, equation (\ref{integraloflog}) is actually equivalent to the stronger statement
\begin{equation}\label{logiszero}
\int_{S^1}\log(1-wz)d\mu=0
\end{equation}
for all $w\in S^1.$  Now, since $F(z)$ is continuous it is also in $L^2(S^1)$ and thus has a Fourier expansion $F(z)=\sum_{m\in\mathbb{Z}}F_me^{mz}$ for $\{F_m\}\in l^2(\mathbb{Z})$.
Since $\log(1-z)$ is also in $L^2(S^1)$ we can write
\begin{align*}
\int_{S^1}\log(1-wz)d\mu_X &=\int_{S^1}\log(1-wz)F(z)d|z|\\
&=\displaystyle\sum_{m\neq 0} \frac{F_mw^m}{m}\\
\end{align*}
Since $\frac{F_m}{m}\in l^1(\mathbb{Z})$, for $\displaystyle\sum_{m\neq 0} \frac{F_mw^m}{m}$ to be identically zero we must have $F_m=0$ for $m\neq 0$, or equivalently that $d\mu$ is the usual Haar measure on $S^1$. This completes the proof.

\end{proof}

The following Lemma shows that not every family of curves is Lindel\"of, and thus being Lindel\"of is not a formal consequence of the Riemann hypothesis.
We call the \emph{universal family} the family consisting of \emph{all} curves.\\
%The first of the following two results show that \emph{not} every family is Lindel\"of. It is based on an argument of Serre constructing curves with many
%rational points. The second one is the one in interest in the current paper and states that if we have a reasonable growth condition on the number of rational%
%points of the curves in the family then the family is Lindel\"of. More precisely

\begin{cor} The universal family is \emph{not} Lindel\"of.
\end{cor}

\begin{proof}
This follows from the work of Tsafman and Vladut \cite{TV} and Serre \cite{Ser} who showed that the family of modular curves $X_0(N)$ over $\mathbb{F}_p$ has Weil measures converging to
the p-adic Plancherel measure $\mu_p,$ and thus is not Lindel\"of.
\end{proof}

The distribution of zeroes for a curve $X$ is related to the number of rational points of $X$ by (see for example formula $(76)$ of \cite{Ser})
$$|X(\mathbb{F}_{p^r})|=p^r+1-p^{r/2}\displaystyle\sum_{i=1}^{2g(X)}\alpha_{X,i}^r$$

As all the Weil measures are invariant under $z\rightarrow\bar{z}$ the above relation combined with Theorem \ref{lindelofcondition} implies
that for a family to be Lindel\"of, it is neccessary and sufficient that
the number of rational points grows slower than the genus, i.e. $\forall r\in\mathbb{N}, |X(\mathbb{F}_{p^r})|=o_r(g(X))$. \\

\begin{definition*} The \emph{gonality}, $G(X)$, of a curve $X$ is the minimal degree of a morphism from $X$ to $\mathbb{P}^1_{/\mathbb{F}_p}.$
\end{definition*}

Because of the inequality $|X(\mathbb{F}_{p^r})|\leq G(X)(p^r+1)$, we immediately deduce the following theorem.\\

\begin{thm} A family $\mathcal{F}$ is Lindel\"of if we have $G(X)=o(g(X))$ as $X$ varies over $\mathcal{F}$.
\end{thm}

The authors have been noted that this result has also appeared in (\cite{F}, proposition 5.1) for the case of hyperelliptic curves, and in \cite{Shp} in the general case.

We now specialize to the case of hyperelliptic curves $X_D$, which are by definition those algebraic curves with gonality 2. These are obtained by adjoining $\sqrt{D}$ to the function field $\mathbb{F}_p(T)$ for a square-free
polynomial $D$. Define the L-function $L(s,\chi_D)$ as follows:
$$L(s,\chi_D) = \sum_{a\in R}\left(\frac{D}{a}\right)p^{-\deg(a)s}.$$ Then we have the identity
 $$P(p^{-s},X_D) = L(s,\chi_D).$$ As hyperelliptic curves have gonality 2, the family is Lindel\"of, but we can actually
do better for this family of curves. The following explicit bound is the analogue of a result of Littlewood \cite{L} in the number field case, and we follow very closely a simplified proof of Soundararajan, \cite{S}.\\

\begin{thm}\label{lindelofhyper}
Let $D$ be a polynomial of degree $2g$ or $2g+1$.
Then \[|L(1/2,\chi_D)|\leq e^{\frac{2g}{\log_p(g)} + 4p^{\frac12}g^{\frac12}}.\]
\end{thm}

\begin{proof}

We begin with some notation. Let $g=g_D$ be the genus of the curve $X_D,$ so that depending on the parity of $\deg(D)$, $|D|=p^{2g}$ or $p^{2g+1}$. We also define $p^{-1/2}\alpha_i, 1\leq i \leq 2g$ to be the roots of $P(t,X_D)$. By the Riemann hypothesis for curves we know that $|\alpha_i|=1$, and since
 $L(s,\chi_D)$ is real on the real line, the set of $\alpha_i$ is closed under inversion. Taking the logarithmic derivative of $L(s,\chi_D)$ gives
\[\frac{L'}{L}(s,\chi_D) = \log(p)\left(-2g+\sum_{i=1}^{2g}\frac{1}{1-\alpha_i p^{\frac{1}{2}-s}}\right)\]

Define $$F(s)=\sum_{i=1}^{2g}\Re\left(\frac{1}{1-\alpha_ip^{\frac12-s}}-\frac{1}{2}\right),$$ so that
\begin{equation}\label{logd}
\frac{L'}{L}(s,\chi_D) = \log(p)\left(-g+F(s)\right).
\end{equation}
Let $s_0$ be a real number such that $s_0>\frac12$. Integrating the above from $\frac{1}{2}$ to $s_0$ and taking real parts gives
\begin{equation}\label{Lfun1}
\log|L(1/2,\chi_D)|-\log\left|L\left(s_0,\chi_D\right)\right| = g\log(p)\left(s_0-\frac12\right)-\log(p)\sum_{i=1}^{2g}\int_{\frac12}^{s_0}\Re\left(\frac{1}{1-\alpha_i p^{\frac{1}{2}-s}}-\frac12\right)ds.
\end{equation}
To control the second term on the right hand side, we have the following lemma:

\begin{lemma}\label{hardineq}

Let $\theta, 0<t<1$ be real numbers. Then \[\int_0^t\Re\left(\frac{1}{1-e^{-s-i\theta}}-\frac12\right)ds\geq 2\cdot\frac{1+e^{-t}}{1-e^{-t}}\cdot \Re\left(\frac{1}{1-e^{-t-i\theta}}-\frac12\right)\]
\end{lemma}

\begin{proof}

Computing the integral on the left hand side, we arrive at \[\frac{t}{2}+\Re\log\left(\frac{1-e^{-t-i\theta}}{1-e^{-i\theta}}\right).\]
Using that $\Re(\log(z))=\log|z|$, we see the expression above equals:
\begin{align*}
\frac{t}{2}+\frac{1}{2}\cdot\log\left|\frac{1-2\cos(\theta)e^{-t}+e^{-2t}}{2-2\cos(\theta)}\right|&=
\frac{1}{2}\cdot\log\left|\frac{e^t-2\cos(\theta)+e^{-t}}{2-2\cos(\theta)}\right|\\
&=\log\left|1+\frac{e^t+e^{-t}-2}{2-2\cos(\theta)}\right|.\\
\end{align*}

Now using the inequality $\log(1+x)\geq \frac{x}{x+1}$ for $x>0$, we have that
\begin{align*}
\log\left|1+\frac{e^t+e^{-t}-2}{2-2\cos(\theta)}\right|&\geq \frac{1-2e^{-t}+e^{-2t}}{1-2e^{-t}\cos(\theta)+e^{-2t}}\\
&=\frac{1-e^{-t}}{1+e^{-t}}\cdot\left(\frac{1-e^{-2t}}{1-2e^{-t}\cos(\theta)+e^{-2t}}\right)\\
&=2\cdot\frac{1-e^{-t}}{1+e^{-t}}\cdot\left(\Re\left(\frac{1}{1-e^{-t-i\theta}}\right)-\frac{1}{2}\right)\\
\end{align*}
\end{proof}

 Using equation \eqref{Lfun1} and lemma \ref{hardineq} gives
\begin{equation}\label{firstsoundformula}
\log\left|L\left(\frac12,\chi_D\right)\right|-\log\left|L\left(s_0,\chi_D\right)\right|\leq g\log(p)\left(s_0-\frac12\right) -2\cdot\frac{1+p^{\frac12-s_0}}{1-p^{\frac12-s_0}}\cdot F(s_0)
\end{equation}

Define $h=\lceil\log_p(g)\rceil.$

Next, for $\Re(s)>0$ we compute $$\frac{1}{2\pi i}\displaystyle\int_{2}^{2+\frac{2\pi i}{\log(p)}}-\frac{L'}{L}(s+w,\chi_D)\frac{p^{hw}p^{-w}}{(1-p^{-w})^2}dw$$ in two different ways:

First by
expanding into $$\frac{L'}{L}(s,\chi_D) = \sum_{n\in\mathbb{N}}\frac{\lambda_D(n)}{p^{ns}}$$ and integrating term by term, and second by analytically continuing to the left
and picking up the residues. Note that there are double poles at $w=0$, as well as single poles at the values of $w$ for which $p^{s+w}=\alpha_ip^{\frac12}$.  We end up with the relation
$$-\log(p)^{-2}\sum_{0\leq n\leq h}\frac{\lambda_D(n)\log(p^{h-n})}{p^{ns}} =
h\log^{-1}(p)\frac{L'}{L}(s,\chi_D)+\log^{-2}(p)\left(\frac{L'}{L}(s,\chi_D)\right)'+\sum_{i=1}^{2g}\frac{(\alpha_i p^{\frac12-s})^h\alpha_i^{-1}p^{s-\frac12}}{(1-\alpha_i^{-1}p^{s-\frac12})^2}$$

We integrate the above from $s_0$ to $\infty$ , take real parts and multiply by $\log(p)^2$ to get

\begin{align*}
h\log(p)\log|L(s_0,\chi_D)| &= -\frac{L'}{L}(s_0,\chi_D)	+\log(p)^{-1}\sum_{0\leq n\leq h}\frac{\lambda_D(n)\log(p^{h-n})}{np^{ns_0}}\\
&+\log(p)^2\sum_{i=1}^{2g}\int_{s_0}^{\infty}\Re\left(\frac{(\alpha_i p^{\frac12-s})^h\alpha_i^{-1}p^{s-\frac12}}{(1-\alpha_i^{-1}p^{s-\frac12})^2}\right)ds
\end{align*}

\begin{lemma}\label{calcineq}
For $s>s_0>\frac12$ we have the inequality

\begin{align*}
\log(p)\left|\frac{\alpha_i^{-1}p^{s-\frac12}}{(1-\alpha_i^{-1}p^{s-\frac12})^2}\right|&\leq\left(s-\frac12\right)^{-1}\Re\left(\frac1{1-\alpha_ip^{\frac12-s}}\right)\\
&\leq\left(s_0-\frac12\right)^{-1}\Re\left(\frac1{
1-\alpha_ip^{\frac12-s_0}}\right)
\end{align*}

\end{lemma}

\begin{proof}
We first prove the first inequality: set $y=p^{s-\frac12}$ and $\alpha_i=e^{\theta_i}$. Using
$\Re(\frac1z)=\frac{z+\overline{z}}{2|z|^2}$ the inequality in the lemma rearranges to

\[\left|\frac{y}{(1-\alpha_i^{-1}y)^2}\right|\leq\frac{y}{\log(y)}\cdot\frac{y-\cos(\theta_i)}{\left|1-\alpha_i^{-1}y\right|^2}\]
which is true since
\[\log(y)\leq y-1\leq y-\cos(\theta_i)\] for $y>1$.

To see the second inequality, set $y = \Re\alpha_i$ so that $|y|<1$ and set $F(x)=x^{-1}\cdot\frac{1-ye^{-x}}{1-2ye^{-x}+e^{-2x}}$. Then setting $x=\log(p)\cdot(s-\frac12)$ the claim amounts to proving
that $F(x)$ is decreasing in the range $x>0$. Taking the derivative, we see $$F'(x) = -\frac{e^-t(e^{3t}+ e^t(1+2s^2-2t)+e^{2t}(-3s+st)+s(-1+t))}{4t^2(s-\cosh(t))^2}.$$
We need to show that $F'(x)<0$, or equivalently that $W_s(t):=e^-t(e^{3t}+ e^t(1+2s^2-2t)+e^{2t}(-3s+st)+s(-1+t))>$ is positive. Differentiating with respect to $t$ we see that
$W_s'(t)=e^{-t} (-1 + e^{2t}) (2e^t + s (-2 + t))$ which is clearly positive as $t>0$ and $s<1$. Finally, $W_s(0)=2(s-1)^2>0$. This completes the proof.

\end{proof}

The above lemma implies
\begin{align*}
\int_{s_0}^{\infty}\Re\left(\frac{(\alpha_i p^{\frac12-s})^h\alpha_i^{-1}p^{s-\frac12}}{(1-\alpha_i^{-1}p^{\frac12-s})^2}\right)ds&\leq \int_{s_0}^{\infty}\left|\frac{(\alpha_i p^{\frac12-s})^h\alpha_i^{-1}p^{s-\frac12}}{(1-\alpha_i^{-1}p^{\frac12-s})^2}\right|ds\\
&\leq\log(p)^{-1}\left(s_0-\frac12\right)^{-1}\Re\left(\frac1{1-\alpha_ip^{-s_0}}\right)\int_{s_0}^{\infty}|p^{h(1/2 -s)}|ds
\end{align*}
so that

\begin{align*}
\log|L(s_0,\chi_D)|&\leq -h^{-1}\log(p)^{-1}\frac{L'}{L}(s_0)+h^{-1}\log(p)^{-2}\sum_{0 < n\leq h}\frac{\lambda_D(n)\log(p^{h-n})}{np^{ns_0}}\\
 &+ h^{-2}\log(p)^{-1}\left(s_0-\frac{1}{2}\right)^{-1}F(s_0)p^{h(1/2-s_0)}
 \end{align*}

 Adding this to (\ref{firstsoundformula}) and using (\ref{logd}) we get

 \begin{align*}
 \log|L(1/2,\chi_D)|&\leq g\left(\log(p)\left(s_0-\frac12\right)+h^{-1}\right)+\\ &+ F(s_0)\left(\frac{\left(s_0-\frac12 \right)^{-1}p^{h(1/2-s_0)}}{h^2\log(p)}-2\cdot\frac{1+p^{\frac12-s_0}}{1-p^{\frac12-s_0}}-\frac{1}{h}\right)\\
&+ h^{-1}\log(p)^{-2}\sum_{0 < n\leq h}\frac{\lambda_D(n)\log(p^{h-n})}{np^{ns_0}}
\end{align*}

Taking $s_0=\frac{1}{2}+\frac{1}{h\log(p)}$  ensures the coefficient of $F(s_0)$ is negative. Since $F(s_0)>0$, we arrive at
\begin{equation}\label{mainineqlindelof}
\log|L(1/2,\chi_D)|\leq \frac{2g}{h} +h^{-1}\log(p)^{-2}\sum_{0 < n\leq h}\frac{\lambda_D(n)\log(p^{h-n})}{np^{ns_0}}
\end{equation}

We need to estimate $|\lambda_D(n)|$. We have the product formula \[L(s,\chi_D)=(p-1)\prod_{P}\frac{1}{1-\chi_D(P)p^{-\deg(P)s}}\]
 where the product is over all monic irreducible polynomials $P\in\F_p[T]$. Therefore
 \[\frac{L'}{L}(s,\chi_D)=\sum_{P}\sum_{n\geq 1} \chi_D^n(P)p^{-n\deg(P)s}.\]

 Thus $$\lambda_D(n)=\sum_{d|n}\sum_{\deg(P)=d}\chi_D^n(P)\leq n\#\{P\mid\deg(P)=n\}\leq np^n.$$
Using this and $s_0>\frac12$ we arrive at
$$\log|L(1/2,\chi_D)|\leq \frac{2g}{h} + 4p^{\frac{h}{2}}\leq\frac{2g}{\log_p(g)} + 4p^{\frac12}g^{\frac12}.$$

Exponentiating implies the result.

\end{proof}

We will also bound $L(1,\chi_D)$ as it will come up for us as a normalizing coefficient in the Siegel mass formula. We point out that in the number field case, the ineffectivity of Theorem \ref{representationTheorem} is due to the ineffectivity of the bound on $L(1,\chi_D)$ that is caused by a possible Siegel zero. Since in the function field setting Riemann hypothesis eliminates all the possible Siegel zeroes we get an effective bound. The following is the analogue of a result of Littlewood over $\mathbb{Q}$.\\

\begin{lemma}\label{classnumber}
Let $D$ be a square-free polynomial of degree $2g$ or $2g+1$. Then
$$(\log_p(g))^{-1}\ll_p|L(1,\chi_D)|\ll_p (\log_p(g))$$
\end{lemma}
\begin{proof}

We start with the identity $$\log(L(1,\chi_D))=-\sum_{r\in\mathbb{N}}\frac{1}{rp^{\frac{r}2}}\sum_{i=1}^{2g}\alpha_i^r$$ where
$p^{-\frac{1}{2}}\alpha_i$ are the roots of $L(s,\chi_D)$ as above.
Since $|\alpha_i|=1$ we have the bound $\left|\sum_{i=1}^{2g}\alpha_i^r\right|\leq 2g$. Also, since the hyperelliptic family under consideration has gonality 2,
we know that \[2(p^r+1)=2\mathbb{P}^1(\F_{p^r})\geq |X_D(\F_{p^r})| = p^r + 1 -p^{\frac{r}{2}}\sum_{i=1}^{2g}\alpha_i^r\] so we have the bound $\left|\sum_{i=1}^{2g}\alpha_i^r\right|\leq (p^{\frac{r}{2}}+p^{-\frac{r}{2}}).$ We can thus write
$$|\log(L(1,\chi_D))|\leq O(1)+\displaystyle\sum_{1\leq r\leq 2\log_p(g)}\frac1r + 2g\sum_{2\log_p(g)<r}p^{-r/2}\leq O_p(1) + \log\log_p(g)$$
The Lemma then follows by exponentiating the above.
\end{proof}

We will also need to study the degree $2$ L-functions
$$L(s,\phi\times\chi_D)=\displaystyle\sum_{\substack{a\in\mathbb{F}_p[T]\\(a,D)=1}}\lambda_{\phi}(a)\chi_D(a)|a|^{-s}$$
where $\phi$ is an automorphic Hecke eigenform of level $N.$ By Drinfeld's work \cite{Dr} on the Langlands conjectures for $GL(2)$ and Deligne's proof of the Weil conjectures \cite{Del}, this is a polynomial of degree
at most $DN$ in $p^{-s}$ which only has roots on the line $\frac{1}{2}$. It also follows that $|\lambda_{\phi}(P)|\leq 2|P|^{\frac{1}{2}}$ for every monic irreducible
polynomial $P$.  Define $\beta_{\phi,D}$ by expanding the logarithmic derivative $$-\frac{L'}{L}(s,\phi\times\chi_D)=\sum_{n\in\mathbb{N}} \frac{\beta_{\phi,D}(n)}{p^{ns}}.$$

%Using the Hecke relations, we also have the product formula \[L(s,\phi\times\chi_D)=\prod_{P\not\mid N}(1-alpha
Using the product formula we have the bound $|\beta_{\phi,D}(n)|\leq  2np^n\log(p)$. Then repeating the proof of Theorem \ref{lindelofhyper} we arrive at the following analogue:\\

\begin{thm}\label{lindelofhyper2}
Let $D$ be a polynomial of degree $2g$ or $2g+1$. Then

\[|L(1/2,\phi\times\chi_D)| \ll e^{\frac{2g}{\log_p(g)} + 8p^{\frac12}g^{\frac12}}\]
\end{thm}

We conclude this section by mentioning that the above is not limited to the specific family of $GL_2$ forms that we are considering. In particular, let $\mathcal{F}_{cusp}$ denote the family of unitary cuspidal automorphic representations of $GL(n)$ for a \emph{fixed} $n$. Now consider the family of \emph{standard} L-functions, $\mathcal{F}:=\begin{setdef}{L(s,\pi)}{\pi\in \mathcal{F}_{cusp}}\end{setdef}$ attached to each member of $\mathcal{F}_{aut}$. We sketch an analytic argument proving that the family $\mathcal{F}$ is Lindel\"of.

Let $L(s,\pi)\in\mathcal{F}$ be given by
\[L(s,\pi)=\sum_a\frac{\lambda_{\pi}(a)}{|a|^s}\]
The L-function $L(s,\pi)$ also factors as $L(s,\pi)=\prod_{v}L_v(s,\pi)$ where each local factor is given by
$$L_v(s,\pi) = \prod_{i=1}^n \left(1-\frac{\beta_{\pi,i}(v)}{|v|^s}\right)^{-1}$$ where some of the $\beta_{\pi,i}(v)$ could vanish if $\pi$ is ramified at $v$.
Since we are fixing the dimension of the underlying group, $n$, and considering only the standard $L$-functions, by the unitarity of $\pi$'s there is a $t>0$ such that the $L$-functions $L(s,\pi)$ all converge for $\Re(s)>t$. Then we see that each eigenvalue uniformly satisfies the upper bound $|\beta_{\pi,i}(v)|\leq |v|^t$. Hence the coefficients $\lambda_{\pi}(a)$ satisfy
\begin{equation}\label{rankbound}
|\lambda_{\pi}(a)|\leq n|a|^t
\end{equation}
We now follow the argument in the proof of Theorem \ref{lindelofhyper}. By the work of Lafforgue and
Deligne \cite{La}, the zeroes of $L(s,\pi)$ are all on the critical line $\Re(s)=\frac12$. The non-zero moments of the zeroes are
controlled by the coefficients $\lambda_{\pi}(a)$ and are thus bounded by (\ref{rankbound}). The zeroes are therefore equidistributed, and the Lindel\"of bound follows as before.

\section{Shimura Correspondence and Niwa's Lemma}\label{sectionshimura}

The goal of this and the following section is to prove a ``Waldspurger type formula" relating the Fourier-Whittaker coefficients of a metaplectic form to the central values of the
twisted L-function of its Shimura lift. Our method of proof will follow Katok-Sarnak \cite{KS}.

\subsection{Siegel Theta Functions for function fields}
% I realize now that all this nonsense works only if deg D = even. If it isn't even, this has to be redefined somehow. I think we won't bother, because
% it's super annoying? of course, we could... adellically it makes no difference.

In this section we will define the theta functions that will be used to define the theta lifts. We start with some notation. Let $V$ be the three dimensional vector space $(h_1,h_2,h_3)$ over $k_{\infty}$, $Q$ the quadratic form
$Q(\vec{h}) = h_2^2 - 4h_1h_3$, and $B(\cdot,\cdot)$ its associated billinear form. Define $\mathcal{O}_V$ to be the $\mathcal{O}_{\infty}$ lattice
$\{h_1,h_2,h_3\}\in \mathcal{O}_{\infty}^3$, and $L$ to be the $R$ lattice $\{h_1,h_2,h_3\}\in (TR)^3$. Recall that we write $w$ symbolically for the element $\left(\left(\begin{smallmatrix}\sqrt{v}&u/\sqrt{v}\\0&1/\sqrt{v}\end{smallmatrix}\right),1\right)$. Finally, as in the number field case, we let $PGL_2(k_{\infty})$ act on $V$ via the isomorphism with the orthogonal group. Namely,
$$g\left(\left(\begin{smallmatrix} 2h_3 & h_2\\ h_2 & 2h_1\end{smallmatrix}\right)\right) =
\det(g)^{-1}\cdot g\left(\begin{smallmatrix} 2h_3 & h_2\\ h_2 & 2h_1\end{smallmatrix}\right)g^t.$$
For convenience, we defie the function $\tau:k_{\infty}^{\times}\rightarrow\pm \{1\}$ by $\tau(x)=(\varpi_{\infty},x)_{\infty}^{v_{\infty}(x)}.$
We define the Siegel theta function of level $0$, depth $0$, to be

$$ \Theta(w;z) = \tau(\sqrt{v})|v|^{3/4}\displaystyle\sum_{\vec{h}\in L} e(Q(\vec{h})u)\chi_{\mathcal{O}_V}(\sqrt{v}z^{-1}(\vec{h}))$$

where $w\in \widetilde{\mathbb{H}}$ and $z\in\mathbb{H}$. As $\mathcal{O}_V$ is invariant under the action of $PGL_2(\mathcal{O}_{\infty})$, $\Theta(w;z)$ is a well
defined function on $\Tt\times \mathbb{H}$. Also, in the $z$-variable $\Theta$ is evidently left invariant under the action of the full modular group $PGL_2(R)$. In the $w$-variable it is an element of $\widetilde{M}_0(\widetilde{\Gamma})$ by Lemma \ref{invariance}.

We will need more general classes of theta functions to get our lifts. To define them, first let $D$ be a square-free polynomial of even degree, such that $D\in \left(k_{\infty}^{\times}\right)^2$. For a quadratic form $Q$ over $R$ such that $D\mid\disc(Q)$, define $W_D(Q)$, as in
Kohnen \cite{Ko} as follows:\\

\begin{definition*}\label{kohnen}
If for all $v\in \mathcal{O}_V$, $Q(v)$ is not relatively prime to $D$ then $W_D(Q) = 0.$ Else, pick a vector $v_0$ such that $(Q(v_0),D)=1,$ and define
$W_D(Q) =  \left(\frac{Q(v_0)}{D}\right).$ Note that this definition is independent of $v_0.$
\end{definition*}

Identifying $\vec{h}$ with the quadratic form $h_1X^2+h_2XY+h_3Y^2$, we can speak meaningfully of $W_D(\vec{h})$.

We then define

$$\Theta_D(w;z) = \tau(\sqrt{v})|v|^{3/4}\displaystyle\sum_{\substack{\vec{h}\in L\\ D\mid Q(\vec{h})}}W_D(\vec{h})e(Q(\vec{h})u/D)\chi_{\mathcal{O}_V}(\sqrt{v}z^{-1}(h)/\sqrt{D})$$

As before, it is clear $\Theta_D$ is a well-defined function on $\Tt\times \mathbb{H}$ and is invariant under the action of the full modular group in the $z$-variable. The fact
that $\Theta_D(w;z)$ is an element of $\widetilde{M}_0(\widetilde{\Gamma})$ in the $w$-variable follows from Lemma \ref{invariance}.

As we will be working over a general congruence subgroup $\Gamma_0(N),$ we will need the analogue of $\Theta_D(w;g)$ for arbitrary level $N$ that is relatively prime to $D$. Let $N\in R$ such that $\gcd(N,R)=1$, and define the Siegel theta function of level $N$, depth $0$, to be:

\[\Theta^N_D(w;g) = \tau(\sqrt{v})|v|^{3/4}\displaystyle\sum_{\substack{\vec{h}\in L,\, h_1\in TNR\\ D\mid Q(\vec{h})}}W_D(\vec{h})e(Q(\vec(h))u/D)\chi_{\mathcal{O}_V}(\sqrt{v}g^{-1}(h)/\sqrt{D})\]

As before $\Theta_D^N$ defines a function on $\widetilde{\mathbb{H}}\times \mathbb{H}$. Note that $\Theta_D^N$ is invariant under the congruence subgroup $\Gamma_0(N)$ in the $g$-variable. The same argument in Lemma \ref{invariance} and an appropriate choice of functions at various completions yield $\Theta^N_D(w;g)\in \widetilde{M}_{0}(\widetilde{\Gamma}_0(N))$ in the $w$-variable. More specifically, in the notation of \S\ref{classicaltheta}, one chooses $\phi_w$ for all $w$ such that $v_w(N)>0$ to be
$$\phi_w(h_1,h_2,h_3) =\chi_{L_w}(N^{-1}h_1,h_2,h_3)$$ where $L_w:=L\otimes_{R}\Oo_w$ and $\chi_{L_w}$ denotes the characteristic function of $L_w$.

\subsubsection{Depth $1$}

We shall also need a Siegel theta function to transfer depth $1$ forms. Taking our cue from the adelic description, define $\Oo_{V,1}$ to be the lattice $\mathcal{O}_{V,1}:=\begin{setdef}{\vec{h}=(h_1,h_2,h_3)}{(Th_1,h_2,h_3)\in\mathcal{O}_{\infty}^3}\end{setdef}.$ We then pick $\phi_{\infty}(v) = \chi_{\Oo_{V,1}}(v).$ This yields

\[\Theta^{N}_{D,1}(w;z) = \tau(\sqrt{v})|v|^{3/4}\displaystyle\sum_{\substack{\vec{h}\in L,\, D\mid\det(\vec{h})\\ h_1\in NTR}}W_D(\vec{h})
e((h_2^2 - 4h_1h_3)u/D)\chi_{O_{V,1}}(\sqrt{v}z^{-1}(h)/\sqrt{D})\]

Then $\Theta_{D,1}^N$ is a well-defined function on $\widetilde{\mathbb{H}}_1\times \mathbb{H}_1$, that it is invariant under $\Gamma_0(N)$ in the $z$-variable, and once again using the same arguments in Lemma \ref{invariance} we see that it is an element of
$\widetilde{M}_{1}(\widetilde{\Gamma}_0(N))$ in the $w$-variable. Note that the above Fourier expansion in the $w$-variable is valid only on $\Tt^u_{1}\times\mathbb{H}_1.$

\subsection{Maass-Shintani lift}

Let $D,N\in R$, $\gcd(D,N)=1$ be as above. We will use $\Theta^N_D(w;z)$ as a kernel against which we integrate cuspidal automorphic forms and cuspidal metaplectic forms to go from one space to the other. As in \cite{KS}, for a given cuspidal automorphic form $\phi(z)\in S(\Gamma_0(N))$, we define its $D$'th Maass-Shintani lift $F^{[D]}(w)$ to be the  metaplectic function given by $$F^{[D]}(w) = \int_{\Gamma_0(N)\backslash \mathbb{H}} \phi(z)\,\Theta^N_D(w;z)dz.$$

Similarly, if $\phi(z)\in S_1(\Gamma_0(N)$ we define its $D$'th Maass-Shintani lift $F^{[D]}(w)$ to be the  metaplectic function given by $$F^{[D]}(w) = \int_{\Gamma_0(N)\backslash \mathbb{H}_1} \phi(g)\,\Theta^N_{D,1}(w;z)dz.$$

Our main result for this section is:\\
\begin{thm}\label{shintanilift}
Let $\phi(g)$ be a cuspidal automorphic form of level $N$, depth $0$ or $1$, and for $D\in R$, $\gcd(D,N)=1$, $F^{[D]}(w)$ be its D'th Maass-Shintani lift. Then $F^{[D]}$ is a cuspidal metaplectic
 form, and we have the following identity:
 $$\lambda_{F^{[D]}}(D)=\begin{cases}C_{\phi}G_1(D)|D|^{-3/4}L(1/2,\phi\times\chi_D)\prod_{i=1}^m \left(1+\left(\frac{l_i^{\alpha_i}}{D}\right)w_{l_i^{\alpha_i},\phi}\right)&\text{if $\phi$ is of depth $0$}\\ C_{\phi}G_1(D)|D|^{-3/4}L(1/2,\phi\times\chi_D)\prod_{i=1}^m \left(1+\left(\frac{l_i^{\alpha_i}}{D}\right)w_{l_i^{\alpha_i},\phi}\right)(1+w_{\infty,\phi})&\text{if $\phi$ is of depth $1$}\end{cases}$$
Where $N=\prod_{i=1}^m l_i^{\alpha_i}$ is the prime factorization of $N$, $w_{l_i^{\alpha_i},\phi}$ are the eigenvalues of $\phi$ under the Atkin-Lehner
involutions, and $$C_{\phi} = \begin{cases}
\frac{p^{-3/2}}{\left(1-\frac{e^{-i\theta}}{\sqrt{p}}\right)\left(1-\frac{e^{i\theta}}{\sqrt{p}}\right)}& \textrm{ depth 0}\\
\frac{-pw_{\infty,\phi}}{p+w_{\infty,\phi}} & \textrm{depth 1 case}\end{cases}$$ where $\lambda_{\infty,\phi}=p^{-1/2}(e^{i\theta}+e^{-i\theta})$ is the Laplacian eigenvalue of $\phi$ in the depth $0$ case, and $w_{\infty,\phi}$ is the eigenvalue of $\phi$ under the Atkin-Lehner involution $W_{\infty}$ in the depth 1 case.
\end{thm}

\begin{proof} The cuspidality follows Corollary 15.6 of \cite{G-P-S}. For the statement of the lemma we only need the $D$'th Fourier coefficient of $F^{[D]}$. For simplicity,
we first do the case of $N=1.$  Since $F^{[D]}$ is left-invariant under $\widetilde{\Gamma}_{\infty}=\begin{setdef}{\left(\begin{smallmatrix}1&a\\0&1\end{smallmatrix}\right)}{a\in R}\end{setdef}$, we can recover its Fourier coefficients as follows:
\begin{align*} \label{quadraticformssum} % figure out labeling
F^{[D]}_{D}(v) &=\mu(T^{-1}\mathcal{O}_{\infty})^{-1}\int_{T^{-1}\mathcal{O}_{\infty}}F^{[D]}(w)e(-DT^2u)du \\
&=\tau(v)|v|^{3/4}\int_{\Gamma\backslash \mathbb{H}} \phi(z)\sum_{\vec{h}\in T^{-1}L, Q(\vec{h})=D^2}W_D(\vec{h})\chi_{\mathcal{O}_V}(T\sqrt{v}z^{-1}(\vec{h})/\sqrt{D})dz \\
&=\tau(v)|v|^{3/4}\sum_j W_D(\vec{h_j})\int_{\mathbb{H}} \phi(z)\chi_{\mathcal{O}_V}(T\sqrt{v}z^{-1}(\vec{h_j})/\sqrt{D})dz
\end{align*}
Where the last step is obtained by unfolding the integral, and $\vec{h_j}$ is a set of representatives for the action of $PGL_2(R)$ on vectors
with norm $D^2$. A complete such set of representatives is $\left\{\vec{h_j} = (0,D,j)\mid j\in R/DR\right\}$.
%which corresponds to the quadratic forms $DXY + jY^2$
  Note $W_D(h_j) = \left(\frac{j}{D}\right)$. We define $g_j$ to be the element $\left(\begin{smallmatrix}
1 & \frac{j}{D} \\
0 & 1 \\
\end{smallmatrix}
\right)$, so that $g_j^{-1}(\vec{h_j})=\vec{h_0}$. Denoting the summand in the above sum corresponding to $h_j$ by $I_j$ and  making a change of
coordinates we have
$$I_j = \tau(v)|v|^{3/4}\left(\frac{j}{D}\right)\int_{\mathbb{H}}\phi(g_jg)\chi_{\mathcal{O}_V}(T\sqrt{v}g^{-1}(\vec{h_0})/\sqrt{D})dg$$

Recalling $z=\left(\begin{smallmatrix} y & x\\ 0 & 1\end{smallmatrix}\right)$, we have $z^{-1}(\vec{h_0}) = (0,D,-Dx/y)$, so
we
can rewrite the integral as follows
\begin{align*}
I_j &= \tau(v)|v|^{3/4}\left(\frac{j}{D}\right)\int_{\mathbb{H}}\phi(g_jg)\chi_{\mathcal{O}_{\infty}}(Tx\sqrt{vD}/y)\chi_{\mathcal{O}_{\infty}}(T\sqrt{vD})dx\frac{d^{\times}y}{y} \\
&=\tau(v)|v|^{3/4}\chi_{\mathcal{O}_{\infty}}(T\sqrt{vD})\left(\frac{j}{D}\right)\int_{\mathbb{H}}\sum_{a\in R}e(T^2aj/D)\lambda_{\phi}(a)W_{0,i\theta}(ay)e(aT^2x)\chi_{\mathcal{O}_{\infty}}(Tx\sqrt{vD}/y))dx\frac{d^{\times}y}{y} \\
&=p^{-1}\tau(v)|D|^{-1/2}\chi_{\mathcal{O}_{\infty}}(T^2vD)\left(\frac{j}{D}\right)|v|^{1/4}
\int_{k_{\infty}^{\times}}\sum_{a\in R}e(T^2aj/D)\lambda_{\phi}(a)W_{0,i\theta}(ay)\chi_{\mathcal{O}_{\infty}}(Tay/\sqrt{vD})d^\times y\\
&=p^{-1}\tau(v)|D|^{-1/2}\chi_{\mathcal{O}_{\infty}}(T^2vD)\left(\frac{j}{D}\right)|v|^{1/4}
\int_{k_{\infty}^{\times}}\sum_{a\in R}e(T^2aj/D)\lambda_{\phi}(a)W_{0,i\theta}(y)\chi_{\mathcal{O}_{\infty}}(Ty/\sqrt{vD})d^\times y\\
\end{align*}

where the third equality follows by evaluating the integral over $x$, and the fourth equality is the change of variables $y\rightarrow\frac{y}{a}.$

We now want to interchange the sum and integral signs, but unfortunately the sum is not absolutely convergent in the 4th equation. To get around this problem, we introduce the complex variable $s$ and define
$$I_j(s) :=\tau(v)p^{-1}|D|^{-1/2}\chi_{\mathcal{O}_{\infty}}(T^2vD)\left(\frac{j}{D}\right)|v|^{1/4}
\int_{k_{\infty}^{\times}}\sum_{a\in R}e(T^2aj/D)\frac{\lambda_{\phi}(a)}{|a|^s}W_{0,i\theta}(ay)\chi_{\mathcal{O}_{\infty}}(Tay/\sqrt{vD}) d^\times y$$

Note that $I_j(0)=I_j$, and $I_j(s)$ is analytic in $s.$ Now, for $Re(s)>>0,$ the sum is absolutely convergent, so we can write:
$$I_j(s)=\tau(v)p^{-1}|D|^{-1/2}\chi_{\mathcal{O}_{\infty}}(T^2vD)\left(\frac{j}{D}\right)|v|^{1/4}
\sum_{a\in R}e(T^2aj/D)\frac{\lambda_{\phi}(a)}{|a|^s}\int_{k_{\infty}^{\times}}W_{0,i\theta}(y)\chi_{\mathcal{O}_{\infty}}(Ty/\sqrt{vD}) d^\times y$$

Since $\chi_{\mathcal{O}_{\infty}}(T^2vD)$ vanishes for $v_{\infty}(vD)<2$, we restrict to  $v$ such that $v_{\infty}(vD)\geq2$. It follows that $v_{\infty}(Ty/\sqrt{vD})\leq v_{\infty}(y)-2$, so  that $\chi_{\mathcal{O}_{\infty}}(Ty/\sqrt{vD})\neq0$ only if $v_{\infty}(y)\geq2$. Since by equation (\ref{integralWhittaker}) $W_{0,i\theta}(y)\neq0$ only if $v_{\infty}(y)\geq2$, we get

\[\int_{k_{\infty}^{\times}}W_{0,i\theta}(y)\chi_{\mathcal{O}_{\infty}}(Ty/\sqrt{vD})d^{\times}y=\sum_{n=v_{\infty}(\sqrt{vD})+1}^{\infty}W_{0,i\theta}(T^{-n})\]

Evaluating the sum yields

\[\frac{p^{\frac{2-M}{2}}}{e^{i\theta}-e^{-i\theta}}\left(\frac{e^{iM\theta}}{\sqrt{p}-e^{i\theta}}-\frac{e^{-iM\theta}}{\sqrt{p}-e^{-i\theta}}\right)\]

where $M=v_{\infty}(\sqrt{vD})$. This multiplied by $\chi_{\mathcal{O}_{\infty}}(T^2vD)|v|^{1/4}$ is easily seen to be $|D|^{-1/4}C_{\phi}\widetilde{W}_{0,i\theta}(Dv)$. Adding up the $I_j$
and using the fact that

$$\sum_{j\mod{D}} \left(\frac{j}{D}\right)e(T^2aj/D)=G_1(D)\left(\frac{a}{D}\right),$$ we arrive at

$$\sum_j I_j(s) = G_1(D)|D|^{-3/4}L(1/2+s,\phi\times\chi_D)\widetilde{W}_{0,i\theta}(Dv)C_{\phi}.$$

Continuing analytically to $s=0,$ we get

$$\lambda_{F^{[D]}}(D)\widetilde{W}_{0,i\theta}(Dv)=F^{[D]}_D(v)=G_1(D)|D|^{-3/4}L(1/2,\phi\times\chi_D)\widetilde{W}_{0,i\theta}(Dv)C_{\phi}$$

and so $$\lambda_{F^{[D]}}(D)=C_{\phi}G_1(D)|D|^{-3/4}L(1/2,\phi\times\chi_D)$$ as desired.

We now say a few words about the case of general level $N.$ Repeating the above calculation, we arrive at

$$F^{[D]}(v)=\tau(v)|v|^{3/4}\sum_{j,\; t} W_D(\vec{h^t_j})\int_{\mathbb{H}} \phi(g)\chi_{\mathcal{O}_V}(T\sqrt{v}z^{-1}(\vec{h^t_j})/\sqrt{D})dg$$

where now $\vec{h^t_i}$ is a set of representatives for the action of $\Gamma_0(N)$ on vectors $\vec{h}=(h_1,h_2,h_3)$ with $Q(\vec{h})=D^2$ such that $N|h_1.$ The set of
such forms is given by
$$\newline\left\{\vec{h^t_j} = W_t\vec{h_j}\mid \vec{h_j} = (0,D,j), j \in R/DR\right\}$$ and $W_t$ varies over the Atkin-Lehner
involutions of $\Gamma_0(N).$ Recall that this means that $t$ varies over the prime powers $l^{\alpha}$ such that $l^{\alpha}\mid\mid N$. A
proof of this can be found in \cite{B} pg.29 (Biro proves the statment for $\mathbb{Z}$ and the proof carries through verbatim to our setting.).
\begin{lemma}
$W_D(\vec{h^t_j}) = \left(\frac{t}{D}\right)W_D(\vec{h_j}).$
\end{lemma}
\begin{proof}
Note that if we consider the symmetric matrix $M$ corresponding to the qudratic form $\vec{h^t_j}$, then $W_D(\vec{h_j}) = \left(\frac{v_0Mv_0^t}{D}\right)$ for some $v_0$. Now, letting
$v_1=W_t^{-1}v_0$ and recalling that $\det(W_t)=t$ we have $$W_D(\vec{h^t_j})=\left(\frac{t^{-1}\cdot v_1W_tMW_t^{t}v_1^t}{D}\right)=\left(\frac{t}{D}\right)W_D(\vec{h_j})$$ as desired.
\end{proof}

Given the above lemma, we have
\begin{align*}
F^{[D]}(v)&=\tau(v)|v|^{3/4}\sum_{j,\; t} W_D(\vec{h^t_j})\int_{\mathbb{H}} \phi(z)\chi_{\mathcal{O}_V}(T\sqrt{v}z^{-1}(\vec{h^t_j})/\sqrt{D})dz\\
&=\tau(v)|v|^{3/4}\sum_j W_D(h_j)\sum_t \left(\frac{t}{D}\right)\int_\mathbb{H}\phi(W_tz)\chi_{\mathcal{O}_V}(T\sqrt{v}z^{-1}(\vec{h_j})/\sqrt{D})dz\\
&=\tau(v)|v|^{3/4}\prod_{i=1}^m \left(1+\left(\frac{l_i^{\alpha_i}}{D}\right)w_{l_i^{\alpha_i},\phi}\right) \sum_j W_D(h_j)\int_\mathbb{H}\phi(z)\chi_{\mathcal{O}_V}(T\sqrt{v}z^{-1}(\vec{h_j})/\sqrt{D})dz\\
\end{align*}
 the calculation then proceeds as before and the main result follows. \end{proof}

\subsubsection{Depth $1$}

For $\phi(g)\in S_1 (\Gamma_0(N))$ a cuspidal automorphic form of level $N$, depth $1$, we proceed as above to get
\begin{align*}
F^{[D],u}_D(v) &=\mu(T^{-1}\mathcal{O}_{\infty})^{-1}\int_{T^{-1}\mathcal{O}_{\infty}}F^{[D],u}(w)e(-DT^2u)du\\
&=\int_{\Gamma_0(N)\backslash \mathbb{H}_1} \phi(g)\sum_{\vec{h}\in R^3, \det(\vec{h})=D^2}W_D(\vec{h})\chi_{O_{V,1}}(T\sqrt{v}z^{-1}(\vec{h})/\sqrt{D})dz\\
&=\sum_{j,\; t} W_D(\vec{h^t_j})\int_{\mathbb{H}_1} \phi(g)\chi_{O_{V,1}}(T\sqrt{v}z^{-1}(\vec{h_j})/\sqrt{D})dz\\
\end{align*}

where the $h^t_j$ are as before. Now, since the Atkin Lehner involution $W_{\infty}$ preserves $O_{V,1},$ we can restrict to the integral over the upper component $\mathbb{H}^u_1:$

\begin{equation}
F^{[D],u}_D(v)=(1+w_{\infty,\phi})\sum_{j,\; t} W_D(\vec{h^t_j})\int_{\mathbb{H}_1^u} \phi(g)\chi_{O_{V,\infty}}(T\sqrt{v}g^{-1}(\vec{h_j})/\sqrt{D})dg\\
\end{equation}

We can now proceed as before except that the Whittaker function is different. The result follows as before from the following identity, which amounts to evaluating a geometric series as before:
\begin{align*}
&(\varpi_{\infty},v)^{deg(v)/2}|Dv|^{1/4}\chi_{\Ooi}(T^2vD)\int_{k_{\infty}^{\times}}W_{1,\lambda_{\infty,\phi}}(y)\chi_{\Ooi}(Ty/\sqrt{vD})d^{\times}y=\\
&\frac{-p^2w_{\infty,\phi}}{w_{\infty,\phi}+p} \widetilde{W}_{1,\lambda_{\infty,F}}(vD)
\end{align*}

where $\lambda_{\infty,F} = -w_{\infty,\phi}p^{3/2}$.

\subsection{Niwa's Lemma}

In \cite{Shim} it is shown how to construct a holomorphic modular form of integral weight starting from a holomorphic modular form of half integral weight. The
method that Shimura uses to show that the resulting function is a modular form is the converse theorem of Weil. Later after Shintani's \cite{Sh} proof
of the ``reverse'' correspondence (which was studied earlier by Maass \cite{M} in certain specific cases), Niwa \cite{Niwa} gave a more direct proof of Shimura's result by using theta functions. In \cite{KS} the result is generalized to the non-holomorphic case following the methods of Niwa. In this section, we generalize this approach to the function field setting.

For $F$ a cuspidal metaplectic form (of depth $0$ or $1$), let the Fourier-Whittaker expansion of $F$ be given by
\[F(w)=\begin{cases}\sum_{l\in R}\lambda_F(l)\widetilde{W}_{0,\lambda_{\infty,F}}(v)e(T^2lu)&\text{if $F$ is of depth $0$}\\ \sum_{l\in R}\lambda^u_F(l)\widetilde{W}_{1,\lambda_{\infty,F}}(v)e(T^2lu)&\text{if $F$ is of depth $1$} \end{cases}\]

Now suppose $F$ is of level $N$ and let $D\in R$ be a square-free polynomial that is relatively prime to $N$. We define the L-function $L(s,F,D)$ by
\[L(s,F,D)=\begin{cases}\sum_{l\in R}\frac{\lambda_F(Dl^2)}{|l|^s}&\text{if $F$ is of depth $0$}\\ \sum_{l\in R}\frac{\lambda^u_F(Dl^2)}{|l|^s}&\text{if $F$ is of depth $1$}\end{cases}\]

For such a $D$, we define the $D$'th Shimura lift of $F$ as follows:\\

\begin{definition*}[$D'th$ Shimura lift] Let $F(w)$ be a cuspidal metaplectic form of level $N$ (and of depth $0$ or $1$) and $D\in R$ a square-free polynomial. Define the $D'th$ Shimura lift $\phi(z)$ of $F$ by:\\
\[\phi(z) = \begin{cases}\displaystyle\int_{\widetilde{\Gamma}_0(N)\backslash\Tt} F(w)\overline{\Theta^N_D(w;z)}dw& \text{if $F$ is of depth $0$}\\ \displaystyle\int_{\widetilde{\Gamma}_0(N)\backslash\Tt_1} F(w)\overline{\Theta^{N}_{D,1}(w;z)}dw& \text{if $F$ is of depth $1$}\end{cases} \]
\end{definition*}

The main result of this section is the following theorem.\\
\begin{thm}\label{Niwa}
Let $F(w)\in \widetilde{S}_0^{new}(\Gamma_0(N))$ or $\widetilde{S}_1^{new}(\Gamma_0(N))$ be a cuspidal metaplectic form and $D$ a square-free polynomial which is relatively prime to $N.$ Let $\phi(g)$ be the $D'th$ Shimura lift of $F(w)$. Then we have the following equalities of L-functions:
$$L(s,\phi) =\begin{cases} p^{-1/2}L^N(1+s,\chi_D)|D|^{3/4}L(s-1/2,F,D)&\text{if $F$ is of depth $0$}\\  \frac{\lambda_{\infty,F}p^{-1/2}}{p+1}L^N(1+s,\chi_D)|D|^{3/4}L(s-1/2,F,D)&\text{if $F$ is of depth $1$}\end{cases}$$
Where $L^N(s,\chi_D) = \displaystyle\prod_{P\nmid N}\left(1-\frac{\left(\tfrac{P}{D}\right)}{|P|^{s}}\right)^{-1}	$.
\end{thm}

For clarity of exposition and convenience of the reader, we will first prove Niwa's Lemma for the case $N=1$ and forms unramified at infinity, and then outline the necessary modifications for the general level and ramification.

\begin{proof}

During the proof unless otherwise stated we will use (as usual)
\begin{align*}
w&=(u,v)=\left(\left(\begin{smallmatrix}\sqrt{v}&u/\sqrt{v}\\0&1/\sqrt{v}\end{smallmatrix}\right),1\right)\in\Tt\\
z&=(x,y)=\left(\begin{smallmatrix}y&x\\0&1\end{smallmatrix}\right)
 \in \mathbb{H}.
\end{align*}

\textbf{Case of N=D=1, depth 0}

We will calculate the Mellin transform of $\phi$ and show that it is the Mellin transform of the form in the Theorem.

 \[M\phi(s)=\int_{k_{\infty}^{\times}}\phi(0,y)|y|^sd^{\times}y\]
 \[=\int_{k_{\infty}^{\times}}\left(\int_{\widetilde{\Gamma}\backslash\Tt} F(w)\overline{\Theta(w,(0,y))}\frac{du\,dv}{|v|^2}\right)|y|^sd^{\times}y\]

The theta function factors as:
 \[\Theta(w,(0,y))=\tau(v)|v|^{1/4}\sum_{h_2\in TR}e(uh_2^2)\chi_{\Oo_{\infty}}(vh_2^2)\;\;\times|v|^{1 /2}\sum_{h_1,h_3\in TR}e(-4h_1h_3u)\chi_{\Oo_{\infty}}\left(\frac{\sqrt{v}}{y}h_3\right)\chi_{\Oo_{\infty}}(\sqrt{v}y h_1)\]

 \[=\Theta(w)\theta_2(w,y)\]

We now use Poisson summation in the $h_3$ variable in order to separate the $h_1$ and $h_3$ variables in $\theta_2$. Define
 \[f(\alpha)=\sum_{h_1\in TR}e(-4h_1u\alpha)\chi_{\Oo_{\infty}}\left(\frac{\sqrt{v}}{y}\alpha\right)\chi_{\Oo_{\infty}}(\sqrt{v}y h_1)\]

The Fourier transform of $f$ is
\begin{align*}
\hat{f}(\beta)&=\int_{k_{\infty}}f(\alpha)e(\beta \alpha)d\alpha\\
 &=\sum_{h_1\in TR}\chi_{\Oo_{\infty}}\left(\sqrt{v}yh_1\right)
\int_{k_{\infty}}e(-4h_1u\alpha+\alpha\beta)\chi_{\Oo_{\infty}}\left(\frac{\sqrt{v}}{y} \alpha\right)d\alpha\\
&=\sum_{h_1\in TR}\chi_{\Oo_{\infty}}\left(\sqrt{v}yh_1\right)\int_{\Oo_{\infty}}e\left(\frac{y(-4h_1u+\beta)\alpha}{\sqrt{v}}\right)\frac{|y|d\alpha}{|\sqrt{v}|}\\
&=\frac{|y|}{|\sqrt{v}|}\sum_{h_1\in TR}\chi_{\Oo_{\infty}}(\sqrt{v}{y}h_1)\chi_{\Oo_{\infty}}\left(\frac{y(-4h_1u+\beta)}{\sqrt{v}}\right)
\end{align*}

By Poisson summation we get
\[\theta_2(w,y)=|y|\sum_{h_1,h_3\in TR}\chi_{\Oo_{\infty}}(\sqrt{v}{y}h_1)\chi_{\Oo_{\infty}}\left(\frac{y(-4h_1u+h_3)}{\sqrt{v}}\right)\]

Define $I=\int_{k_{\infty}^{\times}}\overline{\theta_2(w,y)}|y|^sd^{\times}y.$ Then

\[ I=\sum_{h_1,h_3\in TR}\int_{k_{\infty}^{\times}}\chi_{\Oo_{\infty}}(\sqrt{v}h_1y)\chi_{\Oo_{\infty}}\left(\frac{(4h_1u-h_3)y}{\sqrt{v}}\right)|y|^{s+1}d^{\times}y\]

Let $\alpha=\frac{(4h_1u-h_3)y}{\sqrt{v}}$, then
\[I=\sum_{h_1,h_3\in TR}\frac{|v|^{(s+1)/2}}{|4h_1u-h_3|^{s+1}}\int_{k_{\infty}^{\times}}\chi_{\Oo_{\infty}}\left(\frac{vh_1\alpha}{4h_1u-h_3}\right)\chi_{\Oo_{\infty}}(\alpha)|\alpha|^{s+1}d^{\times}\alpha\]

\[=\sum_{h_1,h_3\in TR}\frac{|v|^{(s+1)/2}}{|4h_1u-h_3|^{s+1}}\int_{\Oo_{\infty}\backslash\{0\}}\chi_{\Oo_{\infty}}\left(\frac{vh_1\alpha}{4h_1u-h_3}\right)|\alpha|^{s+1}d^{\times}\alpha\]

Let $v_{\infty}\left(\frac{vh_1}{4h_1u-h_3}\right)=A$, then
\[\int_{\Oo_{\infty}\backslash\{0\}}\chi_{\Oo_{\infty}}\left(\frac{vh_1\alpha}{4h_1u-h_3}\right)|\alpha|^{s+1}d^{\times}\alpha=\sum_{j=\max\{-A,0\}}^{\infty}\int_{\Oo_{\infty}^{\times}}p^{-j(s+1)}d^{\times}\alpha\]
\[=\Gamma_k(s+1)p^{\min\{A,0\}(s+1)}\]

So we get
\begin{align*}
I&=\Gamma_k(s+1)\sum_{h_1,h_3\in TR}\frac{|v|^{(s+1)/2}}{|4h_1u-h_3|^{s+1}}p^{(s+1)\min\{v_{\infty}\left(\frac{vh_1}{4h_1u-h_3}\right),0\}}\\
&=\Gamma_k(s+1)\sum_{h_1,h_3\in TR}\frac{|v|^{(s+1)/2}}{|4h_1u-h_3|^{s+1}}\min\left\{\frac{|4h_1u-h_3|^{s+1}}{|vh_1|^{s+1}},1\right\}\\
&=\Gamma_k(s+1)\sum_{h_1,h_3\in TR}\frac{|v|^{(s+1)/2}}{\max\{|vh_1|,|-4h_1u+h_3|\}^{s+1}}\\
&=\frac{\Gamma_k(s+1)\zeta_k(s+1)}{p^{s+1}}E(w,(s+1)/2)
\end{align*}

Where the Eisenstein series is defined as in the usual way
\[E(w,s)=\sum_{\gamma\in\widetilde{\Gamma}_{\infty}\backslash\widetilde{\Gamma}}Im(\gamma w)^s=
\sum_{\substack{c,d\in R\\ (c,d)=1}}\frac{|v|^s}{\max\{|cv|,|cu+d|\}^{2s}}.\] This is absolutely convergent for $Re(s)>1$ and evidently left-invariant by $\widetilde{\Gamma}$.

Putting things together we have
\begin{equation}\label{mellin}
M\phi(s)=\frac{\Gamma_k(s+1)\zeta_k(s+1)}{p^{s+1}}\int_{\widetilde{\Gamma}\backslash\Tt}F(w)\overline{\Theta(w)}E(w,(s+1)/2)\frac{du\,dv}{|v|^2}\end{equation}

Define $J =\int_{\widetilde{\Gamma}\backslash\Tt}F(w)\overline{\Theta(w)}E(w,(s+1)/2)\frac{du\,dv}{|v|^2}.$ Then
\begin{align*}
J&=\int_{\widetilde{\Gamma}\backslash\Tt}F(w)\overline{\Theta(w)}E(w,(s+1)/2)\frac{du\,dv}{|v|^2}\\
&=\int_{\widetilde{\Gamma}_{\infty}\backslash\Tt}F(w)\overline{\Theta(w)}|v|^{(s+1)/2}\frac{du\,dv}{|v|^2}
\end{align*}

where the second equality is obtained by unfolding the Eisenstein series..
%Now, since $F(z)$ vanishes for $v_{\infty}(v)\leq 0,$ we get
\begin{align*}
J&=\int_{\widetilde{\Gamma}_{\infty}\backslash\Tt}F(w)\overline{\Theta(w)}|v|^{(s+1)/2}\frac{du\,dv}{|v|^2}\\
&=2\iint\limits_{\substack{v\in k_{\infty}^{\times} \\ u\in T^{-1}\Oo_{\infty}}}\tau(v)|v|^{s/2-5/4}\left(\sum_{a\in R}F_a(v)e(T^2au)\right)\overline{\left(\sum_{l\in R}e(T^2l^2u)\chi_{\Oo_{\infty}}(T^2l^2v)\right)}du\,dv
\end{align*}

The inner integral over $u$ vanishes unless $a=l^2,$ so we get
\begin{align*}
J&=\sum_{l\in R}\frac{\lambda_F(l^2)}{|l|^{s-1/2}}\int_{v\in k_{\infty}^{\times}}|v|^{s/2-1/4}\widetilde{W}_{0,i\gamma}(v)d^{\times}v\\
&=p^{-s+1/2}\Gamma_k(s-1/2+i\xi)\Gamma_k(s-1/2-i\xi)\Gamma^{-1}(1+s)L(s-1/2,F)
\end{align*}

Where $p^{i\xi}=e^{i\gamma}$.

Plugging $J$ back in we get,
\[M\phi(s) =  p^{-2s-1/2}\zeta_k(s+1)\Gamma_k(s-1/2+i\xi)\Gamma_k(s-1/2-i\xi)L(s-1/2,F)\]

which completes the proof in the case $N=D=1$ and depth $0.$\\

\textbf{Case of N=D=1, depth 1}
We first introduce the following modified Eisenstein series on $\Tt_1$:
 $$E_1(w,s) :=\sum_{\gamma\in\widetilde{\Gamma}_{\infty}\backslash\widetilde{\Gamma}}Im\left(\gamma w\left(\begin{smallmatrix} \varpi_{\infty}^{-1} & 0\\0 & 1 \end{smallmatrix}\right)\right)^s $$
(We note that since $\left(\begin{smallmatrix}\varpi_{\infty}^{-1} & 0\\0 & 1\end{smallmatrix}\right)^{-1}\widetilde{K}_0(\varpi_{\infty})\left(\begin{smallmatrix}\varpi_{\infty}^{-1}& 0\\0 & 1\end{smallmatrix}\right)\subset SL_2(\Ooi)$ we see that $E_1(w,s)$ is well defined and of depth 1.)

Let $F(w)$ be a newform of depth 1, and define $\phi(z)$ to be its Shimura lift. Then following the above computations replacing $\Theta(w;z)$ by
$\Theta_1(w;z)$ we arrive at the following analogue of (\ref{mellin}):

\begin{equation}\label{meldep}
M\phi(s)=\frac{\Gamma_k(s+1)\zeta_k(s+1)}{p^{s+1}}p^{-\frac{s+1}{2}}\int_{\widetilde{\Gamma}\backslash\Tt_1}F(w)\overline{\Theta_1(w)}E_1(w,(s+1)/2)\frac{du\,dv}{|v|^2}
\end{equation}

Define $J =\int_{\widetilde{\Gamma}\backslash\Tt_1}F(w)\overline{\Theta_1(w)}E_1(w,(s+1)/2)\frac{du\,dv}{|v|^2}.$ Then
\begin{align*}
J&=\int_{\widetilde{\Gamma}\backslash\Tt_1}F(w)\overline{\Theta_1(w)}E_1(w,(s+1)/2)\frac{du\,dv}{|v|^2}\\
&=\frac{1}{p+1}\int_{\widetilde{\Gamma}_{\infty}\backslash\Tt^u_1}F(w)\overline{\Theta_1(w)}p^{(s+1)/2}|v|^{(s+1)/2}\frac{du\,dv}{|v|^2}\\
&+\frac{p}{p+1}\int_{\widetilde{\Gamma}_{\infty}\backslash\Tt^l_1}F(w)\overline{\Theta_1(w)}p^{-(s+1)/2}|v|^{(s+1)/2}\frac{du\,dv}{|v|^2}
\end{align*}

Using the fact that $F_a^l(v)=-\frac{F_a^u(v)}{p}$, we Fourier expand as before and get the following:
\begin{align*}
J &=\frac{1}{p+1} (p^{\frac{s+1}{2}} - p^{-\frac{1+s}{2}})\sum_{l\in R}\frac{\lambda_F(l^2)}{|l|^{s-1/2}}\int_{v\in T^{-1}\Ooi\backslash \{ 0 \} }|v|^{s/2-1/4}
\widetilde{W}_{1,\lambda_{\infty,F}}(v)d^{\times}v\\
 &= \frac{1}{p+1}p^{\frac{1+s}{2}}\Gamma_k^{-1}(1+s)L(s-\frac{1}{2},F)\frac{\lambda_{\infty,F}\:p^{1/2-s}}{1-\lambda_{\infty,F}\:p^{1/2-s}}\\
\end{align*}

Plugging this expression back into (\ref{meldep}), we get

\[M\phi(s)=\frac{\lambda_{\infty,F}\:p^{-1/2}}{p+1}L(s-\frac{1}{2},F)\zeta_k(s+1)\frac{p^{-2s}}{1-\lambda_{\infty,F}\:p^{1/2-s}}\]

as desired.\\

\textbf{General case, depth 0}

We now do the above computations but with the general theta function $\Theta^N_D(w;g).$

As before, we shall compute the Mellin transform \[M\phi(s)=\int_{k_{\infty}^{\times}} \phi(0,y)|y|^{s} d^{\times}y. \]

We shall need the theta functions $$\theta(D,w,h) = \tau(v)\displaystyle\sum_{\substack{m\in TR,\\ m\equiv h\;\text{mod $D$}}} e\left(\frac{um^2}{D}\right)\chi_{\Ooi}\left(\frac{vm^2}{D}\right)$$
where $h\in R/DR.$ We begin by writing:
\[\Theta^N_D(w;g) = \tau(v)\displaystyle\sum_{\substack{\vec{h}\in L,\;N|h_1\\ D\mid Q(\vec{h})}}W_D(\vec{h})e((h_2^2 - 4h_1h_3)u/D)\chi_{O_V}(\sqrt{v}g^{-1}(h)/\sqrt{D})\]

Splitting over representatives $\vec{r}$ defined modulo $D$ we get
\begin{multline*}
\Theta^N_D(w;(0,y)) = \tau(v)|v|^{3/4}\sum_{\substack{\vec{r}\in TR^3/DTR^3\\ r_2^2-4r_1r_3\equiv 0\:(D)}}W_D(\vec{r})\theta(D,w,r_2)\\
\times\sum_{\substack{h_1\equiv r_1\;(D)\\h_3\equiv r_3\;(D) \\NT|h_1\\T|h_3}}e\left(\frac{-4h_1h_3u}{D}\right)\chi_{\Ooi}\left(\frac{\sqrt{v}h_3}{y\sqrt{D}}\right)
\chi_{\Ooi}\left(\frac{\sqrt{v}y h_1}{\sqrt{D}}\right)
\end{multline*}

Applying Poisson summation over $h_3$, we get (we denote the dual variable by $s_3$ which should not be confused with the exponent $s$ of the Mellin transform)

\begin{multline*}
\Theta^N_D(w;(0,y)) = \frac{\tau(v)|y||v|^{3/4}}{|\sqrt{vD}|}\sum_{\substack{\vec{r}\in TR^3/DTR^3\\ r_2^2-4r_1r_3\equiv 0\:(D)}}W_D(\vec{r})\theta(D,w,r_2)\\
\times\sum_{\substack{h_1\equiv r_1\;(D)\\ NT|h_1\\T|s_3}}e\left(\frac{-s_3r_3}{D}\right)\chi_{\Ooi}\left(\frac{\sqrt{v}yh_1}{\sqrt{D}}\right)
\chi_{\Ooi}\left(\frac{y(-4h_1u+s_3)}{\sqrt{vD}}\right)
\end{multline*}

So that

\begin{multline*}
\int_{k^{\times}_{\infty}}\Theta^N_D(w;(0,y))|y|^s d^{\times}y = \Gamma_k(1+s)|D|^{s/2}\tau(v)|v|^{1/4}
\sum_{\substack{\vec{r}\in TR^3/DTR^3\\ r_2^2-4r_1r_3 \equiv 0\:(D)}}W_D(\vec{r})\theta(D,w,r_2)\\
\times\sum_{\substack{h_1\equiv r_1\;(D)\\ NT|h_1\\T|s_3}}e\left(\frac{-s_3r_3}{D}\right)\frac{|v|^{\frac{1+s}{2}}}{\max\{|vh_1|,|-4h_1u+s_3|\}^{1+s}}
\end{multline*}

Define $G_{N,D,r_2}(w,s)$ by
$$G_{N,D,r_2}(w,s)=\sum_{\substack{r_1,r_3\in R/DR\\4r_1r_3\equiv r_2^2\;(D)}}W_D(\vec{r})\sum_{\substack{h_1\equiv r_1\:(D)\\ NT|h_1\\T|s_3}}\left(\frac{-s_3r_3}{D}\right)\frac{|v|^{\frac{1+s}{2}}}{\max\{|vh_1|,|-4h_1u+s_3|\}^{1+s}}$$
We now express $G_{N,D,r_2}(w,s)$ in terms of Eisenstein series. Specifically, we define the following congruence subgroups of $\eta(SL_2(R)):$
$$\widetilde{\Gamma}_1(D)=\left\{\eta(\gamma)\in \eta(SL_2(R))\;\left|\;\gamma \equiv \begin{pmatrix} 1 & *\\ 0 & 1\end{pmatrix}\mod D\right.\right\}$$ and $\widetilde{\Gamma}_{N,D}=\widetilde{\Gamma}_0(N)\cap\widetilde{\Gamma}_1(D).$

Now define the Eisenstein series via $$E_{N,D}(w,s) = \sum_{\gamma\in\widetilde{\Gamma}_{\infty}\backslash\widetilde{\Gamma}_{N,D}}Im(\Gamma_k(w))^s $$
Then for $\gamma_{c_0,d_0}=\eta\left(\left(\begin{smallmatrix} * & *\\ c_0 & d_0 \end{smallmatrix}\right)\right)\in SL_2(R),$ we readily compute
$$E_{N,D}(\gamma_{c_0,d_0}(w),(1+s)/2) = \sum_{\substack{c\equiv c_0\;\text{mod $D$}\\ d\equiv d_0\;\text{mod $D$}\\ (c,d)=1}} \frac{|v|^{\frac{1+s}{2}}}{\max\{|vc|,|cu+d|\}^{1+s}}$$
Though $E_{N,D}(w,s)$ looks similiar to $G_{N,D,r_2}(w,s),$ we have to deal with the fact that $h_1$ and $s_3$ might not be relatively prime. In order to separate the gcd from the sum we define the following function
\[\tilde{G}_{N,D,r_2}(w,s)=\sum_{\substack{r_1,r_3\in R/DR\\r_1r_3\equiv r_2^2\;(D)}}W_D(\vec{r})\sum_{\substack{h_1\equiv r_1\:(D)\\ NT|h_1\\(s_3,N)=1\\T|s_3}}e\left(\frac{-s_3r_3}{D}\right)\frac{|v|^{\frac{1+s}{2}}}{\max\{|vh_1|,|-4h_1u+s_3|\}^{1+s}}\]
The following Lemma separates the factors of $DN$ from the sum.

\begin{lemma}\label{sieve}
\[G_{N,D,r_2}(w,s)=\sum_{\substack{r_1,r_3\in R/DR\\r_1r_3\equiv r_2^2\;(D)}}W_D(\vec{r})\sum_{\substack{h_1\equiv r_1\:(D)\\ NT|h_1\\((r_1,s_3),D)=1\\T|s_3}}e\left(\frac{-s_3r_3}{D}\right)\frac{|v|^{\frac{1+s}{2}}}{\max\{|vh_1|,|-4h_1u+s_3|\}^{1+s}}\]
and
\[G_{N,D,r_2}(w,s)=\sum_{d\mid N}d^{-s-1}\left(\frac{d}{D}\right)\tilde{G}_{\tfrac{N}{d},D,r_2}(w,s)\]

\end{lemma}
\begin{proof} Let $((r_1,s_3),D)=b$. Since $D$ is square-free, $D$ and $D/b$ are relatively prime. Hence by the Chinese Remainder Theorem we have $$R/DR=R/bR\oplus R/(D/b)R.$$ Using this decomposition we separate the $r_3$ sum into $r_3\;\text{mod $b$}$ and $r_3\;\text{mod $(D/b)$}$, and $W_D$ splits accordingly as $W_D=W_{(D/b)}W_b$. Since $b\mid s_3$ the character $e(-s_3r_3/D)$ is constant on $R/bR$ and the sum over $r_3\;\text{mod $b$}$ becomes
\[\sum_{r_3\; (b)}W_b(r_3)=\sum_{r_3}\left(\frac{b}{r_3}\right)=0\]
This establishes the first statement.

For the second statement let $c=(s_3,N)$. The substitution $\{h_1\rightarrow \tfrac{h_1}{c},\;s_3\rightarrow
\tfrac{s_3}{c},\; r_1\rightarrow \tfrac{r_1}{c},\; r_3\rightarrow cr_3\}$ leaves the character $e(-r_3s_3/D)$ invariant.  Note that with the flip, $W_D$ changes
by the quadratic character: $$W_D((r_1,r_2,r_3))\rightarrow W_D((r_1/b,r_2,br_3))=\left(\tfrac{b}{D}\right)W_D((r_1,r_2,r_3).$$

Summing over all such $c$ gives
%Then by the inclusion-exclusion principle we have
\[\tilde{G}_{N,D,r_2}(w,s)=\sum_{d\mid N}\mu(d)d^{-s-1}\left(\frac{d}{D}\right)G_{\tfrac{N}{d},D,r_2}(w,s)\] as desired.
%Applying M\"obius inversion to the above sum we get
%\[G_{N,D,r_2}(w,s)=\sum_{d\mid N}d^{-s-1}\left(\frac{D}{d}\right)\tilde{G}_{\tfrac{N}{d},D,r_2}(w,s)\]
\end{proof}

Now for $((h_1,s_3),ND)=c\neq 1$ the same substitution  $\{h_1\rightarrow \tfrac{h_1}{c},\;s_3\rightarrow
\tfrac{s_3}{c},\; r_1\rightarrow \tfrac{r_1}{c},\; r_3\rightarrow cr_3\}$ as in the proof above gives
\begin{multline*}
\tilde{G}_{N,D,r_2}(w,s)=\prod_{P\nmid DN}\left(1-\frac{\left(\tfrac{P}{D}\right)}{|P|^{1+s}}\right)^{-1}\\ \times\sum_{\substack{r_1,r_3\in R/DR\\
r_1r_3\equiv r_2^2\;(D)}}W_D(\vec{r})\sum_{\substack{h_1\equiv r_1\:(D)\\ NT|h_1 \\ (h_1,s_3)=T}}e\left(\frac{-s_3r_3}{D}\right)\frac{|v|^{\frac{1+s}{2}}}{\max\
{|vh_1|,|-4h_1u+s_3|\}^{1+s}}}
\end{multline*}

Note that the first factor is $L^N(s+1,\chi_D) = \displaystyle\prod_{P\nmid N}\left(1-\frac{\left(\tfrac{P}{D}\right)}{|P|^{1+s}}\right)$, and so the above simplifies to
\begin{equation}\label{eisensteinappears}
\tilde{G}_{N,D,r_2}(w,s)=L^N(1+s,\chi_D)\sum_{\substack{r_1,r_3\in R/DR\\
r_1r_3\equiv r_2^2\;(D)}}W_D(\vec{r})\sum_{s_3\in R/DR} p^{-(1+s)}E_{N,D}(\gamma_{-4r_1,s_3}(w),(1+s)/2)e\left(\frac{-s_3r_3}{D}\right)
\end{equation}

We now compute the Mellin transform of $\phi((0,y))$ as intended.
\begin{align*}
\overline{M\phi(\overline{s})} &= p^{-(1+s)}\int_{k_{\infty}^{\times}}\int_{\widetilde{\Gamma}_0(N)\backslash\Tt}\overline{F(w)} \Theta_D^N(w;(0,y))|y|^s dwdy\\
&=p^{-(1+s)} \Gamma_k(1+s)|D|^{s/2}\int_{\widetilde{\Gamma}_0(N)\backslash\Tt}|v|^{1/4}\overline{F(w)}\sum_{r_2}\theta(D,w,r_2)G_{N,D,r_2}(s)dw\\
&= p^{-(1+s)}\Gamma_k(1+s)|D|^{s/2}\int_{\widetilde{\Gamma}_0(N)\backslash\Tt}|v|^{1/4}\overline{F(w)}\sum_{r_2}\theta(D,w,r_2)\sum_{d\mid N}d^{-s-1}\left(\frac{d}{D}\right)\tilde{G}_{\tfrac{N}{d},D,r_2}(s)dw\\
\end{align*}

Now, for $d\neq 1, d|N $, $\theta(D,w,r_2)E_{\frac{N}{d},D}(w,s)$ is invariant under $\widetilde{\Gamma}(D)\cap\widetilde{\Gamma}(\frac{N}{d})$  which is not contained inside $\widetilde{\Gamma}_0(N)$.  However, $F(w)$ is a newform for $\widetilde{\Gamma}_0(N)$ and so is orthogonal to any function invariant by a bigger group
 as it would then be in the old space. Thus the above integral vanishes unless $d=1.$ So
\begin{align*}
\overline{M\phi(\overline{s})} &= p^{-(1+s)}\Gamma_k(1+s)|D|^{s/2}\int_{\widetilde{\Gamma}_0(N)\backslash\Tt}|v|^{1/4}\overline{F(w)}\sum_{r_2\in(R/DR)}\theta(D,w,r_2)\tilde{G}_{N,D,r_2}(s)dw\\
&= p^{-(1+s)}\Gamma_k(1+s)|D|^{s/2}\int_{\widetilde{\Gamma}_0(N)\backslash\Tt}|v|^{1/4}\overline{F(w)}\sum_{\vec{r}\in(R/DR)^3}W_D(\vec{r})\theta(D,w,r_2)L^N(1+s,\chi_D)\\
&\times\sum_{s_3\in R/DR} E_{N,D}(\gamma_{-4r_1,s_3}(w),(1+s)/2)e\left(\frac{-s_3r_3}{D}\right)dw\\
&=p^{-(1+s)}\Gamma_k(1+s)|D|^{s/2}[\widetilde{\Gamma}_0(N):\widetilde{\Gamma}_{N,D}]^{-1}\\
&\times\int_{\widetilde{\Gamma}_{N,D}\backslash\Tt}|v|^{1/4}\overline{F(w)}\sum_{\substack{\vec{r}\in(R/DR)^3\\r_2^2-4r_1r_3=0}}W_D(\vec{r})\theta(D,w,r_2)L^N(1+s,\chi_D)\\
&\times\sum_{s_3\in R/DR} E_{N,D}(\gamma_{-4r_1,s_3}(w),(1+s)/2)e\left(\frac{-s_3r_3}{D}\right)dw\\
\end{align*}

Since $(N,D)=1$, we can pick representatives for $r_1,s_3$ such that $\gamma_{r_1,s_3}\in\widetilde{\Gamma}_0(N).$ Then $F(w)$ is invariant under the change of variables
$w\rightarrow\gamma_{-4r_1,s_3}^{-1}(w).$ As in \cite{Shim} we have the following transformation property:

$$\theta(D,\gamma^{-1}_{-4r_1,s_3}(w),r_2) = |D|^{-1/2}i(u/d)\sum_{\nu\in R/DR} C_{r_2,\nu}\theta(D,w,\nu)$$

where $C_{r_2,\nu}$ are constants and $C_{r_2,0} = \left(\frac{r_1}{D}\right)e\left(\frac{s_3(4r_1)^{-1}r_2^2}{D}\right) = W_D(\vec{r})e\left(\frac{s_3r_3}{D}\right).$ Now, since $E_{N,D}(w,s)$ and $F(w)$ are both invariant under $w\rightarrow w+1,$ and $\theta(D,w+1,a)=e(\frac{a^2}{D})\theta(D,w,a),$ everything but the $\theta(D,w,0)$ contribution vanishes. Therefore, we are left with

\begin{align*}
\overline{M\phi(\overline{s})} &=p^{-(1+s)}\Gamma_k(1+s)    L^N(1+s,\chi_D)|D|^{\frac{s+1}{2}}[\Gamma_0(N):\Gamma_{N,D}]^{-1}\kappa(D)\\
&\times\int_{\widetilde{\Gamma}_{N,D}\backslash\widetilde{\mathbb{H}}}|v|^{1/4}\overline{F(w)}\theta(D,w,0)
 E_{N,D}(w,(1+s)/2)dw\\
\end{align*}

where $\kappa(D)=\#\{\vec{r}\mid W_D(\vec{r})\neq 0\}$. A simple computation yields
\[[\widetilde{\Gamma}_0(N):\widetilde{\Gamma}_{N,D}]^{-1}\kappa(D) = 1.\]

Unfolding the Eisenstein series, we get

\[\overline{M\phi(\overline{s})} =p^{-(1+s)}\Gamma_k(1+s)L^N(1+s,\chi_D)|D|^{\frac{1+s}{2}}\int_{\widetilde{\Gamma}_{\infty}\backslash\Tt}\overline{F(w)}\theta(D,w,0)|v|^{s/2+3/4}\frac{du\,dv}{|v|^2}\]

Unfolding as before leads to

\[M\phi(s) =p^{-2s-1/2}L^N(1+s,\chi_D)|D|^{3/4}\Gamma_k(s+1/2-i\xi)\Gamma_k(s+1/2+i\xi)L_D(s-1/2,F)\]

from which the result follows. \\

\textbf{General Case, depth 1}

The above calculations go through, except $$E_{N,D}(w,s) = \sum_{\gamma\in\widetilde{\Gamma}_{\infty}\backslash\widetilde{\Gamma}_{N,D}}Im(\gamma(w))^s $$
is replaced by $$E_{1,N,D}(w,s) = \sum_{\gamma\in\widetilde{\Gamma}_{\infty}\backslash\widetilde{\Gamma}_{N,D}}Im\left(\gamma(w)\left(\begin{smallmatrix} \varpi_{\infty}^{-1} & 0\\ 0 & 1\end{smallmatrix}\right)\right)^s$$ as in the case where $N=D=1$.

As before, $E_{1,N,D}(w,s)$ is a depth 1 function. Following the above calculations gives the result.	

\end{proof}

\section{A Waldspurger type formula}\label{sectionwaldspurger}
\subsection{Qualitative facts about the Shimura correspondence}\label{qualwalds}
We will need two general facts about the correspondences that we have defined:  that they preserve cuspidality and Hecke eigenspaces. To establish these we use the fact that locally, the correspondences that we defined (for any $D\in\mathbb{F}_p[T]$ that we are considering) are instances of a general
correspondence known as the \emph{local Howe correspondence}. To define the correspondence we start with two closed subgroups $H$ and $G$ of a symplectic group that are maximal commutants of each other. (Such pairs of groups are called \emph{dual reductive pairs}.) Denote their inverse image in the metaplectic cover of the symplectic cover by $\widetilde{H}$ and $\widetilde{G}$ respectively. Then a representation $\pi_{\widetilde{H}}$ of $\tilde{H}$ is said to correspond to a representation $\pi_{\widetilde{G}}$ of $\widetilde{G}$ if the tensor product $\pi_{\widetilde{H}}\otimes\pi_{\widetilde{G}}$ can be realized as a quotient of the Weil representation of the symplectic group that $H$ and $G$ sit in. Howe's conjectures (proved for archimedean local fields by Howe \cite{H} and for nonarchimedean local fields of characteristic not $2$ by Waldspurger \cite{Wa}) roughly states that for any representation $\pi_{\widetilde{G}}$ of $\widetilde{G}$ there is at most one irreducible representation $\pi_{\widetilde{H}}=\theta(\pi_{\widetilde{G}})$ of $\widetilde{H}$ that it corresponds to, and if $\theta(\pi_{\widetilde{G}})\neq0$ and $\theta(\pi_{\widetilde{G}})\cong\theta(\pi_{\widetilde{G}}')$ then $\pi_{\widetilde{G}}\cong\pi_{\widetilde{G}}'$. We will only need the first part of the statement which gives a local multiplicity one result. The case that is pertinent here is that of $H=PGL(2)$ and
$G=\widetilde{SL}(2)$.  Translated to the language of the present paper, the Howe correspondence, combined with the strong multiplicity one theorem for $PGL(2)$,  says that if $F$ is a Hecke eigenform on $\widetilde{SL}(2)$, than so is its Shimura lift $\phi$. So the Shimura lift
preserves Hecke eigenspaces.

In \cite{G-P-S}, Gelbart and Piatetski-Shapiro proved cuspidality results for general theta lifts (cf. Theorem 15.1 and Corollary 15.6), which in our case gives:
\begin{itemize}
\item[(i)]  If $F$ is a cuspidal Hecke eigenform on $\widetilde{SL}(2)$ that is not a theta function of 1-variable, then its ($D$'th) Shimura lift is also cuspidal.
\item[(ii)]  If $\phi$ is a cuspidal Hecke eigenform on $PGL(2)$, then its ($D$'th) Maass-Shintani lift is also cuspidal.
\end{itemize}
In what follows we shall also use these two facts.

\subsection{Waldspurger's formula}

We are now ready to deduce Waldspurger's formula. We start with the depth 0 case. Pick an Hecke-eigenbasis for $\widetilde{S}_0(\widetilde{\Gamma}_0(N))$: $\{F_1,F_2,..,F_r\}$, which
 is orthonormal for the Petterson inner product. Let $N=\prod_{i=1}^m l_i^{\alpha_i}$ be the decomposition of $N$ into irreducible prime powers. Let $\phi(z)$ be an
automorphic Hecke eigenform of depth 0 of level $N$, with first Fourier-Whittaker coefficient $\lambda_{\phi}(1)$ equal to 1. Let $F^{[D]}(w)$ be the $D$'th Maass-Shintani lift of $\phi$. We are going to write the $D$-th Fourier-Whittaker coefficient $\lambda_{F^{[D]}}(D)$
of $F$ in 2 different ways. On the one hand by Theorem \ref{shintanilift} we have

\[ \lambda_{F^{[D]}}(D) =  C_{\phi}G_1(D)|D|^{-3/4}L(1/2,\phi\times\chi_D)\prod_{i=1}^m \left(1+\left(\frac{l_i^{\alpha_i}}{D}\right)w_{l_i^{\alpha_i},\phi}\right)(1+w_{\infty,\phi})\]

On the other hand, since we know that $F^{[D]}$ is cuspidal, we can expand it as
\[F^{[D]}(w)=\sum_j\langle F^{[D]},F_j\rangle\; F_j(w)\]
We also have by definition
\begin{multline*}
\langle F^{[D]},F_j\rangle=\int_{\widetilde{\Gamma}_0(N)\backslash \widetilde{\mathbb{H}}}F^{[D]}(w)\overline{F}_j(w)dw=\int_{\widetilde{\Gamma}_0(N)\backslash \widetilde{\mathbb{H}}}\int_{\Gamma_0(N)\backslash \mathbb{H}}\Theta_D^N(w;z)\phi(z)\overline{F}_j(w)dzdw\\ =\int_{\Gamma_0(N)\backslash \mathbb{H}}\phi(z)\overline{\phi}_j(z)dz=\langle \phi,\phi_j\rangle
\end{multline*}
where $\phi_j$ is the $D$'th Shimura lift of $F_j$. Since $\phi_j$ is a Hecke eigenform, by strong multiplicity one it is either a constant multiple of $\phi$ or orthogonal to it. If it is a constant multiple $M_j$ of $\phi$, then $$\langle\phi_j,\phi\rangle = M_j\langle\phi,\phi\rangle,$$ and thus

\begin{align*}
C_{\phi}G_1(D)|D|^{-3/4}L(1/2,\phi\times\chi_D)\prod_{i=1}^m \left(1+\left(\frac{l_i^{\alpha_i}}{D}\right)w_{l_i^{\alpha_i},\phi}\right)(1+w_{\infty,\phi})&=\lambda_F(D)\\
&=\sum_j \langle F,F_j\rangle\lambda_{F_j}(D)\\
&=\sum_j\overline{M}_j\langle\phi,\phi\rangle\lambda_{F_j}(D)\\
\end{align*}

where $C_{\phi} = \frac{p^{-3/2}}{\left(1-\frac{e^{-i\theta}}{\sqrt{p}}\right)\left(1-\frac{e^{i\theta}}{\sqrt{p}}\right)}$.

By Niwa's Lemma, if $\lambda_{F_j}(D)\neq 0,$ then $M_j=p^{-1/2}|D|^{3/4}\lambda_{F_j}(D)/\lambda_{\phi}(1)$. Putting this together, we arrive at:\\
\begin{thm}[Waldspurger's formula for depth 0]

Let $\phi\in S_0(\Gamma_0(N))$ with $D$ coprime to $N$, squarefree and $D\in k_{\infty}^2$. Then we have the following identity:
$$\frac{1}{\langle\phi,\phi\rangle} p^{1/2}C_{\phi}G_1(D)|D|^{-3/2}L(1/2,\phi\times\chi_D)\prod_{i=1}^m \left(1+\left(\frac{l_i^{\alpha_i}}{D}\right)w_{l_i^{\alpha_i},\phi}\right) = \sum_{\textrm{ Shim}(F_j) = \phi} |\lambda_{F_j}(D)|^2$$ where $C_{\phi} = \frac{p^{-3/2}}{\left(1-\frac{e^{-i\theta}}{\sqrt{p}}\right)\left(1-\frac{e^{i\theta}}{\sqrt{p}}\right)}.$	
\end{thm}

For the depth 1 case, we pick an orthonormal Hecke eigenbasis for $\widetilde{S}^{new}_0(\widetilde{\Gamma}_0(N),1)$.  Now repeating the argument
above,we arrive at:\\

\begin{thm}[Waldspurger's formula for depth 1]

Let $\phi\in S^{new}_1(\Gamma_0(N))$ with $D$ coprime to $N$, square-free, and $D\in k_{\infty}^2$. Then we have the following identity:
$$\frac{1}{\langle\phi,\phi\rangle}\frac{(p+1)}{p+w_{\infty,\phi}}C_{\phi}G_1(D)|D|^{-3/2}L(1/2,\phi\times\chi_D)\prod_{i=1}^R \left(1+\left(\frac{l_i^{\alpha_i}}{D}\right)w_{l_i^{\alpha_i},\phi}\right)(1+w_{\infty,\phi}) = \sum_{\textrm{ Shim}(F_j) = \phi} |\lambda_{F_j}(D)|^2$$ where
$C_{\phi} = \frac{-pw_{\infty,\phi}}{p+w_{\infty,\phi}}.$
\end{thm}

By Drinfeld's \cite{Dr} and Deligne's \cite{Del} works, $L(s,\phi\times\chi_D)$ satisfies the Riemann hypothesis. Hence by Theorem \ref{lindelofhyper2} we know that it satisfies the Lindel\"of Hypothesis. Therefore we get\\ \\

\begin{thm}[Metaplectic Ramanajun conjecture]\label{metramcon}

Let $F$ be a fixed cuspdial meteplectic form of level $N$, depth $0$ or a newform of depth $1$. Then for $D\in R$ square-free of even degree and relatively prime to $N$ we have the following bound
\[\begin{array}{cc}
|\lambda_F(D)|^2\ll_F |D|^{-1/2+o(1)}\,,&\text{as }|D|\rightarrow\infty
\end{array}\]
\end{thm}

\begin{proof}

Without loss of generality we can assume that $F$ is a Hecke eigenform. If $F$ is a 1-variable theta function then the Fourier-Whittaker coefficients $\lambda_F(D)$ are supported only on one square class, and the theorem follows tautologically.

Else, consider the $D$'th Shimura lift, $\phi$, of $F$. By the discussion in \S\ref{qualwalds}, $\phi$ will is a cuspidal Hecke eigenform itself. Then Waldspurger's formula combined with the fact that $|G_1(D)|=|D|^{1/2}$ and $|w_{\infty,\phi}|=|w_{l^{\alpha},\phi}|=1$ gives us
\begin{equation}\label{ineqram}
|\lambda_{F}(D)|^2\leq \sum_{Shim(F_j)=\phi}|\lambda_{F_j}(D)|^2\ll_F\frac{L(1/2,\phi\times\chi_D)}{\langle\phi,\phi\rangle}|D|^{-1}\ll_F L(1/2,\phi\times \chi_D)|D|^{-1}
\end{equation}

Where the implied constant in the last inequality is \emph{independent} of $D$ (Note that the Shimura lifts depended on $D$).  This is because the space $S_0(\Gamma_0(N))$ is finite dimensional, so we have $\langle\phi,\phi\rangle\geq \min\{\langle\phi_1,\phi_1\rangle,\cdots,\langle\phi_n,\phi_n\rangle\}$ independently of $D$, where $\{\phi_j\}_{j=1}^n$ is a Hecke eigenbasis for $S_{0}(\Gamma_0(N))$ (and similarly $S_1(\Gamma_0(N))$).

Then \eqref{ineqram} combined with Theorem \ref{lindelofhyper2} and gives the result for $D\in k_{\infty}^2$. For $D$ of even degree and not in $k_{\infty}^2$, take $\epsilon\in\F_p$ not a square so that we have $\epsilon D\in k_{\infty}^2$. We appeal to section \ref{dnotsquare} to see that $|\lambda_F(D)|=|\lambda_{F_{\epsilon}}(\epsilon D)|$ and the result follows.

\end{proof}

Note that the implied constants in $\ll_F$ notation are \emph{effective} and the stronger bound in Theorem \ref{lindelofhyper2} applies.

\section{Proof of Theorem \ref{representationTheorem}}\label{sectionquadratic}

We are now ready to finish the proof of Theorem \ref{representationTheorem}. Let $Q$ be an anisotropic ternary quadratic form, which is rationally equivalent over $k_{\infty}$
 to $T\epsilon x^2 + Ty^2 +\epsilon z^2$, where $\epsilon\in \mathbb{F}_p\backslash \mathbb{F}_p^2$. Define $\Theta_Q(z)$ and $\Theta_G(z)$ as in equations \ref{thetaG} and \ref{thetaQ}. Then by Siegel's Theorem (Theorem (\ref{SiegelsTheorem})) we can expand it as a sum of cuspidal Hecke eigenforms of type 1, depth 1 and level $\disc(Q)$ as

$$ \Theta_Q(z) = \Theta_G(z) + \displaystyle\sum_j F_j(z)$$

Recall that the sum on the right hand side is a finite sum.

Take $D\in R$ to be of even degree, square-free and relatively prime to $\disc(Q).$ By looking at the $D$'th Fourier-Whittaker coefficients of the above equality we get

$$ |v|^{3/4}r_Q(D)\chi_{\mathcal{O}_\infty}(Dv) = |v|^{3/4}r_G(D)\chi_{\mathcal{O}_\infty}(Dv) + \displaystyle\sum_j \lambda_{F_j}(D)\widetilde{W}_{F_j}(Dv)$$

where $\widetilde{W}_{F_j}$ is the appropriate Whittaker function, depending on which of the spaces
$\widetilde{S}_1^{new}(\widetilde{\Gamma}_0(N))$ and $\widetilde{S}_0(\widetilde{\Gamma}_0(N))$ contains
$F_j$.

By theorem \ref{metramcon} we have $|\lambda_{F_j}(D)|\ll_j |D|^{-1/2+\epsilon}$. Also by Siegel's mass formula \ref{Siegelmassformula} we have $r_G(D)\backsim L(q,\chi_D)|D|^{1/2}$. This, combined with Lemma \ref{classnumber} shows that $r_G(D)\gg_Q|D|^{1/2}/\log_p\log_p|D|$. (See \S\ref{Siegelmassformula}.) We conclude that, for $\epsilon>0$

\[r_Q(D) = r_G(D) + O_{\epsilon}(|D|^{\frac{1}{4}+\epsilon})\]

as desired. Note that the effectivity of Lemma \ref{classnumber} implies that the above is effective. In particular for a given form $Q$ we can actually write down all the even degree polynomials $D\in\mathbb{F}_p[T]$ that it represents.

\section{An Example}\label{Example}

In order to give a flavor of the kind of representability questions to which the methods developed in this paper applies to, we give the following example. Let $p=5$, $k=\mathbb{F}_5(T)$, $\epsilon\in\mathbb{F}_{5}$ a non-square, and $Q_1,\; Q_2$ be the anisotropic quadratic forms given by
\[Q_1(X,Y,Z)=X^2+(T^3+T+1)Y^2+\epsilon Z^2\]
\[Q_2(X,Y,Z)=(T^2-T-1)X^2+(T+1)XY+TY^2+\epsilon Z^2\]
First note that the two forms have the same discriminant $$\disc(Q_1)=\disc{(Q_2)}=\epsilon(T^3+T+1)=N$$ which is square-free over $\mathbb{F}_{5}$. \\
\begin{lemma}
$Q_1$ and $Q_2$ belong to the same genus.
\end{lemma}
\begin{proof}
 In order to check the equivalence of these forms at each completion we check the equivalence of the Kohnen symbol for each form. Note that for $w\nmid N$ the two forms are diagonalizable over $\mathcal{O}_w$ and since they have the same determinant (which is a unit in $\mathcal{O}_w$) they are equivalent and we only need to check $w$ dividing $N$. We start with $Q_1$. $Q_1$ represents $1$ hence the Kohnen symbol is $1$. The second form represents $T^2-T-1$. The Kohnen symbol is
\[\left(\frac{T^3+T+1}{T^2-T-1}\right)=1\]
Hence the forms belong to the same genus.
\end{proof}

Now, Taking $X=0,Y=1,Z=0$ we see that $Q_2$ represents $T$, while it easy to see by degree considerations that $Q_1$ does not represent any degree 1 polynomial.
Therefore there is no local to global principle in this case. However, our Corollary \ref{representationCor} implies that for all sufficiently large odd degrees,
$Q_1$ and $Q_2$ will represent the same polynomials.

\subsection{Acknowledgements}
The authors would like to thank Peter Sarnak for originally suggesting the problem under consideration and for many helpful discussions. We would also like to thank Benjamin Bakker and Bhargav Bhatt for related conversations. We would also like to thank the referee for a careful read of the first draft of the paper and pointing out to numerous inaccuracies in the exposition.

\appendix\label{sectiongauss}
\section{Weil representation}

In this section we will recall the Weil representation and then use it to get the desired transformation properties of the theta functions used throughout the text.

\begin{subsection}{Gauss sums}\label{sectiongauss1}

\begin{subsubsection}{Gauss sums on quadratic spaces}\label{gausquad}

Let $k_w$ be the completion of $\mathbb{F}_p(T)$ at a place $w$, and $p\equiv1\bmod 4$ (actually for this subsection all we need is a local field of characteristic not $2$), and let $\psi_w$ be a non-trivial additive character of $k_w$. In this section we will define Gauss sums for a quadratic space $(V_w,Q_w)$ over $k_w$. Since the cahracteristic of $k_w$ is not $2$, there exists $1$-dimensional quadratic subspaces $(V_{w,i},Q_{w,i})$ such that $(V_w,Q_w)=\bigoplus (V_{w,i},Q_{w,i})$. Let $d\mu_{w,i}$ be the self-dual Haar measure on $V_{w,i}$ with respect to the Fourier transform defined by the bilinear form associated to $Q_{w,i}$. More explicitly; define the bilinear form $B_{Q_{w,i}}:V_{w,i}\times V_{w,i}\rightarrow \mathbb{C}$ by $B_{Q_{w,i}}(x,y)=(Q_{w,i}(x+y)-Q_{w,i}(x)-Q_{w,i}(y))/2$. Then the measure $d\mu_{w,i}$ is normalized so that $\hat{\hat{f}}(x)=f(-x)$, where the Fourier transform is defined by
\begin{equation}\label{four}
\hat{f}(y):=\int_{V_{w,i}}f(x)\psi_w(-2B_{Q{w,i}}(x,y))d\mu_{w,i}(x)
\end{equation}
Now given a triple $(V_w,Q_w,\psi_w)$ define $G_{\psi_w,Q_w}$ by
\[G_{\psi_w,Q_w}:=\prod_{i}G_{\psi_w,Q_{w,i}}\]
Where we define the $1$-dimensional Gauss sums $G_{\psi_w,Q_{w,i}}$ as follows: Let $\{U_j\}_{j=1}^{\infty}$ be a set of compact open subsets of $V_{w,i}$ such that $U_i\subset U_j$ for $i\leq j$, and $\cup_{j=1}^{\infty}U_j=V_{w,i}$. Then define
\[G_{\psi_w,Q_{w,i}}:=\lim_{j\rightarrow \infty}\int_{U_j}\psi_w(Q_{w,i}(x))d\mu_{w,i}(x)\]
The fact that this limit exists and is independent of the choice of the cover $\{U_j\}_{j=1}^{\infty}$ follows from Lemma 1.5.1 of \cite{Sha} (the sequence of integrals stabilize after a sufficiently large $j$). The independence of $G_{\psi_w,Q_w}$ of the decomposition into $Q_{w,i}$'s also follow from the same lemma.

Finally, for $\alpha\in k_w^{\times}$ we define $G_{\psi_w,Q_w}(\alpha)$ by $\prod_{i}G_{\psi_{w},Q_{w,i}}(\alpha)$, where the $1$-dimensional sums are defined by
\[G_{\psi_w,Q_{w,i}}(\alpha):=\lim_{j\rightarrow \infty}\int_{U_j}\psi_w(\alpha Q_{w,i}(x))d\mu_{w,i}(x)\]

\end{subsubsection}

\begin{subsubsection}{Gauss sums for a general additive character}

Following \cite{Sha}, in this section we will calculate the $1$-dimensional Gauss sums for the quadratic space $(k_w,Q_w)$, where $Q_w(x)=x^2$. The notation is as in \S\ref{gausquad} Let $\varpi_w$ be a uniformizer for the maximal ideal $\mathcal{O}_w\subset k_w$, and denote the conductor of $\psi_w$ by $\mathfrak{f}_w=\varpi_w^{m_w}$. Let $\mu_w$ be the self-dual Haar measure on $k_w$ as in \S\ref{gausquad}. i.e. $\mu_w$ is normalized by $\hat{\hat{f}}(x)=f(-x)$ where the Fourier transform is defined by (\ref{four}).
We denote the Gauss sum $G_{\psi_w,Q_w}$ simply by $G_{\psi_w}$. i.e.
\begin{equation}\label{g1}
G_{\psi_w}=\lim_{m\rightarrow\infty}\int_{\varpi_w^{-m}\mathcal{O}_w}\psi_w(x^2)d\mu_w(x)
\end{equation}
As before, for any $\alpha\in k_w^{\times}$ we have the Gauss sums, $G_{\psi_w}(\alpha)$;
\begin{equation}\label{g2}
G_{\psi_w}(\alpha)=\lim_{m\rightarrow\infty}\int_{\varpi_w^{-m}}\psi_{w}(\alpha x^2)d\mu(x)
\end{equation}
A quick computation (cf. \cite{Sha}) shows that the Gauss sums above reduces to:\\

\begin{lemma}[Shalika, \cite{Sha}]\label{sh1}Let $\alpha\in k_{w}^{\times}$. Then
\[G_{\psi_w}(\alpha)=\left|\alpha\right|_w^{-1/2}\begin{cases}1&\text{if $v_w(\alpha)-m_w\equiv0\bmod 2$}\\ |\varpi_w|_w^{\frac{1}{2}}\displaystyle\sum_{x\in\mathcal{O}_w/\varpi_w\mathcal{O}_w}\psi_w(\alpha_0\varpi_w^{m_w-1}x^2)&\text{if $v_w(\alpha)-m_w\equiv1\bmod 2$}\end{cases}\]
Where $\alpha_0\in\mathcal{O}_{w}^{\times}$ is defined by $\alpha=\alpha_0\varpi_{w}^{v_{w}(\alpha)}$.
\end{lemma}

\begin{proof}This is Lemma 1.3.2 of \cite{Sha}.
\end{proof}

\end{subsubsection}

\begin{subsubsection}{Explicit computation of a concrete Gauss sum}\label{expcomp}

In what follows we will be considering various quadratic forms and $\theta$-functions associated to them. We will derive their transformation properties from the underlying Weil representation. In this subsection we will calculate various Gauss sums that will appear in that context.

The setup of this subsection is almost identical to that of \S\ref{sectiongauss1} but with the difference that now we will explicitly chose a global additive character and uniformizers at each place. Let $p\equiv1\bmod4$ be an odd prime, and $k=\mathbb{F}_p(T)$. Let $w$ be a valuation of $k$ and denote the completion of $k$ at $w$ by $k_w$. Let $\mathcal{O}_w$ be the ring of integers, $\mathbb{F}_w$ be the residue field at $w$, and $\varpi_w$ be a uniformizer for $\mathcal{O}_w$. Let $\psi_w$ denote the character $\psi_w(\alpha)=e^{2\pi i(\frac{Tr_w(Res_w(-\alpha d^{\times}T))}{p})}$, where $\alpha \in k_w$, $d^{\times}T=\frac{dT}{T^2}$ is a meromorphic differential on $\mathbb{P}^1$ (for a precise definition of a differential see for instance chapter 4 of \cite{Sti}). For any differential form $\omega_w$, $Res_w(\omega_w)$ denotes the residue of the form at the point $w$ (note that this is independent of the choice of the uniformizer, cf \cite{Sti}) and $Tr_w$ is the trace function, $Tr_w:\mathbb{F}_w\rightarrow \mathbb{F}_p$.

In this setup we can explicitly compute the Gauss sum that appears in Lemma \ref{sh1}. We start with the base case when $\mathbb{F}_w\cong\mathbb{F}_p$.\\

\begin{lemma}\label{sh2}Let $w$ be a degree $1$ valuation, i.e. $\mathbb{F}_w\cong\mathbb{F}_p$, and $\alpha\in k_w^{\times}$, and let $\psi_w$ be the character defined above. Then,
\[|\alpha|_w^{1/2}G_{\psi_w}(\alpha)=\begin{cases}1&\text{if $v_w(\alpha)-m_w\equiv0\bmod2$}\\ (\alpha_0,\varpi_w)_w&\text{if $v_{w}(\alpha)-m_w\equiv 1\bmod 2$}\end{cases}\]
Where $\alpha_0\in\mathcal{O}_w^{\times}$ is defined by $\alpha=\alpha_0\varpi_w^{v_w(\alpha)}$.
\end{lemma}

\begin{proof}The part when $v_w(\alpha)-m_w\equiv0\bmod 2$ is a direct consequence of Lemma \ref{sh1}. Therefore for the rest of the proof assume that $v_w(\alpha)-m_w\equiv1\bmod 2$. Therefore we are reduced to computing
\[|\varpi_w|_w^{1/2}\sum_{x\in\mathcal{O}_w/\varpi_w\mathcal{O}_w}\psi_w(\alpha_0\varpi_w^{m_w-1}x^2)\]
 We first start with the case $w$ is not the valuation induced by $T,$ or $T^{-1}$. In this case $d^{\times}T$ is holomorphic (hence $m_w=0$) and the residue of $\alpha_0 \varpi_w^{m_w-1}x^2d^{\times}T$ is $\alpha_{0,0}x_0^2$, where $\alpha_0\in \alpha_{0,0}+\varpi_w\mathcal{O}_w$ and $x\in x_0+\varpi_w\mathcal{O}_w$, and we have identified $\mathcal{O}_w/\varpi_w\mathcal{O}_w$ with $\mathbb{F}_p$. Then by definition
\[\sum_{x\in\mathcal{O}_w/\varpi_w\mathcal{O}_w}\psi_w(\alpha_0\varpi_w^{m_w-1}x^2)=\sum_{\beta\bmod p}e^{\frac{2\pi i -\alpha_{0,0} \beta^2}{p}}\]
which is a classical Gauss sum. Now recall that $p\equiv1\bmod4$, therefore the above sum is equal to $\left(\frac{\alpha_{0,0}}{\varpi_w}\right)\sqrt{p}$. Then the lemma follows by noting that $(\alpha_0,\varpi_w)_w=\left(\frac{\alpha_{0,0}}{\varpi_w}\right)$.

If $w=\infty$, i.e. it is attached to the degree valuation, then in the local coordinates $d^{\times}T$ becomes $-d\varpi_{\infty}$ and the above argument works verbatim with the flip of a sign (Note that since $p\equiv1\bmod 4$ this sign flip does not change the answer.).

Finally for $w$ induced by $T$, in local coordinates we have $d^{\times}T=\frac{dT}{T^2}$ hence $m_w=2$. The residue of $\alpha_0\varpi_w^{m_w-1}x^2d^{\times}T$ is again $\alpha_{0,0}x_0^2$ where $\alpha_0\in\alpha_{0,0}+\varpi_w\mathcal{O}_w$ and $x\in x_0+\varpi_w\mathcal{O}_w$. The above argument then gives the result.

\end{proof}

We will now pass to the general case of arbitrary degree valuations by using Hasse-Davenport relations which we quickly recall. Let $\mathbb{F}_{q^s}/\mathbb{F}_q$ be an extension of finite fields of degree $s$. Let $\psi$ be a non-trivial additive character of $\mathbb{F}_q$ and $\chi$ be a non-trivial multiplicative character of $\mathbb{F}_q^{\times}$. Let $Tr$ and $N$ denote the trace, $Tr:\mathbb{F}_{q^s}\rightarrow \mathbb{F}_q$, and norm, $N:\mathbb{F}_{q^s}^{\times}\rightarrow \mathbb{F}_q$, maps respectively. Let $\tau(\psi,\chi)$ denote the Gauss sum attached to any pair, $(\psi,\chi)$, of additive and multiplicative character of a field $\mathbb{F}_q$. i.e.
\[\tau(\psi,\chi):=\sum_{\alpha\in\mathbb{F}_q^{\times}}\chi(\alpha)\psi(\alpha)\]
The Hasse-Davenport relation relates the two Gauss sums $\tau(\psi,\chi)$ and $\tau(\psi\circ Tr,\chi\circ N)$, and states that
\[(-\tau(\psi,\chi))^s=-\tau(\psi\circ Tr,\chi\circ N)\]
For a proof of this relation see \cite{IK} pg. 274-278. We can now pass to the general case of computing the Gauss sums over an arbitrary residue field.\\

\begin{lemma}\label{sh3}Let $w$ be a place of $k$, $k_w$ be the completion of $k$ at $w$, $\mathcal{O}_w$ the ring of integers of $k_w$, and $\varpi_w$ be a uniformizer at $w$. Let $\mathbb{F}_w$ be the residue field at $w$, and let $s_w$ be such that $\#\mathbb{F}_w=p^{s_w}$. Let $\psi_w$ be the character defined above. Then  for each $\alpha \in k_w^{\times}$ we have
\[|\alpha|_w^{\frac{1}{2}}G_{\psi_w}(\alpha)=\begin{cases}1&\text{if $v_w(\alpha)-m_w\equiv 0\bmod 2$}\\(-1)^{1+s_w}(\alpha_0,\varpi_w)_w&\text{if $v_w(\alpha)-m_w\equiv1\bmod2$}\end{cases}\]
Where $\alpha_0\in\mathcal{O}_w^{\times}$ is defined by $\alpha=\alpha_0\varpi_w^{v_w(\alpha)}$.
\end{lemma}

\begin{proof}Again by Lemma \ref{sh1} the proof reduces to the case where $v_w(\alpha)-m_w\equiv1\bmod 2$, therefore we assume this for the rest of the proof. We need to compute
\[|\varpi_w|_w^{\frac{1}{2}}\displaystyle\sum_{x\in\mathcal{O}_w/\varpi_w\mathcal{O}_w}\psi_w(\alpha_0\varpi_w^{m_w-1}x^2)\]
We start with noting the following; if we denote the quadratic character that is $1$ on squares and $-1$ on non-squares of $\mathbb{F}_{p^{s_w}}$ by $\left(\tfrac{\beta}{\mathbb{F}_{p^{s_w}}}\right)$ (note that this is \emph{not} the Legendre symbol $\bmod\, p^{s_w}$), then we have
\[\sum_{x\in\mathcal{O}_w/\varpi_w\mathcal{O}_w}\psi_w(\alpha_0\varpi_w^{m_w-1}x^2)=\sum_{x\in\mathbb{F}_{p^{s_w}}}e^{\frac{2\pi i Tr_w(-Res_w(\alpha_{0}\varpi_w^{m_w-1}x^2d^{\times}T))}{p}}=\sum_{x\in\mathbb{F}_{p^{s_w}}^{\times}}\left(\frac{x}{\mathbb{F}_{p^{s_w}}}\right)e^{\frac{2\pi i Tr_w(-Res_w(\alpha_{0}\varpi_w^{m_w-1}xd^{\times}T))}{p}}\]
Since the multiplicative group of $\mathbb{F}_{p^{s_w}}$ is isomorphic to its own dual, and is cyclic of order $p^{s_w}-1$ and $p$ is odd, it has a \emph{unique} subgroup of index $2$ which means there is a unique character of order $2$ on $\mathbb{F}_{p^{s_w}}$. Denoting the norm function from $\mathbb{F}_{p^{s_w}}$ to $\mathbb{F}_p$ by $N_w$, note that both $\left(\frac{\cdot}{\mathbb{F}_{p^{s_w}}}\right)$ and $\left(\frac{N_w(\cdot)}{p}\right)$ (note that we are using the ordinary Legendre symbol this time) are characters of $\mathbb{F}_{p^{s_w}}^{\times}$ of oerder $2$, and hence they are the same. Therefore we get that the sum above is
\[\sum_{x\in\mathbb{F}_{p^{s_w}}^{\times}}\left(\frac{x}{\mathbb{F}_{p^{s_w}}}\right)e^{\frac{2\pi i Tr_w(-Res_w(\alpha_{0}\varpi_w^{m_w-1}xd^{\times}T))}{p}}=\left(\frac{\alpha_{0,0}}{\mathbb{F}_{p^{s_w}}}\right)\sum_{x\in\mathbb{F}_{p^{s_w}}^{\times}}\left(\frac{N_w(x)}{p}\right)e^{\frac{2\pi i Tr_w(-Res_w(\varpi_w^{m_w-1}xd^{\times}T))}{p}}\]
\[=\left(\frac{\alpha_{0,0}}{\mathbb{F}_{p^{s_w}}}\right)\tau\left(\psi\circ Tr_w,\left(\frac{\cdot}{p}\right)\circ N_w\right)=-\left(\frac{\alpha_{0,0}}{\mathbb{F}_{p^{s_w}}}\right)\left(-\tau\left(\psi,\left(\frac{\cdot}{p}\right)\right)\right)^{s_w}\]
by the Hasse-Davenport relation. Now in order to finish the proof note that the inner Gauss sum was computed in Lemma \ref{sh2} and its value is $\sqrt{p}$, and the quadratic character $\left(\tfrac{\alpha_{0,0}}{\mathbb{F}_{p^{s_w}}}\right)$ is equal to the Hilbert symbol $(\alpha_0,\varpi_w)_w$. The lemma follows.
\end{proof}

\end{subsubsection}

\end{subsection}

\subsection{Weil representation and theta functions}

The theory of Weil representation provides a natural framework in which theta functions can be viewed as automorphic forms. We will introduce the Weil representation and the Metaplectic group and show that our theta functions fit nicely into this framework, and use this to show the modularity of such. Referances for this section are \cite{G}, \cite{P}, \cite{Sha}, \cite{W}.

\subsubsection{Local Weil representation}
We begin by sketching the general construction of the Weil representation and the metaplectic group. Let $k_w$ denote the completion of the function field $k=\mathbb{F}_p(T)$ at a place $w$ as usual. ($char(k)\neq2$) Let $\Omega$ be an $n$-dimensional vector space over $k_w$ with a non-degenerate symplectic form, $\langle\cdot\,,\cdot\rangle:\Omega\times\Omega \rightarrow k_w$. Define the Heisenberg group of $\Omega$ by
\[H(\Omega)=\left\{(u,t)\mid u\in \Omega,\: t\in k_w\right\}\]
where multiplication is given by
\[(u_1,t_1)(u_2,t_2)=(u_1+u_2,t_1+t_2+\frac12 \langle u_1,u_2\rangle)\]

For any nontrivial additive character $\psi_w$ of $k_w$, the Stone-von Neumann Theorem guarantees the existence of a unique irreducible representation $\rho_{\psi_w}$ of $H(\Omega)$ on which $k_w$ acts by the character $\psi_w$. We also have, by definition of $H(\Omega)$, $Sp(\Omega)$ acting on $H(\Omega)$ by $g(u,t)=(gu,t)$. For each $g\in Sp(\Omega)$ the composition $\rho_{\psi_w}(g\:\cdot)$ gives another irreducible representation of $H(\Omega)$ on which $k_w$ acts by $\psi_w$. Then by the uniqueness part of Stone-von Neumann Theorem we get an intertwiner $\tilde{\sigma}_{\psi_w}$, unique up to a scalar such that
\[\rho_{\psi_w}(gu,t)=\tilde{\sigma}_{\psi_w}(g)\rho_{\psi_w}(u,t)\tilde{\sigma}_{\psi_w}(g)^{-1}\]
Since the intertwiner is defined up to a scalar, $g\rightarrow \tilde{\sigma}_{\psi_w}(g)$ defines a projective representation of $Sp(\Omega)$. By the standard theory of projective representations (\cite{Sha}), this representation gives rise to a unique cohomology class of order two in $H^2(Sp(\Omega),S^1)$ (\cite{W}, \cite{G}), where $S^1=\{z\in\mathbb{C}\mid |z|=1\}$ is the unit circle. This cohomology class then defines a double cover of $Sp(\Omega)$ which is called the \emph{metaplectic group} and denoted by $\widetilde{Sp}(\Omega)$. Furthermore we also know that $H^2(Sp_n,\mathbb{Z}/2\mathbb{Z})\cong \mathbb{Z}/2\mathbb{Z}$ (cf. Theorem 10.4 of \cite{Mo}) so this construction gives \emph{the} double cover of $Sp(\Omega)$.

We will now adapt the above general picture to our situation of $SL_2$ (mainly following Gelbart \cite{G}) and give an explicit realization of the Weil representation (We give the construction for $n=3$ dimensions here, but the construction works almost verbatim for general $n$.). Let $V_w$ be a $3$ dimensional vector space over $k_w$ and $Q_w$ a non-degenerate quadratic form on $V_w$. Fix an additive character $\psi_w$ of $k_w$. We define our symplectic group, $Sp(V_w\times V_w)$, to be the group of automorphisms (over $k_w$) of $V_w\times V_w$ fixing the symplectic form: $(v,u)\rightarrow B_{Q_w}(v_1,u_2)-B_{Q_w}(v_2,u_1)$, $(v,u)\in V_w\times V_w$. Notice that $k_w$ acts on $V_w$ by scalar multiplication, which induces an action of $SL_2(k_w)$ on $V_w\times V_w$ by $(v,u)\rightarrow (av+bu,cv+du)$, where $\left(\begin{smallmatrix}a&b\\c&d\end{smallmatrix}\right)\in SL_2(k_w)$. This shows that $SL_2(k_w)$ embeds into $Sp(V_w\times V_w)$ and then the above construction associates a projective representation to $SL_2(k_w)$ whose cocycle is of order $2$, and hence defines a representation of the double cover of $SL_2(k_w)$. Therefore for each quadratic space, $(V_w,Q_w)$, and a non-trivial additive character, $\psi_w$, this construction\footnote{The above constuction works perfectly well over an arbitrary local field $F$ of residual characteristic \emph{not} equal to $2$ and for any $n$ dimensional vector space $V$ over the field. When the dimension $n$ is even the projective representation reduces a genuine representation of the group $SL_2(F)$ and when $n$ is odd we get a representation of $\widetilde{SL}_2(F)$.} gives a representation of the double cover of $SL_2(k_w)$.

We can realize the Weil representation explicitly as follows. Let $V_w\times V_w=X\oplus Y$, with $X$ and $Y$ maximal isotropic subspaces, be a complete polarization of the symplectic space $V_w\times V_w$. (Note that the decomposition $V_w\times V_w=V_w\oplus V_w$ produces such a decomposition and we will be working with this particular decomposition)  Let $\psi_w$ be the additive character defined by $\psi_w(h_w)=e^{2\pi i(\frac{Tr_w(-Res(h_wd^{\times}T))}{p})}$, where $h_w\in k_w$, $\varpi_w$ is a uniformizing parameter at $w$, $d^{\times}T=\frac{dT}{T^2}$ is a meromorphic differential on $\mathbb{P}^1$, for any differential form $\omega_w$, $Res(\omega_w)$ denotes the residue of the form (at the point $w$) and $Tr_w$ is the trace function, $Tr_w:\mathcal{O}_w/\varpi_w\mathcal{O}_w\rightarrow \mathbb{F}_p$. Let $\mathcal{S}(V_w)$ be the Schwartz space of locally constant functions on $V_w$ and denote the unitary operators on $\mathcal{S}(V_w)$ by $U(\mathcal{S}(V_w))$. Following \cite{Sha} we define a projective representation, which by abuse of language we still will call Weil representation, $\sigma_{w,Q_w}:SL_2(k_w)\longrightarrow U(\mathcal{S}(V_w))$, by
\begin{multline}\label{weilrep}\sigma_{w,Q_{w}}\left(\left(\begin{smallmatrix}a&b\\c&d\end{smallmatrix}\right)\right)f(\nu)=\\ \begin{cases} |a|_w^3 \frac{G_{\psi_w,Q_w}(a)}{G_{\psi_w,Q_w}}\psi_w\left(abB_{Q_w}(\nu,\nu)\right)f(a\nu)&\text{if $c=0$}\\ G_{\psi_w,Q_w}(c)\int_{V_w}\psi_w\left(\frac{aB_{Q_w}(\nu,\nu)-2B_{Q_w}(\nu,\nu_1)+dB_{Q_w}(\nu_1,\nu_1)}{c}\right)f(\nu_1)d\mu_w(\nu_1)&\text{if $c\neq0$}\end{cases}
\end{multline}
Where the measure $d\mu_w$ is normalized so that $\hat{\hat{f}}(\nu)=f(-\nu)$, where the Fourier transform is defined by
\begin{equation}\label{menor}
\hat{f}(\nu)=\int_{V_w}f(\nu_1)\psi_w(-2B_{Q_w}(\nu,\nu_1))d\mu_w(\nu_1)
\end{equation}
and $G_{\psi_w,Q_w}(\alpha)$ is as defined in \S\ref{gausquad}. We note that the normalization of measure used above is the same normalization used in defining the Gauss sums $G_{\psi_w,Q_w}$.

It follows from Theorem 2.22 of \cite{G} that this definition makes sense and defines a projective representation of $SL_2(k_w)$ whose cocycle is non-trivial (cf. \cite{G}, Corollary 2.24, and note that $\dim(V_w)=3$) and is of order $2$, therefore defines \emph{the} double cover, $\widetilde{SL}_2(k_w)$, of $SL_2(k_w)$.

\begin{subsubsection}{Comparison of cocyles associated to certain quadratic forms}\label{comparison}

For our applications to quadratic forms we will need to calculate the cocylcles defined above explicitly for certain quadratic forms $Q_w$. We start with the quadratic form $Q_w=\sum_{i=1}^3x_1^2$. We show that the cocycle defined in by the above construction matches the cocycle given in \S\ref{metgp}.\\

\begin{prop}\label{cocyclestandard}Let $Q_w$ be equivalent over $k_w$ to $\sum_{i=1}^3x_i^2$. Recall the general definition of the 2-cocycle $\epsilon$ given in \S\ref{metgp}. Given $g_1,g_2, g_1g_2=g_3\in SL_2(k_w)$, $\epsilon(g_1,g_2)=(X(g_1),X(g_2))_w(X(g_1),X(g_3))_w(X(g_2),X(g_3))_w$, where $X(g)$ is the lower left or lower right entry depending on the lower left entry being non-zero or not respectively. Then the cocycle defined by $\sigma_{w,Q_w}$ is equal to $\epsilon$.
\end{prop}

\begin{proof} We start the proof by noticing the following; since our form $Q_w$ is equivalent to $\sum_{i=1}^3x_i^2$, the Gauss sums $G_{\psi_w,Q_w}(\alpha)$ (which will appear below momentatily) which are defined as in \S\ref{gausquad} becomes the same as the explicit Gauss sums in (\ref{g2}), whose values are calculated by lemmas \ref{sh1} and \ref{sh3}. We will be using these computations without further reference. It might also be instructive to go through the computation given in the proof of Proposition \ref{cocycleaniso} to see an explicit calculation of Gauss sums.

For convenience, throughout the proof we will denote $\sigma_{w,Q_w}$ by $\sigma_w$. Recall that by the Bruhat decomposition element $g=\left(\begin{smallmatrix}a&b\\ c&d\end{smallmatrix}\right)\in SL_2(k_w)$ can be expressed as $g=\gamma$ or $g=\gamma T_w\gamma^{'}$ depending on $c=0$ or $c\neq0$ respectively, where $T_{w}=\left(\begin{smallmatrix}0&-1\\1&0\end{smallmatrix}\right)$, and $\gamma=\left(\begin{smallmatrix}s&t\\0&s^{-1}\end{smallmatrix}\right)$ and $\gamma^{'}$ are upper triangular matrices. Therefore in order to check the cocycle we only need to check it for the products $T_wg$ and $\gamma g$. By (\ref{weilrep}) we have
\begin{align*}
\sigma_w(\gamma)f(\nu)&=|s|_w^3\frac{G_{\psi_w,Q_w}(s)}{G_{\psi_w,Q_w}}\psi_w(stQ_w(\nu))f(s\nu)\\
\sigma_w(T_w)f(\nu)&=G_{\psi_w,Q_w}\hat{f}(\nu)
\end{align*}
We now compare $\epsilon$ with the cocycle defined by $\sigma_w$ case by case.
\begin{itemize}
\item $T_wg=\left(\begin{smallmatrix}-c&-d\\ a&b\end{smallmatrix}\right)$. Then
\[\sigma_w\left(T_wg\right)f(\nu)=\\ \begin{cases} |c|_w^3 \frac{G_{\psi_w,Q_w}(c)}{G_{\psi_w,Q_w}}\psi_w\left(cdB_{Q_w}(\nu,\nu)\right)f(-c\nu)&\text{if $a=0$}\\ G_{\psi_w,Q_w}(a)\int_{V_w}\psi_w\left(\frac{-cB_{Q_w}(\nu,\nu)-2B_{Q_w}(\nu,\nu_1)+bB_{Q_w}(\nu_1,\nu_1)}{a}\right)f(\nu_1)d\mu_w(\nu_1)&\text{if $a\neq0$}\end{cases}\]
Let us first assume that $c=0$. Then $a\neq0$ and $\sigma_w(g)f(\nu)=|a|_w^3\frac{G_{\psi_w,Q_w}(a)}{G_{\psi_w,Q_w}}\psi_w(abQ_w(\nu))f(a\nu)$. Then $\sigma_{w}(T_w)\sigma_w(g)f(\nu)$ is given by
\[|a|_w^{3}G_{\psi_w,Q_w}(a)\int_{V_w}\psi_w(abQ_w(\nu_1))\psi_w(-2B_{Q_w}(\nu,\nu_1))f(a\nu_1)d\mu_w\nu_1\]
Changing the variables $\nu_1\mapsto a^{-1}\nu_1$ shows that the cocycle is $1$ in this case. On the other hand $\epsilon(T_w,g)=(1,d)_w(1,a)_w(a,d)_w=1$ since $a=d^{-1}$ and hence they belong to the same square class.

Now assume that $c\neq0$. Then there are two sub-cases, $a=0$ or $a\neq0$. First suppose that $a=0$. Then
\begin{multline}
\sigma_{w}(T_w)\sigma_w(g)f(\nu)=G_{\psi_w,Q_w}G_{\psi_w,Q_w}(c)\times\\ \int_{V_w\times V_w}\psi_w\left(\frac{-2B_{Q_w}(\nu_1,\nu_2)+dB_{Q_w}(\nu_2,\nu_2)}{c}-2B_{Q_w}(\nu_1,\nu)\right)d\mu_w(\nu_2)f(\nu_2)d\mu_w(\nu_1)
\end{multline}
Changing the order of integration (we remark that the $\nu_2$ integral does not quite convergence but is interpreted in the distributional sense, we omit the details) we get that this is equal to
\[|c|_w^3 G_{\psi_w,Q_w}G_{\psi_w,Q_w}(c)\psi_w\left(cdB_{Q_w}(\nu,\nu)\right)f(-c\nu)\]
Therefore the cocycle in this case is seen to be $G_{\psi_w,Q_w}^2=1$. On the other hand $\epsilon(T_w,g)=(1,c)_w(1,d)_w(c,b)_w=1$ since $c=b^{-1}$.

Finally assume that $c\neq0$ and $a\neq0$. In this case we have
\begin{multline}
\sigma_{w}(T_w)\sigma_w(g)f(\nu)=G_{\psi_w,Q_w}G_{\psi_w,Q_w}(c)\times\\ \int_{V_w\times V_w}\psi_w\left(\frac{aB_{Q_w}(\nu_1,\nu_1)-2B_{Q_w}(\nu_1,\nu_2)+dB_{Q_w}(\nu_2,\nu_2)}{c}-2B_{Q_w}(\nu_1,\nu)\right)d\mu_w(\nu_2)f(\nu_2)d\mu_w(\nu_1)
\end{multline}
As before interchanging the order of integration and completing the $\nu_1$ integral to square we get that the above is equal to
\[G_{\psi_w,Q_w}G_{\psi_w,Q_w}(c)G_{\psi_w,Q_w}(ac^{-1})G_{\psi_w,Q_w}(a)^{-1}\sigma_w(T_wg)f(\nu)\]
Therefore the cocycle in this case is given by
\[\frac{G_{\psi_w,Q_w}G_{\psi_w,Q_w}(c)G_{\psi_w,Q_w}(ac^{-1})}{G_{\psi_w,Q_w}(a)}\]
Now a case by case check with respect to the $w$-adic valuations of $a$ and $c$, using the explicit formulas of \S\ref{expcomp} shows that the above is equal to $(a,c)_w$ which in particular is equal to $\epsilon(T_w,g)=(1,c)_w(1,a)_w(a,c)_w$. This finishes the proof for $T_wg$.

\item $\gamma g=\left(\begin{smallmatrix}sa+tc&sb+td\\ s^{-1}c&s^{-1}d\end{smallmatrix}\right)$. As before, we have two cases according to whether $c=0$ or not. If $c=0$, then
\[\sigma_w(\gamma)\sigma_w(g)f(\nu)=\left|as\right|_w^3\frac{G_{\psi_w,Q_w}(a)G_{\psi_w,Q_w}(s)}{G_{\psi_w,Q_w}^2}\psi_w(s(abs+t)Q_w(\nu))f(as\nu)\]
and
\[\sigma_w(\gamma g)=|as|_w^3\frac{G_{\psi_w,Q_w}(as)}{G_{\psi_w,Q_w}}\psi_w(s(abs+t)Q_w(nu))f(as\nu)\]
Therefore the cocycle is given by $\frac{G_{\psi_w,Q_w}(a)G_{\psi_w,Q_w}(s)}{G_{\psi_w,Q_w}G_{\psi_w,Q_w}(as)}$. On the other hand $\epsilon(\gamma,g)=(s,a)_w(s,sa)_w(a,sa)_w=(s,a)_w$. Then again a case by case check (using the computations of \S\ref{expcomp}) depending on the valuations of $a$ and $s$ shows that the two are equal.

Finally if $c\neq0$, then
\begin{multline*}
\sigma_w(\gamma)\sigma_w(g)f(\nu)=|s|_w^3\frac{G_{\psi_w,Q_w}(s)G_{\psi_w,Q_w}(c)}{G_{\psi_w,Q_w}}\times \\ \psi_w(stQ_w(\nu)) \int_{V_w}\psi_w\left(\frac{as^2B_{Q_w}(\nu,\nu)-2sB_{Q_w}(\nu,\nu_1)+dB_{Q_w}(\nu_1,\nu_1)}{c}\right)f(\nu_1)d{\mu_w}(\nu_1)
\end{multline*}
and
\begin{multline*}
\sigma_w(\gamma g)f(\nu)=G_{\psi_w,Q_w}(s^{-1}c)\psi_w(stQ_w(\nu))\times\\ \int_{V_w}\psi_w\left(\frac{as^2B_{Q_w}(\nu,\nu)-2sB_{Q_w}(\nu,\nu_1)+dB_{Q_w}(\nu_1,\nu_1)}{c}\right)f(\nu_1)d\mu_w(\nu_1)
\end{multline*}
Hence the cocycle is defined by $\frac{|s|^3_wG_{\psi_w,Q_w}(s)G_{\psi_w,Q_w}(c)}{G_{\psi_w,Q_w}G_{\psi_w,Q_w}(cs^{-1})}$. On the other hand $\epsilon(\gamma,g)=(s^{-1},c)_w(s^{-1},cs^{-1})_w(c,cs^{-1})_w=(s,c)_w$, and a case by case check as before finishes the proof.

\end{itemize}

\end{proof}

We will also be needing the anisotropic form $Q_w$ equivalent to $u x_1^2+\varpi_wx_2^2+u\varpi_wx_3^2$ over $k_w$, where $u\in \mathcal{O}_w^{\times}\backslash (\mathcal{O}_w^{\times})^2$. \\

\begin{prop}\label{cocycleaniso}Let $Q_w$ be equivalent to the anisotropic form $ux_1^2+\varpi_w x_2^2+u\varpi_wx_3^2$ over $k_w$, where $u\in \mathcal{O}_w^{\times}\backslash (\mathcal{O}_w^{\times})^2$. Then the cocycle defined by $\sigma_{w,Q_w}$ is equal to $\epsilon$ as defined in \S\ref{metgp} and in Proposition \ref{cocyclestandard}.
\end{prop}

\begin{proof} The proof is exactly the same as the proof of Proposition \ref{cocyclestandard}. The only difference is the computation of Gauss sums, $G_{\psi_w,Q_w}$, that show up in the cocycle associated to $\sigma_{w,Q_w}$ which we recall here. All the necessary computations were done in \S\ref{expcomp}, which we take all the notation and definitions from. We have the quadratic form $Q_w$ which is equivalent to $ux_1^2+\varpi_wx_2^2+u\varpi_ux_3^2$ over $k_{w}$. We will follow the recipe given in \S\ref{gausquad} to explicitly write the Gauss sum. Since we are given the diagonal form of $Q_w$, by the notation of \S\ref{gausquad}, we can take $(V_w,Q_w)=\bigoplus (k_w,Q_{w,i})$, where $Q_{w,i}(x)=ux^2$, $\varpi_wx^2$ or $u\varpi_wx^2$ for $i=1,2,3$ respectively. Then the Gauss sum $G_{\psi_w,Q_w}(\alpha)=\prod_{i}G_{\psi_w,Q_{w,i}}(\alpha)$, $Q_{\psi_w,Q_{w,i}}$ is as defined in (\ref{g2}). Now we can compute these Gauss sums using lemmas \ref{sh1} and \ref{sh3}. We only remark that the measures on each $V_{w,i}$ is normalized with respect to $Q_{w,i}$.

We only go through the details for $G_{\psi_{w},Q_{w,3}}$, which is the Gauss sum associated to the quadratic form $u\varpi_wx_3^2$. Let $\alpha\in k_w^{\times}$ by such that $\alpha=\alpha_0\varpi_w^{v_w(\alpha)}$. The measure is normalized as in \S\ref{gausquad} so that the Fourier transform, $\hat{f}(y)=\int_{k_w}f(x)\psi_w(-2u\varpi_wxy)d\mu_{w,i}(x)$, satisfies $\hat{\hat{f}}(x)=f(-x)$. In particular, by Lemma \ref{sh1}, this implies that the size of the Gauss sum $G_{\psi_w,Q_{w,i}}(\alpha)$ is $|\alpha|_w^{-1/2}$. Now that we know the size of the Gauss sum, we are interested in the sign. Once again by Lemma \ref{sh1} we know that the sign is $1$ is $v_w(\alpha)-m_{w,i}$ if even, where $\varpi_w^{m_{w,i}}$ is the conductor of the character $\psi_{w,Q_{w,i}}$, $\psi_{w,Q_{w,i}}$ being defined by
\[\psi_{w,Q_{w,i}}(x)=\psi_w(u\varpi_wx)\]
The conductor, $\varpi_w^{m_{w,i}}$, of $\psi_{w,Q_{w,i}}$ is easily seen to be
\[\varpi_w^{m_{w,i}}=\begin{cases}\varpi_w^{-1}&\text{ if $w$ is not associated to $T$}\\ T&\text{if $w$ is associated to $T$}\end{cases}\]
The essential point for our calculations is that $m_{w,i}$ is odd. Then this implies that $G_{\psi_w,Q_{w,i}}(\alpha)=|\alpha|_w^{-1/2}$ when $v_{w}(\alpha)$ is odd. Now suppose $v_{w}(\alpha)$ is even. Then by Lemma \ref{sh1} we have
\[G_{\psi_w,Q_{w,i}}=|\alpha|_w^{-1/2}|\varpi_w|_w^{1/2}\sum_{x\in \mathcal{O}_w/\varpi_w\mathcal{O}_w}\psi_w(u\varpi_w^{-1}\alpha_0 x^2)\]
By Lemma \ref{sh3} this is equal to $|\alpha|_w^{-1/2}(-1)^{1+s_w}(u\alpha_0,\varpi_w)_w$, where $s_w$ is the residue degree at $w$. i.e $[\mathbb{F}_w:\mathbb{F}_p]=s_w$. In summary
\[G_{\psi_w,Q_{w,1}}(\alpha)=|\alpha|_w^{-1/2}\begin{cases}1&\text{if $v_{w}(\alpha)\equiv 1\bmod2$}\\ (u\alpha,\varpi_w)_w&\text{if $v_w(\alpha)\equiv0\bmod 2$}\end{cases}\]

Now the proof goes through exactly the same as the proof as Proposition \ref{cocyclestandard}. We go through a case by case analysis of the cocycle defined by $\sigma_{w,Q_w}$, whose formulas are already given in the proof of Proposition \ref{cocyclestandard} in terms of Gauss sums. Then evaluating the Gauss sums as above and comparing with the cocycle defined in \S\ref{metgp}, which also is already explicitly written for each specific case in the proof of Proposition \ref{cocyclestandard}. Then a case by case check finishes the proof.

\end{proof}

In fact under the same assumptions it is true that the cocycle defined by \emph{any} ternary quadratic form is $\epsilon$, but since we will not be needing this we won't go through a proof.

\end{subsubsection}

\subsection{Global metaplectic group and theta functions}

For an odd prime $p\equiv1\bmod4$ let $k$ be the global field $k=\mathbb{F}_p(T)$, and let $V$ be an $n$-dimensional vector space over $k$. Let $Q$ a quadratic form over $V$ and $L$ be an $R$-lattice such that $Q$ is $T^2R$ valued on $L$ (for instance $L=\begin{setdef}{v\in V}{\psi_{\infty}(Q(v))=1}\end{setdef}$). For each place $w$ of $k$ this data defines the following local objects $L_w=L\otimes_R \mathcal{O}_w$, $V_w=V\otimes_k k_w$ and the corresponding extension $Q_w$ of $Q$. Let $\mathbb{A}_k$ be the \`adele ring of the field $k$. It is the restricted direct product $\prod_{w}^{'}k_w$  of completions $k_w$ with respect to the open compact subgroup $\mathcal{O}_w\subset k_w$. Let $\psi:\mathbb{A}_k\rightarrow \mathbb{C}^{\times}$ be the additive character defined by $\psi(\alpha)=\otimes_w^{'}\psi_w(\alpha_w)$. Note that on an algebraic curve, sum of the residues of a global meromorphic differential is $0$ so $\psi$ descends to $\psi:k\backslash\mathbb{A}_k\rightarrow\mathbb{C}^{\times}$.

At each place $w$, attached to $\psi_w$ and $Q_w$, we have the local Weil representation of $SL_2(k_w)$ defined by (\ref{weilrep}). This projective representation then defines a 2-cocycle, $\epsilon_w$, and the metaplectic double cover $\widetilde{SL}_2(k_w)$. As mentioned in \S\ref{metgp}, by Lemma 2.9 of \cite{G} this extension splits over the maximal compact $SL_2(\mathcal{O}_w)$. Then since the metaplectic group is unique, this implies that the cocycle defined by (\ref{weilrep}) also splits over $SL_2(\mathcal{O}_w)$. Lemma 2.9 of \cite{G} also gives us the splittings $\iota_w:SL_2(\mathcal{O}_w)\rightarrow \widetilde{SL}_2(k_w)$ defined by $\iota_w(g)=(g,\kappa_w(k(g)))$, and $\kappa(h)$, for $h=\left(\begin{smallmatrix}a&b\\ c&d\end{smallmatrix}\right)\in SL_2(\mathcal{O}_{w})$ is defined by
\[\kappa_w(h)=\begin{cases}(c,d)_{w}&\text{if $c\neq0$ and $c\notin \mathcal{O}_{w}^{\times}$}\\ 1&\text{otherwise}\end{cases}\]

Now using these splittings we will define the adelic metaplectic group as follows. First consider the restricted direct product $M_{\mathbb{A}_k}:=\prod_w\widetilde{SL}_2(k_w)$ where the restricted direct product is with respect to the embedded open compact subgroups $\{\iota_w(SL_2(\mathcal{O}_w))\}_w$. $M_{\mathbb{A}_k}$ is an infinite cover of $SL_2(\mathbb{A}_k)$. Consider the central subgroup $Z_{\mathbb{A}_k}\subset M_{\mathbb{A}_k}$ defined by $Z_{\mathbb{A}_k}:=\begin{setdef}{\prod_w(I_2,\delta_w)}{\delta_w=1\,\text{for a.e $w$, }\prod_w\delta_w=1}\end{setdef}$. Now define the adelic metaplectic group, $\widetilde{SL}_2(\mathbb{A}_k)$, as the quotient $M_{\mathbb{A}_k}/Z_{\mathbb{A}_k}$.\\

\begin{lemma}$\widetilde{SL}_2(\mathbb{A}_k)$ is a double cover of $SL_2(\mathbb{A}_k)$.
\end{lemma}

\begin{proof}Let $\pi:M_{\mathbb{A}_k}\rightarrow SL_2(\mathbb{A}_k)$ be the projection map. Let $g\in \ker(\pi)$, then $g=\prod_w(g_w,\delta_w)$, $\delta_w=\iota_w(g_w)$ for almost every $w$, and $\pi(g)=\prod_wg_w=1$. This means that for almost every $w$, $g_w=I_{2,w}$ and $\delta_w=\iota_w(g_w)=1$, where $I_{2,w}$ stands for the identity matrix in $SL_2(k_w)$. So there exists a finite set $S_g$ of places such that $g=\prod_{w\in S_g}(g_w,\delta_w)\prod_{w\notin S_g}(I_{2,w},1)$. Therefore when we pass to the quotient $\tilde{\pi}: M_{\mathbb{A}_k}/Z_{\mathbb{A}_k}\rightarrow SL_2(\mathbb{A}_k)$, the kernel, $\ker(\tilde{\pi})$, has order $2$, whose elements can be represented by $(I_{2,w_0},-1)\prod_{w\neq w_0}(I_{2,w},1)$ and $\prod_w(I_{2,w},1)$ for some fixed $w_0$. This proves the claim.

\end{proof}

We note that $\widetilde{SL}_2$ is \emph{not} a linear algebraic group. The following lemma allows us to  realize theta functions (and in general any metaplectic form on $SL_2(k)$) as automorphic forms on $\widetilde{SL}_2(\mathbb{A}_k)$. \\

\begin{lemma}\label{modularity} Define the lift $\eta_k:SL_2(k)\longrightarrow\widetilde{SL}_2(\mathbb{A}_k)$ by
\[\eta_k(g)=\prod_{w}(g_w,1)Z_{\mathbb{A}_k}\]
where $g_w=g$ $\forall w$. Then $\eta_k$ is a homomorphism and splits $\epsilon_Q$ over $SL_2(k)$.
\end{lemma}
\begin{proof}This follows from Proposition 5 of \cite{W} which is true in a much more general context then we have here. Since we have everything quite explicit throughout the paper we also would like to sketch a more explicit proof for the case in hand. For $g_1,g_2\in SL_2(k)$ we will show that $\epsilon_Q(g_1,g_2)=1$. As in the proof of Lemma \ref{cocyclestandard} it is enough to check this when $g_1=T_w$ or $\gamma$ (with the notation of that proof).  Then by the computations of the proof of Lemma \ref{cocyclestandard} all the cocycles are given as ratios of Gauss sums each of which was calculated in \S\ref{expcomp} as a Hilbert symbol. The claim then follows from the product formula for the Hilbert symbol.
\end{proof}

Before defining the theta functions we are interested in, we would like to extend the Weil representation as a representation of $\widetilde{SL}_2(\mathbb{A}_k)$. We will define this as the restricted tensor product $\otimes^{'}\sigma_w$. Although the definition is quite intuitive, it needs some explanation. The restricted tensor product is taken, as usual, with respect to vectors invariant under the maximal compact subgroups $SL_2(\mathcal{O}_w)$. The fact that for a given global quadratic form $Q$, almost every $\sigma_w$ has a vector invariant under $SL_2(\mathcal{O}_w)$ follows from Proposition 2.32 of \cite{G}.

We finally define theta functions via the Weil representation. As in the above paragraph we define the \emph{the global Weil representation} by $\sigma:=\otimes^{'}\sigma_w$.  Let $V_{\mathbb{A}_k}=V\otimes_k \mathbb{A}_k$ and $\mathcal{S}(V_{\mathbb{A}_k})=\otimes_w^{'}\mathcal{S}(V_w)$. Define the \emph{theta distribution} $\Theta :\mathcal{S}(V_{\mathbb{A}_k})\rightarrow\mathbb{C}$ by
\[\Theta(\phi)=\sum_{\nu\in V}\phi(\nu)\]
For any $\phi\in\mathcal{S}(V_{\mathbb{A}_k})$ we define the theta function $\Theta_{\phi}:\widetilde{SL}_2(\mathbb{A}_k)\rightarrow \mathbb{C}$ by
\[\Theta_{\phi}(g)=\Theta(\sigma(g)\phi)=\sum_{\nu\in V}\sigma(g)\phi(\nu)\]
The most important property of $\Theta_{\phi}$ is the following:\\

\begin{lemma}$\Theta_{\phi}$ is invariant by the action of $\eta_k(SL_2(k))$ on the left.
\end{lemma}

\begin{proof}It is enough to check the invariance by the action of $T_w=\left(\begin{smallmatrix}0&-1\\ 1&0\end{smallmatrix}\right)$ and upper triangular matrices since they generate the whole group. The invariance on the upper triangular matrices follows since our global character $\psi$ vanishes on $k$, and the product of the Gauss sums is $1$ as in Lemma \ref{modularity}. The invariance under $T_w$ is a consequence of the Poisson summation formula.

\end{proof}

By the above lemma, theta functions descend to well defined functions on the quotient
\[\Theta_{\phi}:\eta_k(SL_2(k))\backslash \widetilde{SL}_2(\mathbb{A}_k)\longrightarrow \mathbb{C}\]
In the next section we will choose specific test functions $\phi$ to get the classical theta functions that we are interested in this paper and deduce their transformation properties from those of the Weil representation.

\subsubsection{Classical theta functions}\label{classicaltheta}

In this section we will choose particular functions $\phi\in \mathcal{S}(V_{\mathbb{A}_k})$ and construct $\theta$-functions. Before defining the $\theta$-functions adelically let us recall embeddings of various groups into the metaplectic group locally and globally.

For each place $w$, we have (cf. Lemma 2.9 of \cite{G}) the group homomorphisms $\iota_w:SL_2(\mathcal{O}_{w})\rightarrow \widetilde{SL}_2(k_{w})$, where $\iota_w(g)=(g,\kappa_w(k(g)))$, and $\kappa(h)$, for $h=\left(\begin{smallmatrix}a&b\\ c&d\end{smallmatrix}\right)\in SL_2(\mathcal{O}_{w})$ is defined by
\[\kappa_w(h)=\begin{cases}(c,d)_{w}&\text{if $c\neq0$ and $c\notin \mathcal{O}_{w}^{\times}$}\\ 1&\text{otherwise}\end{cases}\]

Next, we have $SL_2(R)$. Again in \S\ref{metgp} we defined the homomorphism $\eta:SL_2(R)\rightarrow \widetilde{SL}_2(k_{\infty})$ by
\[\eta(g)=\left(g,\left(\frac{d}{c}\right)\right)\]
for $g=\left(\begin{smallmatrix} a&b\\ c&d\end{smallmatrix}\right)$.

We also have the diagonal embedding $\eta_k:SL_2(k)\rightarrow \widetilde{SL}_2(\mathbb{A}_k)$ defined above by
\[\eta_k(g)=(g,g,\cdots,g;1)\]
$\eta_k$ is also a homomorphism.

We also have the embeddings of $\widetilde{SL}_2(k_w)$ into $\widetilde{SL}_2(\mathbb{A}_k)$ constructed from the canonical embeddings $\widetilde{SL}_2(k_w)\hookrightarrow M_{\mathbb{A}_k}$.

We now restrict to dimension 3, and separate into two distinct cases. Let $L$ be the lattice $L=\begin{setdef}{(x,y,z)}{\in (TR)^3}\end{setdef}$, and $V=L\otimes_{TR}k\cong k^3$.
\begin{itemize}

\item \textbf{$Q_D(x,y,z)=\frac{y^2-4xz}{D}$, $D\in F_q[T]\cap \left(k_{\infty}^{\times}\right)^2$, and $D$ is square-free}.\\

In this case, the function $\phi\in S(V_\mathbb{A})$ we choose is $\phi=\otimes \phi_w$, where $\phi_w$ is defined by

$$\phi_w((x,y,z)) = \begin{cases}
\chi_{\Ooi}(x)\chi_{\Ooi}(y)\chi_{\Ooi}(z)& w=\infty\\
W_w(x,y,z)& v_w(D)>0\\
\chi_{L_w}(v)& v_w(D)=0\\
\end{cases}$$

Where $W_w(x,y,z)$ is defined as follows: It is $0$ unless $x,y,z$ and $Q_D(x,y,z)$ are all in $\Oo_w$. In the latter case,
\[W_w(x,y,z)= \begin{cases}
(x,\varpi_w)_w & \textrm{if }v_w(x)=0 \\
 (z,\varpi_w)_w & \textrm{if $v_w(x)\neq0$ and $v_w(z)=0$}\\
0 &\text{otherwise}
\end{cases}\]

We point out that this is the function that Kohnen defines in \cite{Ko}. The Kohnen function just defined has a natural interpertation. We can identify our
vector space, $V$, with the set of symmetric matrices, $\mathcal{M}$, via
\[v=(x,y,z)\in V\longleftrightarrow h_v=\left(\begin{smallmatrix} 2z & y\\y & 2x\end{smallmatrix}\right)\in \mathcal{M}\]
Furthermore on $V$ we have the quadratic form $Q_D$ defined by $Q_D(v) = -\det(h_v)/D.$ There is an action $r$ of $PGL_2(k_{\infty})$ on $\mathcal{M}$ (and hence on $V$) given by transpose conjugation; that is
\[ r(g)(h_v)=\det(g)^{-1}gh_vg^t\]
This action preserves the quadratic form and thus commutes with the Weil representation. If one considers now the
space of symmetric matrices $\mathcal{M}_{\mathbb{F}_w}$ with entries in the residue field, $\mathbb{F}_w$, of $\mathcal{O}_w$ which have rank $1$ (over $\mathbb{F}_w$), then $\mathcal{M}_{\mathbb{F}_w}$ splits into two orbits under the action of $r(PGL_2(\mathbb{F}_w)),$ and $W_w$ precisely distinguishes these two orbits. It is thus not surprising that
the action of $SL_2(\mathcal{O}_w)$ under the Weil representation on $W_w$ is scalar multiplication.
A computation shows that in fact $\sigma_{D,w}\left(\begin{smallmatrix} 0 & -1\\ 1 & 0 \end{smallmatrix}\right)(W_w(v)) =W_w(v),$ and thus for $\gamma\in SL_2(\mathcal{O}_w)$,
\begin{equation}\label{expkohnen}
\sigma_{D,w}(\gamma)(W_w(v)) = \begin{cases}
W_w(v) & c\in \mathcal{O}_w^{\times}\\
(c,d)_{w}W_w(v) & \textrm{otherwise}\\
\end{cases}
\end{equation}

Using $\phi$, we define the $D$'th $\theta$-function, $\Theta_D$, on $\widetilde{SL}_2(k_{\infty}) $by
\[\Theta_D(g)=\sum_{v\in V}\sigma_D(g)\phi(v)\]
Where $g\in \widetilde{SL}_2(k_{\infty})\hookrightarrow\widetilde{SL}_2(\mathbb{A}_k)$. The first thing to note is that $\Theta_D$ is right invariant under $\iota_{\infty}(SL_2(\mathcal{O}_{\infty}))$.\\

\begin{lemma}\label{descend}$\Theta_D$ descends to a function on $\widetilde{\mathbb{H}}$.
\end{lemma}

\begin{proof}We note the following: Since $SL_2(\mathcal{O}_{\infty})$ is generated by matrices $k=\left(\begin{smallmatrix}a&b\\ c&d\end{smallmatrix}\right)$ where $c=0$ or $c\in \mathcal{O}_{\infty}^{\times}$, and the embedding $k\mapsto \iota_{\infty}(k)=(k,\kappa_{\infty}(k))$ gives an isomorphism between $SL_2(\mathcal{O}_{\infty})$  and its image in $\widetilde{SL}_2(k_{\infty})$, it is enough to check the claim on matrices $k$ such that $c=0$ or $c\in \mathcal{O}_{\infty}^{\times}$. We start with $c=0$. Then $a\in\mathcal{O}_{\infty}^{\times}$, $\iota_{\infty}(k)=(k,1)$, and by (\ref{weilrep}) the action of $k$ on $\phi_{\infty}$ is
\[\sigma_{\infty}((k,1))\phi_{\infty}(v)=\psi_{\infty}(abB_{Q_D}(v,v))\phi_{\infty}(v)=\phi_{\infty}(v)\]
Where we used the fact that the Gauss sums $G_{\psi_{\infty},Q_D}$ and $G_{\psi_{\infty},Q_D}(a)$ are both $1$, which follows from the fact that $D\in k_{\infty}^2$, the conductor of $\psi_{\infty}$ being $\mathcal{O}_{\infty}$, and Lemma \ref{sh1}. The last equality is since $v\in (\mathcal{O}_{\infty}^3)\Rightarrow Q_D(v)\in\mathcal{O}_{\infty}$, also $ab\in\mathcal{O}_{\infty}$, and $\psi_{\infty}$ has conductor $\mathcal{O}_{\infty}$. Now consider the case $c\neq0$ and in $\mathcal{O}_{\infty}^{\times}$. Then we have
\[\sigma_{\infty}((k,1))=\int_{\mathcal{O}_{\infty}^3}\psi_{\infty}\left(\frac{aQ_D(v)-2B_{Q_D}(v,v_1)+dQ_D(v_1)}{c}\right)d\mu_{\infty}(v_1)\]
The first thing to note is that $\psi_{\infty}(Q_D(v_1))=1$ for $v_1\in\mathcal{O}_{\infty}^3$. Then we are left with
\[\psi_{\infty}(ac^{-1}Q_D(v))\int_{\mathcal{O}_{\infty}^3}\psi_{\infty}(-2c^{-1}B_{Q_D}(v_1,v))d\mu_{\infty}(v_1)\]
If $v\notin\mathcal{O}_{\infty}^3$ then the inner integral vanishes because the character is non-trivial in that case (recall that we are assuming $c\in\mathcal{O}_{\infty}^{\times}$). In the case $v\in\mathcal{O}_{\infty}^3$ we get that both the integral and $\psi_{\infty}(ac^{-1}Q_D(v))$ are $1$, and this finishes the proof.

\end{proof}

We are now interested in the automorphy properties of $\Theta_D$.\\

\begin{lemma}\label{invariance}$\Theta_D$ is left invariant under the action of $\widetilde{\Gamma}=\eta(SL_2(\mathbb{F}_p[T]))$.
\end{lemma}

\begin{proof}Let $g\in \widetilde{SL}_2(k_{\infty})$ and $\gamma=\left(\begin{smallmatrix}a&b\\ c&d\end{smallmatrix}\right)\in \widetilde{\Gamma}$. By the invariance of $\Theta_D$ on the left by $SL_2(k)$ we get
\begin{align*}
\Theta_D(\eta(\gamma)g)&=\Theta_D(\eta_k(\gamma^{-1})\eta(\gamma)g) \\
&=\sum_{v\in V}\left(\frac{d}{c}\right)\sigma_{\infty}(g)(\phi_{\infty})(v)\displaystyle\bigotimes_{w\neq\infty}\sigma_w(\gamma^{-1})(\phi_w)(v)\\
&=\sum_{v\in V}\left(\frac{d}{c}\right)\sigma_{\infty}(g)(\phi_{\infty})(v)\displaystyle\bigotimes_{\substack{w\neq\infty\\ w\mid \gcd(D,c)}}(c,d)_w\phi_w(v)\displaystyle\bigotimes_{\substack{w\nmid D\infty\\ w\mid c}}(c,d)_w\phi_w(v)\displaystyle\bigotimes_{w\nmid c\infty}\phi_w(v)\\
&=\left(\frac{d}{c}\right)\prod_{w\mid c}(c,d)_w\sum_{v\in V}\sigma(g)\phi(v)\\
&=\Theta_D(g)
\end{align*}
Where in the third line we used the following: At the places $w\mid D$ the explicit computation of the action of $\sigma_w(\gamma^{-1})$ on $W_w$ as given in (\ref{expkohnen}), and at the rest of the places the action of $\sigma_w(\gamma^{-1})$ on $\chi_{L_w}$ which follow from the same computations as in Lemma \ref{descend}. The final equality follows from Lemma \ref{hilquad}.

\end{proof}

By lemmas \ref{descend} and \ref{invariance}, $\Theta_D$ now gives a well defined function on $\widetilde{\Gamma}\backslash \widetilde{\mathbb{H}}$, which in the coordinates of \S\ref{metsp} ($w=\left(\left(\begin{smallmatrix} \sqrt{v} & u/\sqrt{v}\\ 0 & 1/\sqrt{v}\end{smallmatrix}\right),1\right)$) becomes
\begin{multline*}
\Theta_D\left(\left(w,1\right)\right) =
|v|_{\infty}^{3/4}(\sqrt{v},\varpi_{\infty})_{\infty}^{\frac{v_{\infty}(v)}{2}}\\ \displaystyle\sum_{x,y,z,\in TR^3}W_D(x,y,z)e\left(Q_D(x,y,z)\right)\chi_{\Ooi}(x\sqrt{v})\chi_{\Ooi}(y\sqrt{v})\chi_{\Ooi}(z\sqrt{v})\\
\end{multline*}

We remark that $\psi_{\infty}=e$ with the notation of \S\ref{not11}.

\item \textbf{$Q$ is equivalent to $T\alpha x^2+T  y^2+ \alpha z^2$ over $k_{\infty}$, where $\alpha\in \mathbb{F}_p\backslash \mathbb{F}_p^2$.\\}

In this case, by Lemma \ref{inverselattice}, $Q_{\infty}^{-1}(\mathcal{O}_{\infty})$ is a lattice, and we define $\phi\in S(V_{\mathbb{A}_k})$ via
$$\phi_w(v) = \begin{cases}
\chi_{L_w}(v) & w\neq\infty\\
\chi_{Q_{\infty}^{-1}(\mathcal{O}_{\infty})}(v) & w=\infty
\end{cases}$$
and the $\theta$-function is defined by the same recipe; for $g\in \widetilde{SL}_2(k_{\infty})$
\[\Theta_Q(g)=\sum_{v\in V}\sigma_Q(g)\phi(v)\]
Lemma \ref{invariance} applies to this case verbatim, and shows that $\Theta_Q$ is invariant on the left by the action of $\widetilde{\Gamma}_0(\disc(Q))$. The difference of this case from the one before is the invariance of $\phi_{\infty}$ on the right. Going through the computations in the proof of Lemma \ref{descend} shows that in this case $\Theta_Q$ is invariant only by $\iota_{\infty}(\widetilde{K}_0(\varpi_{\infty}))$, and hence gives us a well defined function on $\widetilde{\Gamma}_0(\disc(Q))\backslash \widetilde{SL}_2(k_{\infty})/\iota_{\infty}(\widetilde{K}_0(\varpi_{\infty}))$. Finally, in our coordinates on $\widetilde{\mathbb{H}}^u$, $\Theta_Q$ becomes
$$\Theta_Q(w) =|v|_{\infty}^{3/4}(\sqrt{v},\varpi_{\infty})_{\infty}^{\frac{v_{\infty}(v)}{2}} \sum_{l\in TR^3} e(Q(l)u)\chi_{\Ooi}(Q(l)v) =|v|^{3/4}_{\infty}(\sqrt{v},\varpi_{\infty})_{\infty}^{\frac{v_{\infty}(v)}{2}} \sum_{D\in TR} r_Q(D)e(Du)\chi_{\Ooi}(Dv)$$

\end{itemize}

\subsection{Siegel's Theorem}

The starting point of all the arguments in the paper is based on the observation of Siegel's that $\Theta_G-\Theta_Q$ is a cusp form, where the genus theta function $\Theta_G$ and the theta function $\Theta_Q$ are as defined in (\ref{thetaG}) and (\ref{thetaQ}). To our knowledge there is no published proof of this fact in the function field case. So in this section we state and prove Siegel's Theorem in the anisotropic case. Let $Q_1,Q_2$ be anisotropic quadratic forms on a $k$ vector space $V,$ and let $L^1$ and $L^2$
be two $R$-lattices in $V$ such that for all completions $w\neq\infty$ of $R$ the lattices $L^1_w$ and $L^2_w$ with the quadratic forms $Q_1$ and $Q_2$ are
isomorphic as abstract quadratic spaces, and at $\infty$ $Q_1$ is $k_{\infty}$-equivalent to $Q_2$. Then define $\Theta_{Q,1}(g)$ and $\Theta_{Q,2}(g)$ by
\[\Theta_{Q,i}(g)=\sum_{v\in V}\sigma_{Q_i}(g)\phi_i(v)\]
Where $\phi_i=\otimes \phi_{i,w}$ is defined by
\[\phi_{i,w}=\begin{cases}\chi_{\Ooi}(Q_i(v)) & w=\infty\\
\chi_{L^i_w}(v) & \textrm{otherwise}\end{cases}\]

\begin{thm}[Siegel's Theorem]\label{Siegel}

$F(g)=\Theta_{Q,1}(g)-\Theta_{Q,2}(g)$ is a cusp form. In particuar:
$$\int_{k\backslash{\mathbb{A}_k}}F\left(\left(\left(\begin{smallmatrix} 1 & x\\0 & 1\end{smallmatrix}\right),1\right)g\right) dx = 0\,\,\, ,\forall g\in\widetilde{SL}_2(\mathbb{A}_k)$$
\end{thm}

\begin{proof}
Let $N(x)$ denote $\left(\left(\begin{smallmatrix} 1 & x\\0 & 1\end{smallmatrix}\right),1\right).$ We now compute

$$\int_{k\backslash{\mathbb{A}_k}} \Theta_{Q,i}(N(x) g)dx = \int_{k\backslash{\mathbb{A}_k}}\sum_{v\in V}\psi(xQ_i(v))\sigma_{Q_i}(g)(\phi_i)(v) dx $$

Switching the summation and integration, we see that the integral over $x$ vanishes unless $Q_i(v)=0.$ Since $Q_i$ is anisotropic, this only happens for $v=0.$
Let $g=\otimes'g_w =\otimes' \left(\left(\begin{smallmatrix} a_w & b_w\\ c_w & d_w\end{smallmatrix}\right),\delta_w\right),$ (note that all but finitely many of the $\delta_w$ are $1$) we compute
\begin{align*}
\int_{k\backslash{\mathbb{A}_k}} \Theta_{Q,i}(N(x) g)dx&=\sigma_{Q_i}(g)(\phi_i)(0)\\
&= \displaystyle \delta_w\prod_w G_{\psi_w,Q_i}(c_w)\int_{L^i_w} e\left(\frac{d_wQ_i(\nu)}{c_w}\right)d\mu_w(\nu)\\
\end{align*}

Where we set $L^i_{\infty}=Q_{i}^{-1}(\mathcal{O}_{\infty})$. Since $(L^1_w,Q_1)$ and $(L^2_w,Q_2)$ are isomorphic as quadratic spaces at $w\neq\infty$, and $Q_1$ and $Q_2$ are equivalent over $k_{\infty}$, $$\delta_w G_{\psi_w,Q_i}(c_w) \int_{L^i_w} e\left(\frac{d_wQ_i(\nu)}{c_w}\right)d\mu_w(\nu)$$ is independent of $i.$
This completes the proof.
\end{proof}

The following corrolary is an immediate consequence of Theorem \ref{Siegel}.\\

\begin{cor}\label{SiegelsTheorem}

Let $(L_1,Q_1),(L_2,Q_2),\dots,(L_g,Q_g)$ constitute a single genus of $R$-lattices for a quadratic form $Q.$ Form the theta functions $\Theta_{Q_i}(g)$ as above,
and define $$\Theta_G(g)=\frac{1}{n_G}\sum_{i=1}^g\frac{\Theta_{Q_i}(g)}{n_{Q_i}}$$ where $n_{Q_i}:=\#SO_{Q_i}(R)$ and $n_G=\sum_{i=1}^g n^{-1}_{Q_i}$. Then for all $1\leq i\leq g$ the function $\Theta_{G}(g)-\Theta_{Q_i}(g)$ is a cusp form.
\end{cor}

\subsection{Siegel's Mass Formula}

We give a quick overview of the Siegel mass formula over function fields. This formula expresses an identity between an appropriately weighted average number of
representations of a polynomial $D$ by quadratic forms in a fixed genus $G$, and a product of local densities. More precisely, fix a quadratic
form $Q$ on an $R$-lattice $L$ of dimension $n\geq 2$ which is anisotropic over $k_{\infty}$, and let $G$ be its genus. Define $$r_G(D)=n_G^{-1}\sum_{Q_i\in G}\frac{1}{n_{Q_i}}\#\begin{setdef}{l\in L_i}{Q_i(l)=D}\end{setdef}$$ where $n_G$ and $n_{Q_i}$ are as in Corollary \ref{SiegelsTheorem}. Now for each valuation $w$ of $R$ define the local representation densities $$\alpha_{Q,w}(D)=\lim_{r\rightarrow\infty}|\varpi_w|_w^{(n-1)r}\#\begin{setdef}{l\in L_w/\varpi_w^rL_w}{Q_w(l)\equiv D\bmod \varpi_w^r}\end{setdef}.$$
Then we have the following identity:\\

\begin{thm}[Siegel's Mass Formula]\label{Siegelmassformula}
$$r_G(D)=C_Q|D|^{\frac{n}{2}-1}\prod_{w}\alpha_{Q,w}(D)$$
where $C_Q$ is a non-zero constant depending only on $Q$.
\end{thm}
\begin{proof}(Sketch)
Standard references for the mass formula over $\mathbb{Z}$ are \cite{Kn}, \cite{Ma1}, \cite{Ma2} and \cite{Si}. The adelic proof carries over to function fields which we outline now. First note that the number $r_G(D)$ is the $D$'th Fourier coeeficient of the genus theta function $\Theta_G$. Let $O_Q$ define the orthogonal group of the quadratic form $Q$, and define the theta function, $\Theta_Q(g,h)$, on $\iota_k(SL_2(k))\backslash \widetilde{SL}_2(\mathbb{A}_k)\times O(k)\backslash O(\mathbb{A}_k)$ by $\Theta_{Q,\phi^{(h)}}(z)$ where $\phi$ is the function defined by

$$\phi_w(v) = \begin{cases}
\chi_{L_w}(v) & w\neq\infty\\
\chi_{Q_{\infty}^{-1}(\mathcal{O}_{\infty})}(v) & w=\infty
\end{cases}$$
and
$$\phi^{(h)}(v):=\phi(h^{-1}(v))$$

Note that the genus of $Q$ is given by the double coset $O_Q(k)\backslash O_Q(\mathbb{A}_k)/\prod_{w}O_Q(\mathcal{O}_w)$, and $\Theta_Q(g,h)$ interoplates the $\Theta_{Q_i}(g)$. Moreover, by breaking up over the genus one sees that $$\Theta_G(g)=\int_{O_Q(k)\backslash O_Q(\mathbb{A}_k)}\Theta_Q(g,h)dh.$$
Define $N(x)=\left(\left(\begin{smallmatrix} 1 & x\\0 & 1\end{smallmatrix}\right),1\right)$. We now take $D$'th Fourier coefficients:
\begin{align*}
\int_{k\backslash{\mathbb{A}_k}} \Theta_{G}(N(x)g)\psi(-Dx)dx&=\int_{k\backslash{\mathbb{A}_k}\times O_Q(k)\backslash O_Q(\mathbb{A}_k)}\Theta_Q(N(x)g,h)dxdh\\
&= \int_{k\backslash{\mathbb{A}_k}\times O_Q(k)\backslash O_Q(\mathbb{A}_k)}\sum_{v\in L}\psi((D-Q(v))x)\phi(h^{-1}v)dxdh\\
&=\int_{O_Q(k)\backslash O_Q(\mathbb{A}_k)}\sum_{\substack{v\in L \\ Q(v)=D}}\phi(h^{-1}v)dh\\
\end{align*}
Pick $v_D\in L$ such that $Q(v_D)=D$ and let $H_{v_D}\in O_Q$ be its stabilizer.  Sice $O_Q(k)$ acts transitively on elements $v\in L$ with $Q(v)=D$, we have
\begin{align*}
\int_{k\backslash{\mathbb{A}_k}} \Theta_{G}(N(x)g)\psi(-Dx)dx&=\int_{O_Q(k)\backslash O_Q(\mathbb{A}_k)}\sum_{\substack{v\in L\\ Q(v)=D}}\phi(h^{-1}v)dh\\
&=\int_{H_{v_D}(k)\backslash O_Q(\mathbb{A}_k)}\phi(h^{-1}v_D)dh\\
&=\mu(H_{v_D}(k)\backslash H_{v_D}(\mathbb{A}_k))\int_{H_{v_D}(\mathbb{A}_k)\backslash O_Q(\mathbb{A}_k)}\phi(h^{-1}v_D)dh\\
&=\mu(H_{v_D}(k)\backslash H_{v_D}(\mathbb{A}_k))\prod_{w}\int_{H_{v_D}(k_w)\backslash O_Q(k_w)}\phi_w(h_w^{-1}v_{D,w})dh_w\\
\end{align*}

Evaluating the local integrals gives the result.
\end{proof}

We now restrict to the case where $Q$ is a ternary quadratic form. Let $w$ be a place such that $v_w(\disc(Q))=0.$ Then the form $Q_w$ is equivalent over $\mathcal{O}_w$ to $\disc(Q)(x^2+y^2+z^2).$ We then compute the local densities at $w$:
$$\alpha_{Q,w}(D)=\begin{cases}
1+|\varpi_w|_w & v_w(D)=0,(D\disc(Q),\varpi_w)_w = 1\\
1-|\varpi_w|_w & v_w(D)=0,(D\disc(Q),\varpi_w)_w = -1\\
1-|\varpi_w|_w^2 & v_w(D)=1\\
\end{cases}$$

Combining the above with Theorem \ref{Siegelmassformula} we see that $r_G(D)$ grows like $L(1,\chi_{D\disc(Q)})$. Using the lower bound from Lemma
\ref{classnumber}, we finally get the following:\\

\begin{cor}\label{largegenusrepresentations}
$$r_G(D)\gg_Q \frac{|D|^{1/2}}{\log_p\log_p|D|}$$ where the implied constant on $Q$ is effective.
\end{cor}

\end{document}